\documentclass[reqno, 10pt]{amsart}
\RequirePackage[colorlinks,linkcolor=blue,citecolor=blue,urlcolor=blue]{hyperref}
\usepackage[top=3.8cm, bottom=2.3cm, left=3cm, right=3cm]{geometry}

\usepackage{hyperref}
\usepackage[centertags]{amsmath}
\usepackage{amsmath}
\usepackage[dutch,british]{babel}
\usepackage{amsfonts}
\usepackage{amssymb}
\usepackage{amsthm}
\usepackage{multirow}
\usepackage{newlfont}
\usepackage{color}
\usepackage{lipsum}

\usepackage{authblk}

\allowdisplaybreaks

\usepackage{stackengine}
\stackMath

\newcommand\ttilde[2][1]{%
 \def\useanchorwidth{T}%
  \ifnum#1>1%
    \stackon[0pt]{\ttilde[\numexpr#1-1\relax]{#2}}{\hspace{0.2em}\scriptscriptstyle\thicksim}%
  \else%
    \stackon[1pt]{#2}{\hspace{0.2em}\scriptscriptstyle\thicksim}%
  \fi%
}

\usepackage{bbm}
\include{amsthm_sc}

\usepackage{stmaryrd}
	 \SetSymbolFont{stmry}{bold}{U}{stmry}{m}{n}
\usepackage{mathrsfs}

\usepackage{fancyhdr}
\usepackage{fancybox}
\usepackage{setspace}
\usepackage{cleveref}

\usepackage{enumitem}

\theoremstyle{plain} 
\newtheorem{theorem}{Theorem}[section]
\newtheorem*{theorem*}{Theorem}
\newtheorem{lemma}[theorem]{Lemma}
\newtheorem{corollary}[theorem]{Corollary}
\newtheorem*{corollary*}{Corollary}
\newtheorem{proposition}[theorem]{Proposition}
\newtheorem*{proposition*}{Proposition}
\newtheorem{definition}[theorem]{Definition}
\newtheorem*{definition*}{Definition}
\newtheorem{assumption}[theorem]{Assumption}

\theoremstyle{definition} 

\newtheorem*{example*}{Example}
\newtheorem{remark}[theorem]{Remark}

\newtheorem*{remark*}{Remark}
\newtheorem*{remarks*}{Remarks}

\newcommand{\deq}{\mathrel{\mathop:}=}

\newcommand{\N} {\mathbb{N}}

\newcommand{\E} {\mathbb{E}}

\newcommand{\lefnu}{\widetilde{\nu}_N^{\lambda}}

\DeclareMathOperator{\diag}{diag}

\DeclareMathOperator{\Tr}{Tr}

\DeclareMathOperator{\supp}{supp}

\newcommand{\caC}{{\mathcal C}}
\newcommand{\caD}{{\mathcal D}}
\newcommand{\caE}{{\mathcal E}}

\newcommand{\caL}{{\mathcal L}}

\newcommand{\caP}{{\mathcal P}}

\newcommand{\caR}{{\mathcal R}}

\newcommand{\caU}{{\mathcal U}}

\newcommand{\caX}{{\mathcal X}}

\newcommand{\caZ}{{\mathcal Z}}

\newcommand{\bbC}{{\mathbb C}}

\newcommand{\bbN}{{\mathbb N}}

\newcommand{\bbP}{{\mathbb P}}

\newcommand{\bbR}{{\mathbb R}}

\newcommand{\bbT}{{\mathbb T}}

\newcommand{\bbZ}{{\mathbb Z}}

\newcommand{\opunit}{\text{1}\kern-0.22em\text{l}}

\newcommand{\fra}{{\mathfrak a}}
\newcommand{\frb}{{\mathfrak b}}

\newcommand{\frf}{{\mathfrak f}}
\newcommand{\frg}{{\mathfrak g}}
\newcommand{\frh}{{\mathfrak h}}

\newcommand{\frn}{{\mathfrak n}}

\newcommand{\frt}{{\mathfrak t}}

\newcommand{\scrE}{{\mathscr E}}

\newcommand{\pd}{\partial}

\newcommand{\pair}[1]{\langle{#1}\rangle}

\newcommand{\wt}{\widetilde}
\newcommand{\ovl}{\overline}
\newcommand{\wh}{\widehat}
\newcommand{\bv}{\breve}

\newcommand{\mfc}{m_\mathrm{fc}}
\newcommand{\fc}{\mathrm{fc}}

\newcommand{\wg}{\widetilde{\gamma}}

\newcommand{\opr}{O_{\prec}}

\renewcommand{\d}{{\mathrm d}}

\newcommand{\bee}{ \begin{equation} }
\newcommand{\eeq	}{ \end{equation} }

\newcommand{\bigO}[1]{{O}{\left( #1 \right)}}

\newcommand{\norm}{ \|}

\newcommand{\vvv}{[\mathrm{v}]}

\newcommand{\ben}{\begin{arabicenumerate}}
\newcommand{\een}{\end{arabicenumerate}}

\newcommand{\ii}{\mathrm{i}}

\newcommand{\G} {\mathcal{G}}
\newcommand{\MM}{\mathsf{M}}

\newcommand{\FF}{\mathsf{F}}

\newcommand{\RRR}[1]{\wt{\mathsf{R}}^{({#1})}}

\newcommand{\sfP}{\mathsf{P}}

\renewcommand{\P}{\mathbb{P}}

\newcommand{\vk}{\varkappa}
\newcommand{\vt}{\vartheta}
\newcommand{\vs}{\varsigma}

\renewcommand{\Re}{\mathrm{Re\,}}
\renewcommand{\Im}{\mathrm{Im\,}}

\newcommand{\Psii}{\Psi_2}

\newcommand{\mbf}[1]{\mathbf{#1}}

\newcommand{\mc}[1]{\mathcal{#1}}

\newcommand{\Tvpi}{\Theta_{\varpi}}

\newcommand{\revision}[1]{\textcolor{black}{#1}}

\numberwithin{equation}{section} 
\numberwithin{theorem}{section}

\newcommand\blfootnote[1]{%
  \begingroup
  \renewcommand\thefootnote{}\footnote{#1}%
  \addtocounter{footnote}{-1}%
  \endgroup
}

\begin{document}

\renewcommand{\thefootnote}{\fnsymbol{footnote}}	

\begin{minipage}{0.85\textwidth}
 \vspace{2.2cm}
    \end{minipage}
\begin{center}
\Large\bf
Local laws and spectral properties of deformed sparse random matrices \blfootnote{
{{\phantom{\textsuperscript{$\dagger$}}\textit{Keywords}:  Local law, sparse random matrices, deformed semicircle law}\\
{\phantom{\textsuperscript{$\dagger$}}\textit{AMS Subject Classification (2020)}: 15B52; 60B20}\\
{\phantom{\textsuperscript{$\dagger$}}\textit{Date}: \today}\\
}}
\end{center}

\vspace{1.4cm}
\begin{center}
 \begin{minipage}{0.45\textwidth}
\begin{center}
Ji Oon Lee  \\
\footnotesize { KAIST }\\
{\it jioon.lee@kaist.edu}
\end{center}
\end{minipage}
\begin{minipage}{0.45\textwidth}
 \begin{center}
Inyoung Yeo\\
\footnotesize 
{Columbia University}\\
{\it iy2225@columbia.edu}
\end{center}
\end{minipage}

\end{center}

\vspace{1.1cm}

\begin{center}
 \begin{minipage}{0.9\textwidth}
\small
\hspace{10pt}

We consider deformed sparse random matrices of the form $H= W+ \lambda V$, where $W$ is a real symmetric sparse random matrix, $V$ is a random or deterministic, real, diagonal matrix whose entries are independent of $W$, and $\lambda = O(1) $ is a coupling constant. Under mild assumptions on the matrix entries of $W$ and $V$, we prove local laws for $H$ that compare the empirical spectral measure of it with a refined version of the deformed semicircle law. By applying the local laws, we also prove several spectral properties of $H$, including the rigidity of the eigenvalues and the asymptotic normality of the extremal eigenvalues.
 \end{minipage}
\end{center}
 \date{\today}
 \vspace{7mm}

\thispagestyle{headings}

\section{Introduction}

In this paper, we consider the class of \emph{deformed sparse random matrices}, the sum of two large $N \times N$ matrices
\begin{align}\label{eq:1}
	H = W + \lambda V, 
\end{align}
where $W$ is an $N \times N$ symmetric sparse random matrix, $V$ is a random or deterministic diagonal matrix, and $\lambda = O(1)$ is a coupling constant. The matrices are normalized so that the eigenvalues of $V$ and $W$ are of order one. If the entries, $(V_i)$, of $V$ are random, we may think of $V$ as a ``random potential", and if the entries of $V$ are deterministic, $V$ can be considered as an ``external source". We assume that the empirical eigenvalue distribution of $V$ converges weakly (in probability) to a nonrandom measure.

In the limit $\lambda \lesssim N^{-1/2}$, $H$ belongs to the class of \emph{sparse random matrices}. The spectral properties of sparse random matrices have attracted significant interest as a natural generalization of the (centered) adjacency matrix of Erdős-Rényi graphs. A closely related perspective comes from \textit{diluted Wigner matrices}, a symmetric matrix defined as follows: for $i \leq j$, let $D_{ij} =  B_{ij}  M_{ij}$, where $M_{ij}$ are i.i.d. centered random variables with unit variance and all the moments finite, and $B_{ij}$ are Bernoulli-type random variables, independent of $M$, satisfying 
\begin{align}\P (B_{ij} = (Np)^{-1/2} ) = p \quad \text{and} \quad \P (B_{ij} = 0) =1-p.
\end{align} For $p \ll 1$, most of $D_{ij}$ vanish, and the model can be interpreted to characterize sparse interaction among the lattice sites. It is often convenient to introduce the \textit{sparsity parameter} $q \deq \sqrt{Np}$. (See Definition \ref{def:sparse} for the precise definition of the sparsity parameter.)
In terms of global statistics, the limit of the empirical spectral distribution of sparse random matrices is given by Wigner's semicircle law for $q^2 = pN \gg 1$. Furthermore, when $q^2 = pN > b_* \log N $ (for $b_* \approx 2.59$), the spectral norm converges to $2$, and all eigenvectors are delocalized \cite{ADK1, ADK2}.

When the sparsity parameter satisfies $q \gtrsim N^{1/2}$, $W$ recovers the class of Wigner matrices, and $H$ belongs to the class of \emph{deformed Wigner matrices}. The global properties of deformed Wigner matrices are understood through the so-called deformed semicircle law, proved by Pastur in \cite{Pas}. Free probability provides a natural foundation for the sum of random matrices \cite{V} and, in particular, the deformed semicircle law. Free convolution measure is defined as the distribution of the sum of the distributions of two freely independent non-commutative random variables. It is convenient to comprehend the free convolution of probability measures as the unique solution of a system of equations expressed in terms of their Stieltjes transforms \cite{BB}. (See Proposition \ref{thm:fc}.) The deformed semicircle law turns out to be the scaled free convolution between the semicircle law and law of $V$. In this context, it is natural to conjecture that the deformed semicircle law approximates the spectral properties of deformed sparse random matrices. To the best of our knowledge, the only known rigorous result concerns a special case involving Erdős-Rényi graph Laplacians, where the empirical eigenvalue distribution converges weakly to the free convolution of the semicircle law and standard Gaussian \cite{HL}.

Turning to the local eigenvalue statistics, the first step is proving a local law comparing the resolvent (Green function) of a random matrix and the Stieltjes transform of the equilibrium measure. While the local laws depend on the model, versatile methods have been developed, especially for mean-field models \cite{BEKYY,EY,KY,HKR}. These local laws play a crucial role in controlling the strength of eigenvalue estimates. Notably, the recovery of edge universality for both deformed Wigner matrices and sparse random matrices is achieved through a common idea: extremal eigenvalues were compared against an $N$-dependent (and possibly random) endpoint of an $N$-dependent equilibrium measure, instead of its limit at $N \to \infty$. Accordingly, it is necessary to establish a local law comparing the resolvent with respect to a possibly $N$-dependent (and possibly random) equilibrium measure. 

Our first result is a local law for sparse deformed matrices up to the spectral edges. More precisely, we prove an averaged local law (Theorem \ref{thm:refine}) and an entrywise local law (Theorem \ref{thm:strong}), which compares the empirical spectral measure not with the deformed semicircle law but with a refined version of it. The refined deformed semicircle law is defined as the free convolution of the semicircle law with a shifted version of the empirical spectral distribution of $(V_i)$, where the shift size depends on the sparsity. In the large $N$ limit, the refined measure converges to the limit of the deformed semicircle law. An important feature of our construction is that, under a mild condition ensuring square-root decay of the standard deformed semicircle law, the refinement preserves this behavior. Moreover, expansion shows that when $\lambda \lesssim {N}^{-1/2}$, the deterministic shift coincides asymptotically with the correction introduced in \cite{LS18}. See Assumption \ref{assm:stableint} and Proposition \ref{prop:refv} for details.

For deformed Wigner matrices, the shifts of the eigenvalues originate from $\lambda V$, independent of the Wigner matrix. Thus, by working on a probability space that fixes the configuration of $V$, it is possible to decouple the shifts and fluctuations induced from $V$ with the fluctuations of the Wigner component \cite{LS14}. However, the task is more complicated for sparse random matrices. The shifts in the spectral edges arise from the slowly decaying moments of the sparse random matrix, and so the local law depends on the sparsity \cite{EKYY1}. The edge universality for sparse random matrices was also established in \cite{EKYY2} for $q \gg N^{1/3}$, and it was extended to $q\geq N^{\delta}$ for $\delta >1/6$ in \cite{LS18} by introducing a deterministic correction of the semicircle law. The main technique in \cite{LS18} is the so-called recursive moment estimate relying on expansion through integration by parts, or cumulant expansion. 

While we employ the recursive moment estimates to obtain the refined local law, the anisotropic nature of the resolvent entries in the deformed model prevents a direct application of the previous techniques developed for sparse random matrices. In particular, it is not possible to replace $\frac{1}{N}\sum_{i} G_{ii}^n$ by $(\frac{1}{N} \Tr G)^n$ (where $G$ denotes the resolvent) up to negligible error, a key idea used in the estimates for sparse random matrices. Furthermore, the function whose high moments must be controlled, denoted $\caP(G)$, is a specified function depending on $V$ and the entries of $G$, not just the trace of $G$. As a result, the refined deformed semicircle law differs significantly from the refined measures used in sparse random matrices. The distinction also stems from the functional equation-driven nature of the deformed semicircle law.

We apply the refined local law to derive several results on the eigenvalue locations of deformed sparse random matrices, following the methods developed in \cite{EY, EKYY1, LS}. We establish bounds on the extremal eigenvalue and the density of states, and also the eigenvalue rigidity (Theorem \ref{thm:rigidity}), where these results are optimal given $q \gtrsim N^{2/9}$. Additionally, for random $V$, we prove that when $\lambda \gg \max \{q^{-1} , N^{-1/6}, N^{1/2}q^{-3}\}$, the leading fluctuation of the extremal eigenvalues is asymptotically Gaussian, dictated by the randomness of $V$ (Theorem \ref{thm:gauss}). We would emphasize that, without the refinements in this work, estimates at the spectral edge would remain limited to the regime $q \gg N^{1/3}$.

The bulk of the paper is devoted to constructing higher-order self-consistent equations for $|\caP(G)|$. Inspired by the approach for deformed Wigner matrices in \cite{HKR}, we identify an auxiliary quantity $Q$ arising from the derivative of $\caP(G)$. By applying a second cumulant expansion iteratively on $Q$, we obtain a bound that appears to be close to optimal. See Section \ref{sec:Q} for details. The bound is incorporated into the recursive moment estimate of $\caP(G)$, leading to an improved self-consistent equation. The refined local law can then be obtained through a standard continuity argument. However, in the vicinity of $\lambda \lesssim N^{-1/2}$, the resulting bound is weaker than the corresponding result for sparse random matrices in \cite{LS18}. We believe that this discrepancy can be resolved by refining the analysis and discuss the subtle terms responsible for the weaker estimate in Section \ref{sec:I21}.

\subsection{Related works}

The deformed semicircle law was first proved by Pastur \cite{Pas} and its properties were studied in \cite{B, Sh11, ShTi, LS, BES20a}, including sufficient conditions ensuring that the deformed semicircle law is supported on a single interval with square root type decay at the spectral edges. (The assumption is often referred to as stability condition, and these edge behavior is referred to as regularity.) The local law for the deformed Wigner matrices was proved in \cite{LS}, and the edge universality in \cite{LS14}, where the fluctuation of the eigenvalue, relative to the deterministic endpoint obtained by averaging over the randomness of $V$, asymptotically follows a distribution given by the (classical) convolution of a Gaussian of size $N^{-1/2} \lambda$, and the Tracy-Widom distribution, of scale $N^{-2/3}$ \cite{LS14}. Other types of universality including eigenvectors \cite{LS, LS16, Be}, bulk universality \cite{LSSY, OV}, and CLT for linear spectral statistics \cite{JL, LSX, LL} have also been studied. These ideas were extended further to study the free sum of random matrices, of the form $A + UBU^*$, where $A, B$ are deterministic matrices, and $U$ is Haar distributed on the unitary group. For more details, we refer to \cite{BES15, BES20b, JP} and the references therein. In physics literature, the deformed Wigner matrices are known as the Rosenzweig-Porter model. The model serves as an example describing interpolation between integrability and chaos \cite{vSW_rp}.

The local law for sparse random matrices and the edge universality for $q \gg N^{1/3}$ were proved in \cite{EKYY1,EKYY2}. The edge universality was extended to $q\geq N^{\delta}$ for $\delta >1/6$ in \cite{LS18}. Later, it was discovered in \cite{HLY20,HK} that the extremal eigenvalues behave asymptotically as the (classical) convolution of a Gaussian distribution of size $N^{-1/2} q^{-1}$ and Tracy-Widom distribution of scale $N^{-2/3}$. The Gaussian fluctuation is captured by a so called ``random fluctuation term", encoded in the random equilibrium measure. Recently, edge universality was established in the works \cite{Le, HY}, by subtracting all higher order fluctuations. The bulk universality \cite{HLY15} and CLT for linear spectral statistics \cite{ShTi, He} were established up to $q \geq N^{\delta}$ for $\delta>0$.

The cumulant expansion method has been instrumental in analyzing a wide variety of random matrix ensembles. It has enabled the proof of local laws for deformed Wigner matrices without the diagonal assumption on $V$ \cite{HKR, CEHK}, as well as for models such as the random block matrix model, introduced in connection with the quantum chaos transition conjecture \cite{SYY} and viewed as a deformed generalized Wigner matrix. Moreover, local laws of isotropic type were also established for sparse random matrices \cite{HHW} and sparse stochastic block models \cite{HLY25}, using these ideas. The methodology was further developed to treat more structured ensembles such as non-Hermitian models \cite{He1, CCEJ}, correlated matrices \cite{AEKS, EKS}, and random regular graphs \cite{BHKY, HY24, He2}.

\subsection{Organization of the paper}
The rest of the paper is organized as follows: In Sect. \ref{sec:def}, we define the model, present the main results, with emphasis on determining the refined deformed semicircle law up to the edges. In Sect. \ref{sec:prelim}, we collect and extend several known results, and explain the strategy of the proof. In Sect. \ref{sec:primary}, we obtain the weak local law, which serves as an a priori estimate. We also define and prove estimates on an auxiliary quantity $Q$, a central object in the recursive moment estimates. In Sect. \ref{sec:rme}, we present the recursive moment estimate, which is the main technical ingredient in proving the local laws. In Sect. \ref{sec:locallaw}, we prove the strong entrywise local law, and the refined averaged local law. In Sect. \ref{sec:behavior}, we apply the local law to prove the results on the eigenvalues. Several technical parts of the proofs can be found in the Appendix. 

\subsection{Notations}

\begin{remark}[Notational conventions] We use the symbols $O(\cdot)$ and $o(\cdot)$ for the standard big-O and little-o notation. The notation $O, o, \ll, \gg$, refer to the limit $N \to \infty$ unless otherwise stated. Here $a \ll b$ means $a = o(b)$. We use $c$ and $C$ to denote positive constants that do not depend on $N$, usually with the convention $c \leq C$. Their value may change from line to line. We write $a \sim b$, if there is $C \geq 1$ such that $C^{-1} |b| \leq |a| \leq C|b|$. Throughout the paper we denote for $z \in \bbC^+$ the real part by $E = \Re z$ and the imaginary part by $\eta = \Im z$. We use double brackets to denote index sets, i.e. $\llbracket n_1 , n_2 \rrbracket \deq [n_1 ,n_2] \cap \bbZ$ for $n_1, n_2 \in \bbR$.
	
	For matrices $A, B\in \bbC^{N \times N}$, we write $A \odot B$ for the entrywise (Hadamard) product of the matrices $A, B$. The usual operator norm (of $A$) induced by the standard Euclidean vector norm $\norm \cdot \norm$ (on $\bbC^N)$ will be denoted by $\norm A \norm$. We also write $\pair{A} \deq N^{-1} \Tr A$.
\end{remark}

\begin{definition}[High probability event] We say that an $N$-dependent event $\Omega \equiv \Omega^{(N)}$ holds with high probability if, for any large $D > 0$,
	\begin{align*}
		\bbP (\Omega^{(N)}) \geq 1 - N^{-D},
	\end{align*}
	for sufficiently large $N \geq N_0 (D)$.
\end{definition}

\begin{definition}[Stochastic domination]
	Let
	\begin{equation*}
		X = \left( X^{(N)}(u) \;:\; N \in \N, u \in U^{(N)} \right) \,, \qquad
		Y = \left( Y^{(N)}(u) \;:\; N \in \N, u \in U^{(N)} \right),
	\end{equation*}
	be two families of random variables, where $U^{(N)}$ is a possibly $N$-dependent parameter set. 
	We say that $X$ is \emph{stochastically dominated by $Y$} (uniformly in $u$) if for all (small) $\epsilon > 0$ and (large) $D > 0$
	\begin{equation*}
		\sup_{u \in U^{(N)}} \P \left({|X ^{(N)}(u)| > N^\epsilon Y^{(N)}(u)}\right) \;\leq\; N^{-D},
	\end{equation*}
	for any sufficiently large $N\geq N_0(\epsilon, D)$. 
	
	We write $X \prec Y$ or $X = O_{\prec}(Y)$, if $X$ is stochastically dominated by $Y$. We further extend the definition of $\opr(\cdot)$ to matrices. Let $X$ be a family of complex $N \times N$ random matrices, and $Y$ a family of nonnegative random variables. Then we write $X= O_{\prec}(Y)$ to mean $\max_{ij} |X_{ij}|\prec Y$. 
\end{definition}

\section{Definition and main results}\label{sec:def}
\subsection{Definition of the model}\label{sec:defandmod} In this section, we list the definitions and assumptions of our model.
\begin{definition}[Sparse random matrices] \label{def:sparse}
	Fix any small $\delta > 0$. We assume that $W = (W_{ij})$ is a real symmetric $N \times N$ matrix whose entries are independent, up to the symmetry constraint $W_{ij} = W_{ji}$, centered random variables. We further assume that $(W_{ij})$ satisfy the moment conditions
	\begin{align}\label{eq:momentcondition}
		\E W_{ij} = 0, \quad \E (W_{ij})^2 = \frac{1 + \bigO{\delta_{ij}}}{N}, \quad \E |W_{ij}|^{k} \leq \frac{(Ck)^{ck}}{Nq^{k-2}}, \quad (k \geq 3),
	\end{align}
	with sparsity parameter $q$ satisfying
	\begin{align}
		N^{\delta} \leq q \leq N^{1/2}.
	\end{align}
	We denote by $\mc{C}_k$ the $k$-th cumulant of the i.i.d random variables $(W_{ij}: i<j)$, where we have $\mc{C}_1 =0$, $\mc{C}_2 = {1}/{N}$, and 
	\begin{align}
		|\mc{C}_{k}| \leq \frac{(2 Ck)^{2(c+1)k}}{Nq^{k-2}}, \quad (k \geq 3).
	\end{align} We further introduce the normalized cumulants, $s^{(k)}$, by $s^{(1)}\deq 0$, $s^{(2)} \deq 1$, and $s^{(k)} \deq Nq^{k-2} \mc{C}_k$ for $k \geq 3$, where $s^{(k)}$ are all of order $O(1)$. We further assume that
	\begin{align}
		|s^{(k)}| \leq C_k, \quad s^{(4)} \geq c_4,
	\end{align}
	for some $C_k, c_4 >0$. 
\end{definition}
\begin{remark}
	The fourth cumulant is especially important in our analysis, and so we will denote $s^{(4)} = s$ hereafter. The lower bound $s \geq c_4$ ensures that the scaling by $q$ for the ensemble $W$ is ``correct''. (See also \cite{LS18, HLY20, HY}.)
\end{remark}
For the sake of convenience, we will fix the sparsity parameter as $q = N^{\phi}$ for some fixed $\phi \in [\delta, 1/2]$. By doing so, the order of $q$ and $\lambda$ can be compared conveniently, while it is believed that the assumptions can be removed with some care. For more details, refer to Proposition \ref{prop:refv} and Remark \ref{rem:classify}.

We also remove the diagonal of $W$ by replacing $W_{ij}$ by $W_{ij} - \delta_{ij} W_{ii}$, for simplicity. The shift in extremal eigenvalues of $H$ due to the replacement is negligible. On the level of the local laws, insensitivity of the  diagonals can be proved analogously as in Appendix B of \cite{CES}. Suppose we have a local law for $H = W+\lambda V$ (whose resolvent is $G$) , and consider a random matrix $H^{\circ}= W^{\circ} + \lambda V = W-A +\lambda V$ (whose resolvent is $G^{\circ}$), where $A$ is an $N \times N$ diagonal random matrix whose entries are independent, and also independent of $W$, satisfying the moment conditions in \eqref{eq:momentcondition}.  Expanding $G^{\circ}$ up to $n \geq \lfloor 100(\tau^{-1} + 1) \rfloor$ times yield $G^{\circ} -G =  GAG + GAGAG + \cdots + (GA)^{n} G^{\circ}$. Each term in the right hand side can be bounded following the method of the proof of Proposition 3.10 and 3.11 in \cite{HL}, which generalizes the large deviation estimates for sparse random variables obtained in \cite{EKYY1}. Although \cite{HL} studies a different model, we remark the main inputs of Proposition 3.10 and 3.11 in \cite{HL} coincide with our setting; entries of $A$ are independent of $G$, and the use of Ward identities are similar. It should be noted that, even without the removal of the diagonal of $W$, it is possible to prove the same estimates in our paper, and the removal is only for the sake of convenience. For the proofs of insensitivity of the diagonal of $W$ on the level of the fluctuations of the extremal eigenvalues, we refer to \cite[Lemma 4.1]{LS14}.

Next, we turn to the potential $V$. For $V$, we consider random or deterministic real diagonal matrix, with each entries compactly supported with order $O(1)$.

\begin{assumption}\label{assm:esd}
	Let $V = V^{(N)} = \diag(V_1^{(N)}, \dotsc, V_N^{(N)})$ be an $N \times N$ real diagonal, random or deterministic matrix, with empirical spectral distribution (ESD) $\wh{\nu}$, i.e.
	\begin{align}
		\wh{\nu} \equiv \wh{\nu}_N = \frac{1}{N} \sum_{i=1}^{N} \delta_{V_i^{(N)}}
	\end{align}
	Furthermore, we assume that the ESD of $V^{(N)}$ converges weakly (weakly in probability if $V^{(N)}$ is random) to a deterministic distribution $\nu$. We also assume $\nu$ to be a deterministic, centered, compactly supported probability measure on $\mathbb{R}$.
\end{assumption}

As in the preceding literature (see \cite{LSSY, LSX, CES}) we assume $\norm  V \norm $ to be bounded uniformly in $N$. Since $\wh{\nu}_N$ and $\nu$ are compactly supported, we define
\begin{align}
	\wh{r}_N^{-} \deq \inf \supp \wh{\nu}_N , \quad \wh{r}_N^{+} = \sup \supp \wh{\nu}_N, \quad r^- = \inf \supp \nu,\quad r^+ = \sup \supp \nu.
\end{align}

The following assumption formalizes the notion of weak convergence of $\wh{\nu}$ to $\nu$ in the limit $N \to  \infty$. 

\begin{assumption}\label{assm:V}
	We assume that the entries of $V^{(N)}$ satisfy the following:
	\begin{enumerate}
		\item If $V^{(N)}$ is a \textbf{random} matrix, let $\{ V_i^{(N)}: 1 \leq i \leq N \}$ be a collection of i.i.d. random variables with law $\nu$, independent of $W$. 
		\item If $V^{(N)}$ is a \textbf{deterministic} matrix, we assume
		\begin{align}
			\sup_{z \in \caD}\left| \int \frac{1}{x-z} \d \wh{\nu}_N(x) - \int \frac{1}{x-z} \d \nu(x) \right|  = O(N^{-\fra_0}),\label{eq:alpha0}
		\end{align}
		for any fixed compact set $\caD \subset \mathbb{C}^+$ with $\mathcal{D} \cap \supp \nu = \emptyset$, for some $\fra_0 >0$.
	\end{enumerate}
\end{assumption}

We define deformed sparse random matrix of size $N$ as an $N \times N$ symmetric random matrix $H$ that can be decomposed into
\begin{align}
	H = (H_{ij}) \deq W+ \lambda V.
\end{align}
Here, $W$ is a real symmetric sparse random matrix of size $N$ and $V$ is a random or deterministic real diagonal matrix, introduced in Assumption \ref{assm:V}. The coupling constant $\{\lambda \equiv \lambda_N\}$ is a finite, non-negative constant that converges sufficiently fast in the limit of large $N$. For convenience, we denote the distribution $\nu$ and the empirical spectral distribution $\wh{\nu}_N$ scaled by the factor $\lambda $ as $\nu^{\lambda}$ and $\wh{\nu}^{\lambda}_N$, respectively. Similarly, let $m_{\nu}^{\lambda}$ and $m_{\wh{\nu}}^{\lambda}$ be the Stieltjes transforms of $\nu^{\lambda}$ and $\wh{\nu}_N^{\lambda}$, respectively.

In this study, we focus on the regime where square-root decay behavior is exhibited at the spectral edges of $H$. As studied in the previous works on deformed Wigner matrices, we impose several assumptions on $\lambda$ and $V$, often referred to as stability. Given that the eigenvalue distribution is expected to converge weakly to the deterministic deformed semicircle law as $N \to \infty$, it is natural to impose the same condition assumed for deformed Wigner matrices. The following assumptions, previously introduced in \cite{LS14, LSSY, JL}, serve as a sufficient condition for stability, also ruling out the possibility that the matrix $H$ has ``outliers'' in the limit of large $N$.

\begin{assumption}\label{assm:stableint} Let $I_{\nu}$ be the smallest interval such that $\mathrm{supp}\, \nu \subseteq I_{\nu}$. Then, there exist $\varpi > 0$, independent of $N$, such that
	\begin{align}\label{eq:s1}
		\inf_{x \in  I_{\nu}} \int \frac{\d \nu (v)}{ (v - x)^2} \geq (1+ \varpi)\lambda^2 ,
	\end{align}
	for $N$ sufficiently large. Moreover, let $I_{\wh{\nu}_N}$ be the smallest interval such that $\supp \wh{\nu} \subseteq I_{\wh{\nu}_N}$. Then, we assume the following condition to $\wh{\nu}_N$:
	\begin{enumerate}
		\item For \textbf{random} $\{V_i \}$, we assume that there exist a constant $\mathfrak{t}$, such that
		\begin{align}\label{eq:s2}
			\bbP \left[ \inf_{x \in  I_{\wh{\nu}_N} } \int \frac{1}{ ( v - x)^2 } \d \wh{\nu}_N (v) \geq (1 + \varpi)\lambda^2 \right] \geq 1 - N^{- \mathfrak{t}} ,
		\end{align}
		for $N$ sufficiently large.
		\item For \textbf{deterministic} $(V_i)$, we assume
		\begin{align}\label{eq:s3}
			\inf_{x \in   I_{\wh{\nu}_N}} \int \frac{1}{ ( v - x)^2 }\d \wh{\nu}_N(v) \geq (1+ \varpi)\lambda^2,
		\end{align}
		for $N$ sufficiently large.
	\end{enumerate}
\end{assumption}
\begin{remark}
	The left-hand side of the inequalities \eqref{eq:s1}, \eqref{eq:s2}, and $\eqref{eq:s3}$ may be infinite. In this case the inequalities should be understood in the sense that $\lambda$ can be chosen as any positive number that stays bounded and converges sufficiently fast in the limit of large $N$. To simplify the exposition, we only consider the case $\lambda = \lambda_0 N^{-\beta}$ for some constants $\beta \geq 0$ and $\lambda_0 \geq 0$. Upon defining $\Tvpi \deq [0, \lambda']$, for an appropriate $\lambda' \geq 0$ ($\lambda ' = \lambda_0 $ if $\beta = 0$), the conditions \eqref{eq:s1} and \eqref{eq:s2} can be interpreted so that the relations hold uniformly in $\lambda \in \Tvpi$ for $N$ sufficiently large. Note that we typically choose $\beta < {1/2}$ since when $\beta \geq {1/2}$, the matrix reduces into a sparse random matrix.
\end{remark}

As studied in \cite{LS}, when $\lambda \leq 1$, Assumption \ref{assm:stableint} is satisfied for any $\nu$ with absolutely continuous density, supported on $[-1, 1]$. However, when $\lambda= \lambda_0 \geq 1$, square-root decay may fail for $\lambda_0$ above a certain threshold, especially when the density of $\nu$ has convex decay at the end of the support. We examine two prime examples satisfying the assumption.
\begin{enumerate}
	\item Choosing $\nu = \frac{1}{2} ( \delta_{-1} + \delta_1) $, $\lambda_0 \geq 0$, we have $I_{\nu} = [-1, 1]$. For $\lambda < 1$, one checks that there exist $\varpi$ and $\mathfrak{t}$ such that Assumption \ref{assm:stableint} is satisfied and that the deformed semicircle law is supported on a single interval with a square-root type behavior at the edges. However, in case $\lambda>1$, the deformed semicircle law is supported on two disjoint intervals.
	\item Let $\nu$ be a centered Jacobi measure of the form
	\begin{align}
		\nu(v) = Z^{-1} (1+v)^{\mathrm{a}} (1-v)^{\mathrm{b}} d(v) \mathbbm{1}_{[-1, 1]}(v), 
	\end{align}
	where $d \in C^{1} ([-1, 1])$, with $d(v) >0$, $-1 < \mathrm{a}, \mathrm{b} < \infty$, and $Z$ a normalization constant. When $\mathrm{a}, \mathrm{b} \leq 1$, there is, for any $\lambda \geq 0$, $\varpi \equiv \varpi(\lambda)>0$ and $\mathfrak{t}$ such that Assumption \ref{assm:stableint} is satisfied. However, if $\mathrm{a} >1$ or $\mathrm{b} >1$, the condition may fail for $\lambda = \lambda_0 $ above $\lambda_{+}$ or $\lambda_{-}$, where
	\begin{align}
		\lambda_{+} \deq \left( \int \frac{\d \nu(v)}{(1-v)^2} \right)^{1/2}, \quad \lambda_{-} \deq \left(\int \frac{\d \nu(v)}{(1+v)^2} \right)^{1/2}.
	\end{align} In this setting, the deformed semicircle law is still supported on a single interval, but the square root behavior at the edge may fail. For more details, refer to \cite{LS, LS14, LL}. 
\end{enumerate}

To conveniently cope with the cases when $(V_i)$ are random, respectively deterministic, we introduce an event $\Xi$ on which the random variables $(V_i)$ exhibit ``typical'' behavior. 

Define the spectral domain 
            $\caD' \deq \{E + \ii \eta : |E| <E_0 , \eta \in (0, 3] \}$.
\begin{definition} \label{def:Xi}
	Let $\Xi \equiv \Xi_N(\mathfrak{a})$ be the event on which the following holds:
	\begin{enumerate}
		\item We have
		\begin{align}
			\inf_{x \in I_{\wh{\nu}_N}} \frac{1}{\lambda^2}  \int \frac{1}{(v- x)^2} \d \wh{\nu}_N(v) \geq (1 + \varpi).
		\end{align}
        Recall that $I_{\mu}$ is the smallest compact interval that contains $\mu$.
		\item For any $z \in \caD' \setminus I_{\nu}$, there exists a constant $C_z>0$, such that, for any integer $M \geq N$,
        \begin{align}
            |m_{\wh{\nu}_{M}}(z) - m_{\nu}(z)| \leq C_z{M}^{-\mathfrak{a}}.
        \end{align}
        Furthermore, for any fixed compact set $\caD \subset \caD'    \subset \bbC^+$ with $\mathrm{dist}\, (\supp \nu, \caD) + \Im z \geq N^{-\fra}$, there exists a constant $C_{\caD}>0$ such that for any integer $M \geq N$, 
		\begin{align} \label{eq:bound}
			\sup_{z \in \caD} |m_{\wh{\nu}_{M}}(z) - m_{\nu}(z)| \leq C_{\caD}{M}^{-\mathfrak{a}}.
		\end{align}
	\end{enumerate}
\end{definition}

The event was previously introduced in \cite{LS14, LSSY,  JL} to study the spectral properties of deformed Wigner matrices. Slight modifications have been made for clarification. In case $V$ is deterministic, $\Xi(\fra_0)$ has full probability for $N$ sufficiently large, for $\fra_0$ given in \eqref{eq:alpha0}. In case $V$ is random, the following result was proved in \cite{JL}.
\begin{lemma}\label{lem:JL}
	For \textbf{random} $V$ and any fixed $\epsilon_0 >0$, there exists $c>0$ such that 
	\begin{align}
		\mathbb{P} \left[\Xi_N \left( \frac{1}{2} - \epsilon_0 \right) \right] \geq 1 -c N^{-\mathfrak{t}},
	\end{align}
	where $\mathfrak{t}$ was given in Assumption \ref{assm:V}.
\end{lemma}
We revisit its proof in Appendix \ref{Appendix_Xi}. Although Definition \ref{def:Xi} is somewhat stronger than those in the proceeding literature, the claims that are proved through the assumption in Definition \ref{def:Xi}, namely Lemma \ref{lem:stabbound} and Lemma \ref{lem:gaussian}, remain invariant.

In the rest of the paper, we write $\Xi_N \equiv \Xi_N(\fra_0)$ if $V$ is deterministic, and $\Xi_N \equiv \Xi_N(1/2 - \epsilon_0)$ if $V$ is random, where $\epsilon_0 > 0$ is small enough but fixed.

We mention that all our proofs are expected to hold for complex Hermitian deformed sparse random matrices, under minor modifications. Moreover, we believe that following the analysis in \cite{HKR} more closely, our result will hold for Hermitian $V$, that is not necessarily diagonal, with suitable assumptions, while we do not pursue in this direction in the current paper.

\subsection{Deformed semicircle law and free convolution} \label{sec:res1} 
The deformed semicircle law is conveniently defined through its Stieltjes transform: For a (probability) measure $\mu$ on $\bbR$, we define its Stieltjes transform $m_{\mu}$, by
\begin{align}
	m_{\mu}(z) \deq \int \frac{\d \mu(v)}{v -z}, \quad (z \in \bbC^+). \end{align}
Note that $m_{\mu}$ is an analytic function in the upper half plane and that $\Im m_{\mu} (z) \geq 0$, $\Im z \geq 0$. Assuming that $\mu$ is absolutely continuous with respect to Lebesgue measure, we recover the density of $\mu$ from $m_{\mu}$ by the inversion formula
\begin{align}
	\mu(E) = \lim_{\eta \searrow 0} \frac{1}{\pi} \Im m_{\mu} (E+ \ii \eta), \quad (E \in \bbR).
\end{align}
We use the same symbols to denote measures and their densities.

The standard semicircle law asserts that the limiting eigenvalue distribution of $W$ converges weakly to standard semicircle law $\rho_{\mathrm{sc}}$. One can easily check that its Stieltjes transform $m_{\mathrm{sc}} \equiv m_{\rho_{\mathrm{sc}}}$ satisfies the relation
\begin{align}
	m_{\mathrm{sc}} (z)= - \frac{1}{m_{\mathrm{sc}}(z) +z}, \quad \Im m_{\mathrm{sc}}(z) \geq 0, \quad (z \in \bbC^+).
\end{align}
In \cite{Pas}, it was proved that the eigenvalue distribution of the deformed Wigner matrices can be described using the functional equation
\begin{align}\label{eq:mfc}
	\mfc(z) = \int \frac{\d \nu(v) }{\lambda v - z - \mfc(z)},\quad \Im \mfc(z) \geq 0, \quad (z \in \bbC^+),
\end{align}
which has a unique solution satisfying $\limsup_{\eta \searrow 0} \Im \mfc (E+ \ii \eta) < \infty$ for all $E \in \bbR$. The deformed semicircle law, denoted by $\rho_{\fc}$ is then defined through its density
\begin{align}
	\rho_{\fc}(E) \deq \lim_{\eta \searrow 0} \frac{1}{\pi} \Im m_{\fc} (E+ \ii \eta), \quad (E \in \bbR),
\end{align}
Likewise, we can define $\wh{m}_{\mathrm{fc}}$, $\wh{\rho}_{\mathrm{fc}}$ by replacing $\nu$ with $\wh{\nu}_N$ in \eqref{eq:mfc}. The measure $\rho_{\mathrm{fc}}$ has been studied in detail in \cite{B, Sh11, LS}. The following result is relevant for our analysis.
\begin{lemma} \label{lem:singlesupport}
	Let $\nu$, and $\lambda$ satisfy Assumption \ref{assm:stableint}. Then there exist $E^{\nu}_{-}, E^{\nu}_{+} \in \bbR$, such that $\supp \rho_{\mathrm{fc}} = [E^{\nu}_{-}, E^{\nu}_{+}]$. Moreover $\rho_{\mathrm{fc}}$ has a strictly positive density on $(E^{\nu}_{-}, E^{\nu}_{+})$. 
\end{lemma}

Free convolution turned out to be a useful concept in interpreting the sum or product of random matrices. Denoting the free convolution of two probability measures $\nu_1,\nu_2$ as $\nu_1 \boxplus \nu_2$, it was proved in \cite{BB} that its Stieltjes transform can be understood as a solution of a system of equations. 

\begin{proposition}[Theorem 4.1, \cite{BB}]\label{thm:fc} Given Borel probability measures $\nu_1, \nu_2$ on $\mathbb{R}$, there exist unique functions $\omega_1, \omega_2: \mathbb{C}^+ \rightarrow \mathbb{C}^+$, satisfying $\lim_{y \to \infty} \omega_{i} (\ii y)/\ii y =1 $, ($i=1, 2$) such that
	\begin{align}
		m_{\nu_1 \boxplus \nu_2 }(z) &= m_{\nu_1} (\omega_1(z)) = m_{\nu_2}(\omega_2(z)),\\
		\omega_1(z) + \omega_2(z) &= z - \frac{1}{m_{\nu_1 \boxplus \nu_2}(z)} ,
	\end{align}
	for all $z \in \mathbb{C}^+$.
\end{proposition}
Using Proposition~\ref{thm:fc}, the deformed semicircle law can be interpreted as the free convolution of the semicircle law~$\rho_{\mathrm{sc}}$ with the law of~$V$ (possibly scaled by a parameter~$\lambda \in \mathbb{R}$). For the connection between free additive convolution and the sum of random matrices, we refer to \cite{V, PV, Ka, BES20b}.

The deformed semicircle law was further extended in an entrywise sense to study the eigenvalue behaviors of more comprehensive models, including Wigner-type matrices and random matrices with correlations among the entries. Considering the entries of $V$ to be fixed, we consider the quadratic vector equation, or more generally, the Matrix Dyson equations, which is the system
\begin{align}
	\frac{1}{\wh{M}_i(z)} = \lambda V_i -z -  \pair{\wh{M}(z)}, \quad i \in \llbracket 1, N \rrbracket,
\end{align}
where we have $\pair{\wh{M}(z)} = \wh{m}_{\mathrm{fc}}(z)$. Refer to \cite{AEK1, AEK, AEKS, EKS} for more details.

\subsection{Main results: Local laws}
Given a real symmetric matrix $H$, we define its resolvent $G \equiv G^H $ by setting
\begin{align}
	G(z)\equiv G^H (z) = \frac{1}{H-z \mathrm{I}} = \frac{1}{W + \lambda V -z \mathrm{I}}, \quad z \in \bbC^+. 
\end{align}
The matrix entries of $G(z)$ are denoted by $G_{ij}(z)$. In the following we often drop the explicit $z$-dependence from the notation for $G(z)$ and its normalized trace $\pair{G(z)}$. In some occasions, we denote $g(z) = \pair{G(z)} $.

Denoting by $\mu_1 \leq \mu_2 \leq \dotsc \leq \mu_N$ the ordered eigenvalues of $H$, we note that $\pair{G}$ is the Stieltjes transform of the empirical eigenvalue distributions $\mu^H$ of $H$, given by 
\begin{align}
	\mu^H \deq \frac{1}{N} \sum_{i=1}^{N} \delta_{\mu_i}.
\end{align}

For $\tau > 0$, define the spectral domain
\begin{align}
	\mc{D}_{\tau} \deq \{E + \ii \eta : |E| < E_0, N^{-1+\tau} \leq \eta \leq 3 \},
\end{align}
where $\tau$ is typically chosen small, satisfying $N^{\tau} < N^{\frac{20}{9} \delta} \leq q^{20/9}$, and $E_0$  is a fixed constant such that $E_0> 3+ \lambda$. Notice that we know that $\norm W \norm \leq 3$, with high probability, and so from spectral perturbation theory, $\norm H \norm \leq 3+ \lambda$ with high probability.

The main result of the paper is the local law for the normalized trace of the resolvent $\pair{G}$, with refinement in the spectral edges. The refined deformed semicircle law is defined in the following proposition, along with several important properties of the refined measure.

\begin{proposition}\label{prop:refv}
	Assume that $\wh{\nu}_N$ satisfy Assumption \ref{assm:esd}, \ref{assm:V}, and \ref{assm:stableint}. The following holds on the event $\Xi_N$.
	\begin{enumerate}
		\item \label{item:prop1} Define a probability measure $\wt{\nu}_N^{\lambda}$ by 
		\begin{align}\label{eq:p1}
			\wt{\nu}_N^{\lambda}(v) = \frac{1}{2 \lambda } \left[ \wh{\nu}_N \left(\frac{1}{\lambda} \left( v + \frac{\sqrt{s}}{q} \right) \right) +  \wh{\nu}_N \left(\frac{1}{\lambda} \left( v - \frac{\sqrt{s}}{q} \right) \right)\right].
		\end{align}
		The measure $\wt{\nu}_N^{\lambda}$ is supported on a compact interval $I_{\wt{\nu}}^{\lambda}$, and there exist $\varpi>0$ such that
		\begin{align}
			\inf_{x \in I_{\wt{\nu}}^{\lambda}} \int \frac{1}{(v - x)^2} \d \wt{\nu}^{\lambda}_N(v) \geq 1+ \varpi, \label{eq:stabref}
		\end{align}
		for $N$ sufficiently large, in case $V$ is deterministic. In case $V$ is random, \eqref{eq:stabref} holds with high probability.  The first moment  (mean) is $m^{1}(\wt{\nu}_N^{\lambda}) = \lambda m^{(1)}(\wh{\nu}_N)$. The second, third, and fourth moments of $\wt{\nu}_N$ are
		\begin{align}
\label{eq:moment}
				m^{(2)}(\wt{\nu}_N^{\lambda}) &=  \lambda^2 m^{(2)}(\wh{\nu}_N) + \frac{s}{q^2} , \quad m^{(3)}(\wt{\nu}_N^{\lambda})  = \lambda^3 m^{(3)}(\wh{\nu}_N) - \frac{s}{q^2} \lambda m^{(1)} (\wh{\nu}_N), \\
				m^{(4)} (\wt{\nu}_N^{\lambda})& = \lambda^4 m^{(4)} (\wh{\nu}_N)  + \frac{ 6s \lambda^2 }{q^2} m^{(2)}(\wh{\nu}_N)  + \frac{s^2}{q^4}.\label{eq:moment2}
		\end{align}
		\item  Define $\wt{\rho}_{\fc}$ as the free additive convolution of the standard semicircle law and $\wt{\nu}_{N}^{\lambda}$. Then the Stieltjes transform of $\wt{\rho}_{\fc}$, denoted $\wt{m}_{\fc}: \bbC^+ \to \bbC^+$ satisfies the functional equation
		\begin{align}\label{eq:reftrace}
			\wt{m}_{\fc}(z) = \int \left(\lambda v - z -  \wt{m}_{\fc}(z)  - \frac{s}{q^2}  \frac{1}{\lambda v - z -  \wt{m}_{\fc}(z)} \right)^{-1} \d \wh{\nu}_N(v). 
		\end{align}
		There exist $\wt{L}_{-}, \wt{L}_+ \in \bbR$, such that $\supp \wt{\rho}_{\fc} = [\wt{L}_{-} , \wt{L}_+]$, with strictly positive density on $(\wt{L}_{-} , \wt{L}_+)$. In particular, when $\lambda \ll 1$, $\wt{L}_+$ admits the asymptotic expansion
		\begin{align}
			\wt{L}_+ = & 2 +  m^{(1)}(\lefnu) +  m^{(2)}(\lefnu) +  m^{(3) }(\lefnu) +  \left( m^{(4)}(\lefnu) - \frac{9 (m^{(2)} (\lefnu))^2 }{4}  \right) + O\left( {\lambda^5}, \frac{1}{q^6} \right),
		\end{align}
		where $m^{(k)}(\lefnu)$ is the $k$-th moment of $\lefnu$. Analogous result holds at the lower edge.
	\end{enumerate}
	Likewise, consider a deterministic probability measure $\bv{\nu}^{\lambda}$ defined through
	\begin{align}
		\bv{\nu}^{\lambda} = \frac{1}{2 \lambda } \left[ {\nu} \left(\frac{1}{\lambda} \left( v + \frac{\sqrt{s}}{q} \right) \right) +  {\nu} \left(\frac{1}{\lambda} \left( v - \frac{\sqrt{s}}{q} \right) \right)\right].
	\end{align}
	The measure $\bv{\nu}^{\lambda}$ satisfies the properties stated in \eqref{eq:p1} and \eqref{eq:stabref} analogously, and the moments of $\bv{\nu}$ are obtained by simply replacing $m^{(k)}(\wh{\nu}_N)$ by $m^{(k)}(\nu)$ in \eqref{eq:moment} and \eqref{eq:moment2}, and furthermore, replacing the mean to be zero, as $\nu$ is assumed to be centered. Define $\breve{\rho}_{\fc}$ as the free additive convolution of the standard semicircle law and $\bv{\nu}_{N}^{\lambda}$. There exist $\bv{L}_-, \bv{L}_+ \in \bbR$ such that $\supp \bv{\rho}_{\fc} = [\bv{L}_-, \bv{L}_+]$, with strictly positive density on $(\bv{L}_-, \bv{L}_+)$. The Stieltjes transform of $\bv{\rho}_{\fc}$, denoted $\bv{m}_{\fc}$, satisfies the trivial analogue of \eqref{eq:reftrace}.
\end{proposition}
The proof of the proposition will be given in Section \ref{sec:refine}.  Most notably, assuming that $\wh{\nu}_N$ satisfies the stability assumption, its refinement also inherits the behavior. Also, observe that $\wt{\nu}^{\lambda}_N$ (respectively $\bv{\nu}^{\lambda}_N$) can also be understood as the \emph{classical additive convolution} of $\wh{\nu}_N^{\lambda} $ (respectively $\nu^{\lambda}$) and the measure $\frac{1}{2} (\delta_{\sqrt{s}q^{-1}} + \delta_{-\sqrt{s}q^{-1}})$. An interpretation of the observation is provided in the end of the section.

In some occasions, it is easier to work with $\wt{\nu}_N$ and $\bv{\nu}_N$, which are $\wt{\nu}^{\lambda}_N $ and $\bv{\nu}^{\lambda}_N$ scaled by $1/\lambda$. We formally define these measures through
\begin{align} \label{eq:lambda-1}
	\wt{\nu}_N(v) \deq \lambda \wt{\nu}_N^{\lambda} (\lambda v), \quad \bv{\nu}_N(v) \deq \lambda \bv{\nu}_N^{\lambda} (\lambda v), \quad v\in \bbR.
\end{align}

We also define the refined quadratic vector equations where for each $i \in \llbracket 1, N \rrbracket$, we define $\wt{\MM}_i : \bbC^+ \to \bbC^+ $ to satisfy
\begin{align}
	{\wt{\MM}_i (z)} \deq  \left(\lambda V_i - z - \wt{m}_{\fc}(z) - \frac{s}{q^2} \frac{1}{\lambda V_i - z - \wt{m}_{\fc}(z)} \right)^{-1} , \quad i \in \llbracket 1, N \rrbracket.\label{eq:wtdef}
\end{align}
We will define an $N \times N$ matrix $\wt{\MM}$, where we omit the $N$-dependence for it is clear, $(\wt{\MM}_{ij}) \deq \delta_{ij} \wt{\MM}_{i}$. Clearly $\pair{\wt{\MM}(z)} = \wt{m}_{\fc}(z)$, and we will interchangeably use these notations when convenient.

In what follows, every notation with tilde stands for those that depend on $\wh{\nu}_N$, and in most cases also $q$, and thus on $N$. Every notation with breve stands for those that depend on $q$, and thus on $N$. Since the notations with breve are averaged over the law of $V$, these values do not depend on $\wh{\nu}_N$, while it can obviously depend on $\nu$.

Now, we are ready to state the averaged local law up to the edge.

\begin{theorem}[Averaged local law up to the edge] \label{thm:refine} For $H = W +\lambda V$, $\lambda \in \Theta_{\varpi}$, suppose that $W$ is a real symmetric sparse random matrix satisfying the assumptions in Definition \ref{def:sparse}, and $V$ is a deterministic or random real diagonal matrix satisfying Assumptions \ref{assm:esd}, \ref{assm:V}, \ref{assm:stableint}. Fix appropriate $\tau$. On the event $\Xi$, we have
	\begin{align} \label{eq:stronglocal}
		|g(z)-\wt{m}_{\fc}(z) |=  \left|\pair{G(z)} - \pair{\wt{\MM}(z)} \right| \prec \min \left\{ \frac{1}{q^{3/2}} + \frac{1}{N^{1/4} q^{1/2}} , \frac{1}{q^2 \sqrt{\vk_E + \eta}} \right\}+ \frac{1}{N\eta},
	\end{align}
	uniformly for $z=E+\ii \eta \in \mathcal{D}_{\tau}$, where $\vk_E \deq \min \{ | E- \wt{L}_+ |, | E- \wt{L}_- | \} $.
\end{theorem} 

Theorem \ref{thm:refine} is proved by first using recursive moment estimates via cumulant expansion method (Section \ref{sec:rme}), then applying the standard stability argument (Section \ref{sec:locallaw}). For detailed outline of the strategy, we refer to Section \ref{outline}. 

We also obtain an entrywise strong local law, proved concurrently with, yet preceding Theorem \ref{thm:refine}.

\begin{theorem}[Entrywise strong local law] \label{thm:strong} Let $H = W+\lambda V$ satisfy the assumptions in Theorem~\ref{thm:refine}. Fix appropriate $\tau$. On the event $\Xi$, we have
	\begin{align}
		\max_{ij} |G_{ij}(z) - \delta_{ij} \wt{\MM}_i(z)| \prec \frac{1}{q} + \sqrt{\frac{\Im \pair{\wt{\MM} (z)}}{N\eta}} + \frac{1}{N\eta},
	\end{align}
	uniformly for $z = E+ \ii \eta \in \mc{D}_{\tau}$.
\end{theorem}
We can obtain the same entrywise estimate after replacing $\wt{\MM}$ by $\widehat{M}$ in Theorem \ref{thm:strong}, and it is believed to be optimal. The use of $\wt{\MM}$ in Theorem \ref{thm:strong} is only for technical convenience in the proof of the refined averaged local law; see Section \ref{sec:locallaw} for details. Another technical ingredient is the weak local law, which establishes that $\max_{ij} |G_{ij}| \prec 1$ on $\Xi$. Its proof is outlined in Section \ref{sec:weak}. 

A consequence of the local laws is the delocalization of eigenvectors. 
\begin{corollary}[Delocalization of eigenvectors] \label{cor:deloc}
	Denote by $(\mbf{u}_i^H)$ the $\ell^2$-normalized eigenvectors of $H$. On the event $\Xi$, we have
	\begin{align} 
		\max_{i \in \llbracket 1, N \rrbracket }\norm \mbf{u}_i^H \norm_{\infty} \prec \frac{1}{\sqrt{N}}.
	\end{align}
\end{corollary}
The assumption ensuring square-root type decay at the spectral edges, namely Assumption \ref{assm:stableint}, is crucial for the complete delocalization. For deformed Wigner matrices, when there is ``convex decay'' at the spectral edges, the eigenvectors associated to the largest (or smallest) eigenvalues are partially localized \cite{LS16}.

We finish the section with a few remarks on the interpretation of the results. We first recall the averaged local law for sparse random matrices established in \cite{LS18}, which is optimal up to deterministic shift of order $O(q^{-2})$ at the spectral edges. 
\begin{proposition}[Theorem 2.4, \cite{LS18}]  \label{prop:LS}
	There exists an algebraic function $\wt{m}_{\mathrm{sc}}:\bbC^+ \to \bbC^+$ satisfying the polynomial equation
	\begin{align}
		1 + z  \wt{m}_{\mathrm{sc}}(z)+ \wt{m}_{\mathrm{sc}}^2(z) + \frac{s}{q^2}  \wt{m}_{\mathrm{sc}}^4(z) = 0, \label{eq:sparseref} 
	\end{align}
	where $\wt{m}_{\mathrm{sc}}$ is the Stieltjes transform of a deterministic symmetric compactly supported probability measure $\wt{\rho}_{\mathrm{sc}}$. Then, the normalized trace $g^{W}$ of the resolvent of $W$ satisfies
	\begin{align}\label{eq:sparsels}
		|g^{W}(z) - \wt{m}_{\mathrm{sc}}(z)| \prec \frac{1}{q^2} + \frac{1}{N\eta},
	\end{align}
	uniformly on the domain $\caE \deq \{z = E+ \ii \eta \in \bbC^+ : |E|<3, \eta \in (0,3] \}$, $z=E+\ii \eta$.
\end{proposition}

While $H$ reduces to a sparse random matrix in the limit $\lambda\sim N^{-1/2}$, Theorem~\ref{thm:refine} is weaker than Proposition \ref{prop:LS} at the spectral edges. The bulk fluctuations can be improved by choosing an equilibrium measure that accounts for the random fluctuation term, which is expected to be of size $N^{-1/2}q^{-1}$. The bound of $N^{-1/4}q^{-1/2}$ in our averaged local law is believed to reflect this fluctuation.

To analyze the local laws at the spectral edges, we first compare the refined measures to those introduced when studying sparse random matrices (see \cite{LS18, HLY20}) and $\wt{\rho}_{\fc}$ in the vicinity of $\lambda \sim 0$. Although the measure $\wt{\rho}_{\mathrm{sc}}$ obtained through \eqref{eq:sparseref} and $\rho_{\mathrm{sc}} \boxplus \frac{1}{2}(\delta_{\sqrt{s}q^{-1}} +\delta_{-\sqrt{s}q^{-1}})$ are different equilibrium measures, they share important properties: (1) they approach to $\rho_{\mathrm{sc}}$ as $N\to \infty$; (2) they exhibit square-root decay type behavior at the spectral edges; (3) their upper endpoint admits the asymptotic expansion $2+sq^{-2} + O(q^{-4})$ (analogously at the lower endpoint). In the limit $ \lambda \lesssim N^{-1/2} $, the asymptotic expansion of $\bv{L}_+$ becomes $2 + sq^{-2} + O(N^{-1} +q^{-4})$.


In fact, the refinement coincides with the picture of the eigenvalue behavior of sparse random matrices, established in the series of works \cite{LS18, HLY20, HK, Le, HLY25}. In \cite{HK}, He and Knowles suggested that the fluctuation of any single eigenvalue of sparse random matrices consists of two independent components: \textit{a random matrix} component, and a \textit{sparsity} component. In our refinement, the sparsity is encoded by the component $\frac{1}{2}(\delta_{\sqrt{s}q^{-1}} +\delta_{-\sqrt{s}q^{-1}})$, which is indeed independent of $\wh{\nu}^{\lambda}_N$. Assessing these measures as independent components that induce the shifts of the eigenvalues, we take the classical convolution of the sparsity component and the potential (or external source) $V$, then take the free convolution with the semicircle law, a random matrix component. 

From this perspective, we conjecture that the eigenvalues of sparse random matrices decompose into three components: (1) a random matrix component, (2) a sparsity component, and (3) contribution arising from the deformation $\lambda V$. At the level of the local law, the influence of $V$ is decoupled by conditioning on the event $\Xi$. On this event, we expect that the leading fluctuations of extremal eigenvalues are described by the classical convolution of GOE Tracy–Widom distribution of size $N^{-2/3}$ and a Gaussian distribution of order $N^{-1/2}q^{-1}$ shifted by $O(q^{-2})$, asymptotically. With a suitably chosen equilibrium measure, we expect optimality, with fluctuation bounds matching \eqref{eq:sparsels}. Since our equilibrium measure does capture the deterministic shift up to $O(q^{-2})$, we believe that the main difficulty lies in controlling the fluctuations in the cumulant expansion. The issues are discussed further in Section~\ref{outline} and Remark~\ref{rem:import}. Finally, as the extremal eigenvalue analysis concerns $\eta \sim N^{-2/3-o(1)}$, we expect that Theorem \ref{thm:refine} yields optimal results of eigenvalues up to sparsity $q = N^{\phi}$ for any $\phi \geq 2/9$, as supported by the results in the next section.

\subsection{Main results: Eigenvalue behaviors} Applying the local laws, we prove an estimate on the eigenvalue locations and the rigidity phenomena. Define the classical location $\wt{\gamma}_i$ of the eigenvalues $\mu_i$ with respect to the reference density $\wt{\rho}_{\fc}$ by
\begin{align}
	\int_{-\infty}^{\wt{\gamma}_i} \wt{\rho}_{\fc} (x) \d x = \frac{i}{N},\quad i \in \llbracket 1, N \rrbracket .
\end{align}
Note that in case $(V_i)$ are random, $(\wt{\gamma}_i)$ are random as well. Adapting the methods in \cite{EKYY1,LS}, we obtain the following.

\begin{theorem}[Eigenvalue location and rigidity] \label{thm:rigidity}
	Let $H=W +\lambda V$ satisfy the assumptions in Theorem \ref{thm:refine}. Assume the event $\Xi$. Consider $E$ satisfying $|E|<E_0 <3+\lambda$. Denote $\vk_{E} \deq \min\{|E - \wt{L}_-|, |E -\wt{L}_+| \}$, where $\wt{L}{_{\pm}}$ is the endpoint of $\wt{\rho}_{\fc}$, defined in Proposition \ref{prop:refv}. For each $i \in \llbracket 1, N \rrbracket$, let $\mu_i$ be the eigenvalue of $H$, and $\wt{\gamma}_i$ the corresponding classical location with respect to $\wt{\rho}_{\fc}$, in ascending order. Fix arbitrary (small) $\epsilon>0$ and (large) $D\geq 1$. Then, the following statement holds for all $j \in \llbracket 1, N \rrbracket$ with probability at least $1- N^{-D}$, for sufficiently large $N \geq N_0(\epsilon, D)$:
	\begin{enumerate}
		\item If $\max\{ \vk_{\mu_{j}}, \vk_{\wg_{j}} \} \leq N^{\epsilon} (q^{-3} +N^{-2/3})$, then
		\begin{align}\label{eq:each0}
			|\mu_{j}- \wt{\gamma}_j | \leq N^{\epsilon} \left( \frac{1}{q^3} + \frac{1}{N^{2/3}} \right).
		\end{align}
		\item  If $N^{\epsilon} (q^{-3} +N^{-2/3}) \leq  \max\{ \vk_{\mu_{j}}, \vk_{\wg_{j}} \} \leq N^{\epsilon} (q^{-4} (q^{-3} + N^{-2/3})^{-1} + q^{-2} + N^{-2/3})$, then
		\begin{align}\label{eq:each1}
			|\mu_j - \wg_j| \leq N^{{\epsilon}} \left[ \frac{1}{\underline{j}^{1/3}} \left(\frac{1}{ N^{2/3}} +\frac{N^{1/3}}{q^{9/2}} \right)+ \left(\frac{\underline{j}}{N}\right)^{1/3} \left(\frac{1}{q^{3/2}} + \frac{1}{N^{1/3}} \right) \right],
		\end{align}
		where we abbreviated $\underline{j} \deq \min \{j, N-j \}$.
		
		\item  If $ \max\{ \vk_{\mu_{j}}, \vk_{\gamma_{j}} \} \geq N^{\epsilon} (q^{-4} (q^{-3} + N^{-2/3})^{-1} + q^{-2} + N^{-2/3})$, then
		\begin{align}\label{eq:each2}
			|\mu_j - \wg_j| \leq N^{\epsilon} \left[ \frac{1}{\underline{j}^{1/3}}  \frac{1}{N^{2/3}} + \frac{1}{\underline{j}^{2/3}} \frac{N^{2/3}}{q^{5}}  + \frac{1}{q^2} \right],
		\end{align}
		where we abbreviated $\underline{j} \deq \min \{j, N-j \}$.
	\end{enumerate}
	In particular, under the assumption $q =N^{\phi}$ with $\phi \geq 2/9$, the estimate simplifies to the following assertion. On the event $\Xi$, for any (small) $\epsilon>0$ and (large) $D>0$, the event
	\begin{align}
		\begin{split}
			 \Bigg\{  \exists j :  |\mu_j - \wg_j|  > N^{\epsilon} \Bigg(\frac{1}{\underline{j}^{1/3}}\frac{1}{N^{2/3}} + \frac{\underline{j}^{1/3}}{N^{2/3}} \mathbbm{1}_{\left\{ \underline{j} \leq N^{\epsilon }(1+ Nq^{-3})^2 \right\}} + \frac{1}{q^2} \mathbbm{1}_{\left\{ \underline{j} \geq N^{\epsilon } (1+ Nq^{-3})^2  \right\}} \Bigg) \Bigg\} , \label{eq:rigidlong}
		\end{split}
	\end{align}
	holds with probability at most $N^{-D}$, for all sufficiently large $N \ge N_0(\epsilon, D)$.
\end{theorem}
\begin{remark}
	Firstly, we remark that $q^{-4} (q^{-3} + N^{-2/3})^{-1}  \sim q^{-1} $ when $q \lesssim N^{2/9}$, while $q^{-4} (q^{-3} + N^{-2/3})^{-1} \sim N^{2/3}q^{-4}$ when $q \gtrsim N^{2/9}$. Case (2) should be understood so that it interpolates between cases (1) and (3). Also note that case (2) becomes vacuous when $q\gtrsim N^{1/3}$.
	
	Theorem \ref{thm:rigidity} shows that, on the event $\Xi$, when $\phi \geq 2/9$, the fluctuations of the extremal eigenvalues are of order $N^{-2/3}$, consistent with the GOE Tracy–Widom fluctuation scale. Under the assumption $\phi \geq 1/3$, the estimate \eqref{eq:rigidlong} simplifies to
	\begin{align}
		\max_j |\mu_j - \wt{\gamma}_j | \prec \frac{1}{N^{2/3}} \frac{1}{{\underline{j}}^{1/3}} + \frac{1}{q^2 },
	\end{align}
	which coincides with the rigidity result for sparse random matrices \cite{EKYY1}. 
	
	In a different limit $\phi =1/2$, where the model recovers the deformed Wigner matrices, the result reduces to the rigidity result in \cite[Corollary 3.4] {LSSY}.
\end{remark}
Finally, in the random $V$ case, we prove that the leading fluctuation of the extremal eigenvalue converges to a Gaussian, given by the central limit theorem. The limiting fluctuation has the same variance as in the deformed Wigner matrix case \cite{LS14, LS16}, whereas the mean is shifted by the effect of the sparsity. The proof closely parallels that in the existing literature.
\begin{theorem}\label{thm:gauss} 
	Let $L_+$ and $\bv{L}_+$ be the upper endpoints of $\rho_{\fc}$ and $\bv{\rho}_{\fc}$, respectively, and $\mu_N$ be the largest eigenvalue of $H = W + \lambda V$. When $\lambda \gg \max\{q^{-1}, N^{-1/6} , \sqrt{N}q^{-3}\}$, the rescaled fluctuation $\sqrt{N} \lambda^{-1} (\mu_N - \bv{L}_+)$ converges in distribution to a Gaussian random variable with mean $0$ and variance \begin{align}
		\lim_{N \to \infty }\lambda^{-2} (1 - \mfc(L_+)^2).
	\end{align}
	In particular, recalling $\lambda = \lambda_0 N^{-\beta}$, 
	when $\beta >0$, the variance is $m^{(2)}(\nu)$. Analogous statement holds for the smallest eigenvalue.
\end{theorem}

\begin{remark} When $\phi \geq 2/9$, the technical restriction in Theorem~\ref{thm:gauss} becomes vacuous, and the condition reduces to $\lambda \gg \{ q^{-1}, N^{-1/6} \}$. This reduced condition is believed to describe the natural regime where the Gaussian fluctuation, of size $N^{-1/2} \lambda$, dominates over both the expected GOE Tracy-Widom fluctuation of size $N^{-2/3}$ and the Gaussian fluctuation from the random fluctuation term in sparse random matrices, which is of size $N^{-1/2}q^{-1}$. We expect that the restriction on $\phi$ can be removed entirely upon further improvements of the local law.
\end{remark}

Proofs of Theorem \ref{thm:rigidity} and Theorem \ref{thm:gauss} are provided in Section \ref{sec:behavior}. In proving these results, we will also prove a bound on the extremal eigenvalues, Lemma \ref{thm:opnorm}, and estimates on the density of states, Proposition \ref{prop:loc} and Proposition \ref{prop:int}.

\section{Preliminaries}\label{sec:prelim}

\subsection{Properties of the refined measure and stability bounds}\label{sec:refine}

In this section, we recall and extend important properties of $\wt{m}_{\mathrm{fc}} = \pair{\wt{\MM}}$ and $ \mfc$ and their inversions $\wt{\rho}_{\mathrm{fc}}$ and $ \rho_{\mathrm{fc}}$, introduced in Section \ref{sec:res1}. Recall that $\wt{m}_{\mathrm{fc}} = \pair{\wt{\MM}}$ (and $\breve{m}_{\fc}$) is the Stieltjes transform of the free convolution between the semicircle law $\rho_{\mathrm{sc}}$ and $\wt{\nu}_N^{\lambda}$ (and $\breve{\nu}^{\lambda}$, respectively), where $\wt{\nu}_N^{\lambda}$ and $\breve{\nu}^{\lambda}$ respect the stability conditions. 

\begin{remark}\label{rem:classify}
	Before discussing the square-root behavior of the free convolution measures, it is helpful to understand $\wt{\nu}_N^{\lambda}$, where to utilize the known results it is beneficial to identify the ($N$-dependent) scaling factor that gives rise to a ($N$-dependent) measure with $O(1)$ variance, supported compactly on a interval of scale $O(1)$. With the parameters $q^{-1}$ and $\lambda$ competing, we consider two different scaling with the following measures
		\begin{align}\wt{\nu}_N (v) &= \frac{1}{2} (\wh{\nu}_N (v + \sqrt{s} q^{-1} \lambda^{-1}) + \wh{\nu}_N (v - \sqrt{s}q^{-1} \lambda^{-1} )), \\	
			\ttilde[2]{\nu}_N (v) &= \frac{\sqrt{s}}{q \lambda } \frac{1}{2}\left[\wh{\nu}_N \left(\frac{\sqrt{s}}{q\lambda} (v+1) \right) + \wh{\nu}_N \left(\frac{\sqrt{s}}{q\lambda } (v-1) \right) \right].
	\end{align}
		Note that $\wt{m}_{\wh{\nu}_N}$ is the Stieltjes transform of $\lambda^{-1} \wt{\nu}_N (\lambda^{-1} v)$, or $(\sqrt{s}q^{-1})^{-1} \ttilde[2]{\nu}_N ((\sqrt{s}q^{-1})^{-1} v)$ in all regimes. 
	
	When $q^{-1} \lambda^{-1} \ll 1$, we identify $\wt{\nu}_N^{\lambda}$ as $\wt{\nu}_N$ scaled by $\lambda$, as in this regime the support of $\wt{\nu}_N$ is of order $O(1)$ and compact, and its variance is $\sigma^2 + sq^{-2}\lambda^{-2}$. Furthermore $\wt{\nu}_N$ converges weakly to $\nu$ as $N \to \infty$.
	
	When $q^{-1}\lambda^{-1}\gg 1$, we identify $\wt{\nu}_N^{\lambda}$ as $\ttilde[2]{\nu}_N$ scaled by $\sqrt{s} q^{-1}$, as in this regime the support of $\ttilde[2]{\nu}_N$ is of order $O(1)$ and its variance is $1+ s^{-1} \lambda^2 q^2 \sigma^2$. By using the assumption $s >c_4 $, $\ttilde[2]{\nu}_N$ converges weakly to  $\frac{1}{2} (\delta_{-1} + \delta_1)$ as $N \rightarrow \infty$.
	
	In the borderline case where $q^{-1} \lambda^{-1} \sim 1$ both $\wt{\nu}_N$ and $\ttilde[2]{\nu}_N$ are compactly supported probability measures of order $O(1)$. Both $\wt{\nu}_N$ and $\ttilde[2]{\nu}_N$ converges weakly to a compactly supported measure of order $O(1)$ when $s$ converges as $N \to \infty$; else, the convergence is not guaranteed. Assuming that $\sqrt{s}q^{-1} \lambda^{-1}  \to \mathsf{s}$ as $N \to \infty$, then $\wt{\nu}_N$ converges weakly to $\frac{1}{2} (\nu(\cdot +\mathsf{s}) + \nu(\cdot-\mathsf{s}) )$ as $N \to \infty$, and $\ttilde[2]{\nu}_N$ converges weakly to $\frac{\mathsf{s}}{2} (\nu(\mathsf{s} (\cdot+1)) + \nu(\mathsf{s}(\cdot-1)) )$.

	
	This indicates that $\wt{\nu}$ and $\ttilde[2]{\nu}$ may behave differently compared with $\wh{\nu}$. These issues are not significant due to the following reasons. In the regimes of (2) and (3), we have $\lambda = o(1)$, and so the deformed semicircle law will converge to the semicircle law in the limit $N \to \infty$. More importantly, the stability criteria in Proposition \ref{prop:refv} hold regardless of the regimes divided by the size of $q^{-1}\lambda^{-1}$, so the stability bounds in Lemma \ref{lem:stabbound} hold accordingly for $N$ large enough. 
\end{remark}

The following lemma asserts the square root behavior of $\rho_{\fc}$, $\wh{\rho}_{\fc}$, $\breve{\rho}_{\fc}$, and $\wt{\rho}_{\fc}$, and properties of their Stieltjes transform $\mfc, \wh{m}_{\fc}, \breve{m}_{\fc}, \wt{m}_{\fc}$, provided that $\bv{\nu}^{\lambda}_N$ and $\wt{\nu}^{\lambda}_N$ satisfy the stability condition. Moreover, we prove that $\wh{m}_{\fc}$ behaves qualitatively in the same way as $m_{\fc}$, and likewise that $\wt{m}_{\fc}$ behaves qualitatively in the same way as $\bv{m}_{\fc}$.

\begin{lemma}[Square-root behavior]\label{lem:stabbound}
	Let $\nu$ and $\wh{\nu}$ satisfy Assumption \ref{assm:stableint}.  Then $\rho_{\fc}$ and $\mfc$ have the following properties: 
	\begin{enumerate}
		\item  There exist ${L}_{-}, L_{+} \in \mathbb{R}$, with $L_{-} < 0 < L_{+}$, such that $\rho_{\fc}$ has support $[L_{-}, L_{+}]$, and there exists a constant $C>1$ such that, for all $\lambda \in \Theta_{\varpi}$, 
		\begin{align}\label{eq:sqrt}
			C^{-1} \sqrt{\kappa_E} \leq \rho_{\fc} (E) \leq C \sqrt{\kappa_E}, \quad E \in [L_- , L_+],
		\end{align}
		where $\kappa_E \deq \min \{|E- L_{-}|, |E- L_+|  \}$.
		\item Recall the spectral domain 
            $\caD' \deq \{E + \ii \eta : |E| <E_0 , \eta \in (0, 3] \}$. For any $z \in E + \ii \eta \in \caD'$,
		\begin{align}\label{eq:stab1}
			\Im \mfc  \sim \begin{cases} 
				\sqrt{\kappa_E + \eta}, & E \in [L_{-}, L_{+}] , \\
				\frac{\eta}{\sqrt{\kappa_E + \eta}}, & E \in [L_{-}, L_{+}]^c .
			\end{cases}
		\end{align} 
		\item There exist a constant $C >1$ such that for all $z \in \mc{D}'$ and $v \in \supp \nu$, such that 
		\begin{align}\label{eq:stab2}
			C^{-1} \leq | \lambda v - z - \mfc(z) | \leq C.
		\end{align}
		\item There exist a constant $C >1$ such that for any $z = E+ \ii \eta \in \caD'$,  we have
		\begin{align}\label{eq:stab3}
			C^{-1} \sqrt{\kappa_E + \eta} \leq \left| 1 - \int \frac{1}{(\lambda v - z - \mfc(z))^2 } \d \nu(v) \right| \leq C \sqrt{\kappa_E + \eta} ,
		\end{align}
		\item There exist constants $C>1$ and $c_0 >0$ such that for all $z = E+ \ii \eta \in \caD'$ satisfying $\kappa + \eta \leq c_0$,
		\begin{align} \label{eq:stab4}
			C^{-1} \leq \left|\int \frac{\d \nu(v)}{(\lambda v - z - \mfc(z))^3} \right| \leq C,
		\end{align}
		furthermore, there is $C>1$, such that for all $z \in \caD'$
		\begin{align}
			\left| \int \frac{\d \nu(v)}{(\lambda v - z - \mfc(z))^3 } \right| \leq C.
		\end{align}
	\end{enumerate}
	Moreover, all constants in (1)-(5) can be chosen uniformly in $\lambda \in \Theta_{\varpi}$. 
	
	Assume now that \eqref{eq:stabref} holds for $\wt{\nu}_N$ and $\bv{\nu}$. The statements (1)-(5) hold when replacing $\mfc$ with $\breve{m}_{\fc}$ and $L_{\pm}$ with $\breve{L}_{\pm}$, where in this case the constants can be chosen uniformly in $\Theta_{\varpi}$, the normalized fourth cumulant $s = Nq^{2} \caC_4$, and $N$.  
	Furthermore, on $\Xi_N (\mathfrak{a}_0)$, the following hold for sufficiently large $N$:
	The statements (1)-(5) when replacing $\mfc$ with $\wh{m}_{\fc}$ (respectively, $\wt{m}_{\fc}$), and the other quantities analogously, where all constants can be chosen uniformly in $\lambda \in \Theta_{\varpi}$ and $N$ (also the normalized fourth cumulant in case of $\wt{m}_{\fc}$). Finally, there exist constants $C, c >0$ such that for any $z \in \caD_{0}$, 
	\begin{align}
		|\wh{m}_{\fc}(z) -\mfc(z) | \leq N^{-\frac{c \mathfrak{a}_0}{2} }, &\quad |\wh{L}_{\pm} - L_{\pm} | \leq N^{- c \fra_0},\\ 
		\quad |\wt{m}_{\fc}(z) - \breve{m}_{\fc}(z) |  \leq N^{-\frac{c \mathfrak{a}_0 }{2}}, &\quad |\wt{L}_{\pm} - \bv{L}_{\pm} | \leq N^{- c \fra_0},\label{eq:eqeq}
	\end{align}
	for sufficiently large $N$, for all $\lambda \in \Theta_{\varpi}$ and also independent on the fourth normalized cumulant.
	
\end{lemma}
\begin{proof}
	The statements for $m_{\fc}$ and $\wh{m}_{\fc}$ were proved in \cite{LS, LSSY}. Notice that we have assumed that the stability condition for $\wt{\nu}^{\lambda} $ (and $\bv{\nu}^{\lambda}$) holds. (Refer to \eqref{eq:stabref} of Proposition \ref{prop:refv} (1). We note that Proposition \ref{prop:refv} (1) is proved without using the current lemma.) Following the proof of Lemma A.1 in \cite{LSSY}, we can prove that $\wt{\rho}_{\fc}$ (and $\bv{\rho}_{\fc}$) exhibits square-root decay at the spectral edges, uniformly on $N$, $\lambda \in \Theta_{\varpi}$, and the fourth normalized cumulant $s$. Furthermore, analogous statements of (1)-(5) follow from it. The bound for $|\wt{m}_{\fc} - \breve{m}_{\fc}| $ follows from bootstrapping argument, verbatim to the bound on $|\wh{m}_{\fc} - \mfc|$. Details are presented in Appendix \ref{appendix1}.
\end{proof}

We are ready to prove Proposition \ref{prop:refv}. 
\begin{proof}[Proof of Proposition \ref{prop:refv}] Proof to assertion (1) is technical, and it is postponed to Appendix \ref{appendix1}. The form \eqref{eq:reftrace} is then obtained by Proposition \ref{thm:fc}. With assertion (1) in hand, we have the stability bound in Lemma \ref{lem:stabbound}. Gathering several known properties of the free convolution with the semicircle law, namely Lemma \ref{lem:singlesupport} and \cite[Theorem 2.2]{LS14}, we obtain the statement (2). Analogous treatment applies for $\bv{\rho}_{\fc}$. 
\end{proof}

\subsection{Resolvent, minors, and Resolvent identities}

Let $A$ be an $N\times N$ real symmetric matrix, and denote its resolvent as $G^A(z) \deq (A-z)^{-1}$, $z \in \bbC^+$. Below we often drop the superscript $A$ to ease the notation. 

Let $\bbT \subset \llbracket 1, N \rrbracket$. Then we define $A^{(\bbT)}$ as the $(N -| \bbT|) \times (N - |\bbT|) $ minor of $A$ obtained by removing all columns and rows of $A$ indexed by $i \in \bbT$. We do not change the names of the indices of $A$ when defining $A^{(\bbT)}$. More specifically, we define an operation $\pi_i $, $i \in \llbracket 1, N \rrbracket$ on the probability space by
\begin{align}
	(\pi_i (A))_{kl} \deq \mathbbm{1} ( k \not = i) \mathbbm{1} ( l \not = i) A_{kl}.
\end{align}
Then, for $\bbT \subset \llbracket 1, N \rrbracket$, we set $\pi_{\bbT} \deq \Pi_{i \in \bbT} \pi_i$ and define
\begin{align}
	A^{(\bbT)} \deq ((\pi_{\bbT}(A))_{ij})_{i, j \not \in \bbT}.
\end{align}
The resolvents $G^{(\bbT)} $, are defined in an obvious way using $A^{(\bbT)}$. Moreover, we use the shorthand notations
\begin{align}
	\sum_{i}^{(\bbT)} \deq \sum_{\substack{i=1\\ i \not \in \bbT} }^{N}, \quad \sum_{i \not = j}^{(\bbT)} \deq \sum_{\substack{i=1, j=1\\i\not =j, i, j \not \in \bbT}}^N, 
\end{align}
abbreviate $(i) = (\{ i\})$, $(\bbT i) = (\bbT \cup \{i\})$. In resolvent entries $G_{ij}^{(\bbT)}$ we refer to $\{ i, j\}$ as \textit{lower indices} and to $\bbT$ as \textit{upper indices}.

\begin{lemma}
	Let $A = (A_{ij})$ be a Hermitian $N \times N$ matrix. Consider the resolvent $G(z) \deq (A-z)^{-1}$, $z \in \bbC^+$. Then, for $i, j, k, l \in \llbracket 1, N \rrbracket$, the following identities hold:
	\begin{enumerate}
		\item For $i, j \not = k$, 
		\begin{align}
			G_{ij} = G_{ij}^{(k)} + \frac{G_{ik} G_{kj}}{G_{kk}}.
		\end{align}
		\item For $i \not = j$,
		\begin{align}
			G_{ij} = - G_{ii} \sum_{k}^{(i)} A_{ik} G_{kj}^{(i)} = - G_{jj} \sum_{k}^{(j)} G_{ik}^{(j)} A_{kj}, \quad 
			G_{ij} = - G_{ii} G_{jj}^{(i)} \left(A_{ij} - \sum_{kl}^{(ij)} A_{ik} G_{kl}^{(ij)} A_{lj} \right).
		\end{align}
		\item For deterministic $z \in \mathbb{C} \setminus \mathbb{R}$, we have the differential rule
		\begin{align}
			{\frac{\pd G_{ij} }{\pd A_{kl}} = - \frac{1}{1+ \delta_{kl}} (G_{ik}G_{lj} + G_{il} G_{kj}).}
		\end{align}
		\item Ward identity: For any $i$,
		\begin{align}
			\sum_{n=1}^N |G_{in} |^2 = \frac{\Im G_{ii} }{\Im z}.
		\end{align}
	\end{enumerate}
\end{lemma}

In our model, to deal with $\lambda V$, for any vector $\mbf{b} = (b_1, \dotsc, b_N) \in \bbC^N$ with $\norm \mbf{b} \norm = O(1)$, we define an auxiliary quantity $\G_{ij} = G_{ij} b_j$, for each $i, j \in \llbracket  1, N \rrbracket$. The inequalities and bounds that follow from the Ward identity can be applied for $\G_{ij}$, combined with the Cauchy--Schwarz inequality, since we have assumed $\lVert \mbf{b} \rVert= O(1)$. Refer to \cite[Section 3.1]{HKR} for details and proof, which also outlines the treatment of Ward identity to establish the weak local law. Once the weak local law and eigenvector delocalization is proved, the estimate can be improved further, promoting $G_{ii}$ into its average. Refer to Lemma \ref{lem:wardd} for details.

\subsection{Outline of proof} \label{outline}
Our proof of the local laws consists of (1) a \emph{stochastic step} establishing a self-consistent equation of $G$ with a random error term, and (2) a \emph{deterministic step} analyzing the stability of the self-consistent equation \cite{HKR}. The first step depends on the model, where the strength of the self-consistent equation determines the strength of the local law. In this section, we outline the proof of each step. 

The cumulant expansion is the main ingredient of the first step. We refer to \cite{HKR} for the proof of the following lemma.
\begin{lemma}[Cumulant expansion, generalized Stein's Lemma] 
	\label{lem:genStein}
	Fix $\ell \in \mathbb{N}$, and let $F \in C^{\ell +1} (\mathbb{R}; \mathbb{C}^{+})$. Let $X$ be a centered random variable with finite moments to order $\ell +2$. Then, 
	\begin{align}
		\E [ X F(X)] = \sum_{r=1}^{\ell} \frac{\kappa^{(r+1)}(X)}{r!} \E \left[ F^{(r)} (X) \right] + \E \left[ \caR_{\ell} (X F(X)) \right], \label{eq:genStein}
	\end{align}
	where $\E$ denotes the expectation with respect to $X$, $\kappa^{(r+1)} (X)$ denotes the $(r+1)$-st cumulant of $X$ and $F^{(r)}$ denotes the $r$-th derivative of the function $F$. The error term $\caR_{\ell}(XF(X))$ in \eqref{eq:genStein}, depending on $F$ and $X$, satisfies
	\begin{align}\label{eq:Remainder}
		\left| \E \left[ \caR_{\ell} (X F(X)) \right] \right| \leq& O(1) \cdot \left(\E \sup_{|x| \leq |X|} \left| F^{(\ell +1)} (x) \right|^2 \cdot \E \left|X^{2\ell +4} \mathbbm{1} (|X|>t ) \right| \right)^{\frac{1}{2}}\\& + O(1)\cdot \E \left[|X|^{\ell +2} \right] \cdot \sup_{ |x| \leq t }\left|F^{(\ell+1)} (x) \right|,
	\end{align}
	where $t \geq 0$ is an arbitrary fixed cutoff.
\end{lemma}


We now give a heuristic explanation for why our refinement naturally takes the form given in Proposition \ref{prop:refv}. In the spirit of the cumulant expansion method, the deterministic approximation of the resolvent is expected to satisfy the system of equations
\begin{align}\label{eq:ref0}
	\frac{1}{\wt{M}_i (z) } \stackrel{?}{=} \lambda V_i - z - \pair{\wt{M}(z)} - \frac{s}{q^2} \pair{\wt{M}(z) \odot \wt{M} (z)} {\wt{M}}_i (z), \quad i \in \llbracket 1, N \rrbracket.
\end{align}
It turns out that we cannot replace $\pair{\wt{M} \odot \wt{M}} $ by $\pair{\wt{M}}^2$, due to the anisotropy of $G_{ii}$ and $\wt{M}_i$. The system is significantly complicated, and it poses several technical challenges to demonstrate that the solution to the given system maps the upper half-plane to itself. Since the refinement is aimed at the spectral edges, we expect that $\pair{\wt{M} \odot \wt{M}}$ is close to $1$; refer to \cite{LS}. By replacing $\pair{\wt{M} \odot \wt{M}}$ by $1$ and some minor adjustments, we retrieve the system \eqref{eq:wtdef}.

To establish a self-consistent equation of $G$ with a suitable error term, we must construct a function on $G$ on which we will perform the cumulant expansion. We define auxiliary $N\times N$ (random) matrices $P=(P_{ij})$ and $\mathsf{P} = (\mathsf{P}_{ij})$ through
\begin{align}
	P_{ij} \equiv P_{ij}(z; G) \deq& (WG)_{ij} + \pair{G} G_{ij} = \delta_{ij} + z G_{ij} + \pair{G}G_{ij} - \lambda V_i G_{ij}\\ 
	\mathsf{P}_{ij} \equiv \mathsf{P}_{ij}(z, G) \deq& (WG)_{ij}+ \pair{G} G_{ij} + \frac{s}{q^2}\pair{G \odot G} G_{ii}G_{ij} \\
	=&\delta_{ij} + z G_{ij} + \pair{G} G_{ij} - \lambda V_i G_{ij} + \frac{s}{q^2}\pair{G \odot G} G_{ii}G_{ij} ,
\end{align}
for $i, j \in \llbracket 1, N \rrbracket$. To further cope with the anisotropy, we shall consider $P_{ij}$ and $\sfP_{ij}$ after multiplying them with an arbitrary deterministic diagonal matrix $B=(B_{ij})$, whose diagonals we denote by $(B_i)$, that satisfies $\norm{B}\norm =O(1)$. We further define $\caP (B, G) \deq \pair{B\mathsf{P}}$. The entries of $B$ will later be replaced by $(\lambda V - z - \pair{G})^{-1}$ or its variants, where they are of order $O(1)$ with high probability. This method was first introduced in \cite{HKR}. See also the proof of Lemma \ref{lem:61}.

Our first recursive estimate concerns estimates on $ \max_{ij} |B_i P_{ij}|$, required to establish the weak local law. We defer the details to Section~\ref{sec:weak} as they follow standard arguments. Building on the weak local law, we obtain self-consistent estimates for the entrywise quantities $ B_i \mathsf{P}_{ij} $ and their average $ \caP(B, G) $, which lead to the optimal entrywise and refined averaged local laws. These are summarized in the following propositions.

\begin{proposition} \label{prop:weakrme} Recall the assumption on $H=W+\lambda V$ given in Theorem \ref{thm:refine}. Fix appropriate $\tau$. Assume that $\pair{G} - \pair{\wt{\MM}} = O_{\prec} (\vt)$ on $\caD_{\tau}$, for some deterministic $\vt \in [q^{-2}, N^{\tau/10}]$. Then, on the event $\Xi$, we have  
	{\begin{align}
		\max_{ij} |B_i \sfP_{ij} | \prec \frac{1}{q} + \sqrt{\frac{\Im \pair{{\wt{\MM} }}+\vt  }{N \eta }} ,
	\end{align}}uniformly on $z = E+\ii \eta \in \caD_{\tau}$, for any arbitrary deterministic diagonal matrix $B$ with $\norm B \norm = O(1)$. The same statement holds when replacing $\mathsf{P}$ by $P$ and $\wt{\MM}$ by $\wh{M}$. \end{proposition}

\begin{proposition}\label{prop:m1} Recall the assumption on $H=W+\lambda V$ given in Theorem \ref{thm:refine}. Fix appropriate $\tau$.
	Assume that $\pair{G} - \pair{\wt{\MM}} = O_{\prec} (\vt)$ on $\caD_{\tau}$, for some deterministic {$\vt \in [q^{-2}, N^{\tau/10}]$}. Then, on the event $\Xi$, we have 
	\begin{align}
		|\pair{B\mathsf{P}}| \prec \left(\frac{1}{q} + \sqrt{\frac{\Im \pair{{\wt{\MM} }}+\vt  }{N \eta }} \right)^2,
	\end{align}
	uniformly on $z = E+\ii \eta \in \caD_{\tau}$, for an arbitrary deterministic diagonal matrix $B$ with $\norm B \norm = O(1)$.
\end{proposition}

\revision{To prove a refined version of the local law, we utilize the bound of difference between $B$ and $(\lambda V - z - \pair{G})^{-1}$; see also Lemma \ref{lem:I11}.}

\begin{proposition}\label{prop:main}
	Recall the assumption on $H=W+\lambda V$ given in Theorem \ref{thm:refine} \revision{Assume that Theorem \ref{thm:strong} and Lemma \ref{lem:61} hold. Fix appropriate $\tau$.}
	Assume that $\pair{G} - \pair{\wt{\MM}} = O_{\prec} (\vt)$ on $\caD_{\tau}$, for some deterministic \revision{$\vt \in [q^{-2}, N^{\tau/10}]$}. Consider arbitrary deterministic diagonal matrix $B$ with $\norm B \norm = O(1)$. Furthermore assume that $\sup_j|B_{jj} - (\lambda V_{j} - z - \pair{G})^{-1}| \prec \chi$ for some $\chi \in [q^{-2} , N^{\tau/10}]$. Then, on the event $\Xi$, we have 
    \revision{
	\begin{align}
		|\pair{B\mathsf{P}}| \prec \Psi_{3/2}^{2} + \Psi_2 \left(\sqrt{\frac{|1- \pair{\wt{\MM} \odot \wt{\MM}}|+ \chi }{N\eta}  }  \right)  ,
	\end{align}}
	uniformly on $z = E+\ii \eta \in \caD_{\tau}$.
\end{proposition}
\revision{In fact, we insert $B=(B_{i}) = (\lambda V_i - z- \pair{G})^{-1}$, or its variants, and so the difference is small enough. The lattice argument that justifies replacing $B$ by a random term is given in Lemma \ref{lem:gminusrg}, which sets us ready to proceed to the stability analysis.}

Proposition \ref{prop:weakrme} will be proved in Section \ref{sec:weak}. Proposition \ref{prop:m1} and Proposition \ref{prop:main} will be proved along each other in Section \ref{sec:rme}. Notice that in the statement of Proposition \ref{prop:main}, we have assumed Theorem \ref{thm:strong}, which can be proved using Proposition \ref{prop:m1}. The assumption is used to bound $|\pair{G \odot G} - \pair{\wt{\MM}\odot \wt{\MM}}|$. We believe that the assumption can be removed by suitably adapting the methodologies developed for local laws of deformed Wigner matrices \cite{CEHK}. 

Using these self-consistent equations, we establish the bound for $|\pair{G} - \pair{\wt{\MM}}|$ by analyzing the stability. First, by formalizing the heuristic outlined in \eqref{eq:ref0}, we find that $\wt{\MM}$ is a good approximation for $G$, solving $\mathsf{P}(z, G) \approx 0$. Then, we obtain the desired local laws by applying a standard continuity arguments established in \cite{EKYY1, JL}.

We also face a similar technical issue in this step, where an estimate of $|\pair{G \odot G} - \pair{\wt{\MM} \odot \wt{\MM}}|$ is required to obtain a refined fluctuation. We again bypass the issue by utilizing Theorem \ref{thm:strong}. However, this approach also contributes to a weaker refined local law, being controlled by a factor $q^{-3/2}$. We believe that this issue can also be reduced as mentioned previously. Yet, noticing that the same restriction is already given in the stochastic step, we do not attempt to eliminate this technical issue.



To compress the notations, we introduce the control parameters to be used throughout the paper. Assume that \revision{$\theta, \vartheta \in [q^{-2}, N^{\tau/10}]$} are parameters that also satisfy $\max_{ij} |G_{ij} - \delta_{ij}\wh{M}_i|\prec \theta$ and $|\pair{G} - \pair{\wt{\MM}}|\prec \vartheta$. Then, define $\psi(z), \Psi_{\frb}(z)$ through
\begin{align}
	\psi \equiv \psi(z) \deq \frac{1}{q} + \sqrt{\frac{\Im \norm \wh{M} \norm + \theta  }{N\eta}}, \quad \Psi_{\frb} \equiv \Psi_{\frb}(z) \deq \frac{1}{q^{\frb}} + \frac{1}{N^{1/4} q^{1/2}} + \sqrt{\frac{\Im \langle \wt{\MM} \rangle + \vartheta }{N \eta}},
\end{align}
where $\frb$ is a positive real number that we typically choose to be greater or equal to $1$. Note that $\Psi_{\frb_1} \prec \Psi_{\frb_2}$ for $\frb_1 \geq \frb_2$. Also, when $\frb = 2$, $q^{- \frb} \ll N^{-1/4} q^{-1/2}$ holds if and only if $q \gg N^{1/6} $, and when $\frb  = 3/2$, the previous relation holds if and only if $q \gg N^{1/4}$.

We remark that the assumption $N^{\tau} < N^{\frac{20\delta}{9}} \leq q^{20/9}$ ensures 
\begin{align}
	\psi \leq \check{\psi} \deq  (1+ \theta)^3  \bigO{\frac{1}{q} + \sqrt{\frac{1 + \theta}{N \eta}} } <  \bigO{N^{-3 \tau/20}},
\end{align}
for all $z \in \caD_{\tau}$. Hence, the control parameters satisfy $\check{\psi} \leq 1$ and $\psi \geq 1/q$ uniformly on the domain.

Finally, we recall a lemma that effectively improves the high probability bounds iteratively, also called ``self improving high probability bounds''. For proof, we refer to \cite[Lemma 2.6]{HKR}.
\begin{lemma}\label{lem:self}
	Consider a constant $C>0$, $X \geq 0$ and $Y \in [N^{-C}, N^{C}]$. Suppose that there exists a constant $c\in [0, 1)$ such that for any $Z \in [Y, N^{C}]$, we have the implication
	\begin{align}
		X \prec Z \quad  \Longrightarrow \quad X \prec Z^{c} Y^{1-{c}},
	\end{align}
	Then we have $X \prec Y$ provided that $X \prec N^{C}$.
\end{lemma} 

\section{Primary estimates} \label{sec:primary}

\subsection{Weak estimates}\label{sec:weak} In this section, we prove the weak local law, following a standard method. See also \cite{HKR,He1,He2} for the proofs to weak local law by obtaining the self-consistent equation through cumulant expansion or integration by parts formula.
\begin{lemma}[Weak local law] \label{thm:weak}
	Recall the assumption on $H= W +\lambda V$ in Theorem \ref{thm:refine}. Fixing appropriate $\tau$, on the event $\Xi$, we have
	\begin{align}
		\max_{ij} |G_{ij} (z) - \delta_{ij} \wh{M}_i(z)| \prec \frac{1}{q^{1/3}} + \frac{1}{(N\eta)^{1/6}}, 
	\end{align}
	uniformly for $z \in \mc{D}_{\tau}$.
\end{lemma}
The proof follows by applying a standard continuity argument, for example, \cite{BEKYY, KY, HKR}, to the following self-consistent estimate.
\begin{lemma} \label{lem:cumulantweak}
	Assume that $
		\max_{ij} |G_{ij} - \delta_{ij} \wh{M}_i  | \prec \theta$ for some deterministic \revision{$\theta \in [q^{-2}, N^{\tau/10}]$.} Consider an arbitrary \revision{diagonal matrix $B$, where each entries, which we denote by $B_i$,} are of order $O(1)$. Then, on the event $\Xi$, we have
	\begin{align}
		\max_{ij} |B_i P_{ij}| \prec (1+ \theta)^2 \psi,
	\end{align}
	uniformly for $z \in \mc{D}_{\tau}$.
\end{lemma}

Proof of Lemma \ref{lem:cumulantweak} follows from the cumulant expansion methods in the proof of \cite[Theorem 1.5]{HKR} for $q \sim \sqrt{N}$. \revision{In fact, it suffices to check the a priori bound (equation (3.3) of \cite{HKR}), being the only case where a bound slightly stronger than the Wigner case is required. Using the stability bound \eqref{eq:stab2}, we have}
\begin{align}\label{eq:theapriori}
\begin{split}
	\max_{ij} |P_{ij}| = \max_{ij} |(WG)_{ij} +\pair{G} G_{ij} | &= \max_{ij} |\delta_{ij} + zG_{ij} -\lambda V_i G_{ij} + \pair{G}G_{ij} | \prec  q(1+\theta)^2 \psi.
    \end{split}
\end{align}
\revision{When applying the cumulant expansion to $P$,} the factor of $q$ \revision{in the a priori bound} can be compensated by the powers of $q^{-1}$ in the cumulants of the entries of $W$. Then, by simply adjusting the moment conditions from the proof for the Wigner case, we retrieve the lemma. A slightly more involved expansion is necessary to control the remainder; refer to Appendix \ref{sec:remainder}.

\revision{We provide further details of the continuity argument analyzing stability in Appendix \ref{Appendix_text}. We further remark that an analogue of Lemma \ref{lem:cumulantweak} can be obtained upon replacing $\wh{M}$ by $\wt{\MM}$ and $P$ by $\mathsf{P}$, moreover, we prove an improved version of it, namely Proposition \ref{prop:weakrme}. Since analogous continuity argument applies in this case, we also have the weak local law upon replacing $\wh{M}$ by $\wt{\MM}$. }

The complete delocalization of eigenvectors, Corollary \ref{cor:deloc} follows immediately from Lemma \ref{thm:weak}. Refer to Theorem 2.16 of \cite{EKYY1} for proof. Using Corollary \ref{cor:deloc}, we can improve the estimates from Ward identity, thanks to the following lemma from \cite{HLY20}.
\begin{lemma}\label{lem:wardd} On the event $\Xi$, uniformly on the domain $z \in \caD_{\tau}$, we have
	\begin{align}
		\max_{i} \Im [G_{ii}(z)] \prec \Im [\pair{G(z)}].
	\end{align}
\end{lemma}

We are ready to prove Proposition \ref{prop:weakrme}.

\begin{proof}[Proof of Proposition \ref{prop:weakrme}] The result follows immediately from Lemma \ref{thm:weak}, Lemma \ref{lem:cumulantweak}, and two observations. Firstly, the parameters $ (1+\theta)$ and $ \psi$ in Lemma \ref{lem:cumulantweak} can be replaced with $1$ and $\Psi_1$, respectively, by the weak local law and Lemma \ref{lem:wardd}. Secondly, we have
	\begin{align}
		|P_{ij} - \sfP_{ij}| = \frac{s}{q^2} |\pair{G \odot G } G_{ii} G_{ij} | \prec q^{-2},
	\end{align}
	uniformly on $z \in \caD_{\tau}$, and for any $i, j \in \llbracket 1, N \rrbracket$.
\end{proof}

\subsection{Definition and bounds on $Q$} \label{sec:Q}
Throughout this section, we assume the event $\Xi$ and all statements are on the event $\Xi$ even when it is not stated explicitly.

Consider the partial derivative $\pd_{ik} \caP = \pd_{ik} \pair{B \sfP}$, whose expansion yields
\begin{align}
	\begin{split}
		(1+ \delta_{ik})\pd_{ik} \caP =& \frac{1}{N} \G_{ki} + \frac{1}{N} \G_{ik} - \frac{1}{N} (\G WG)_{ki} -\frac{1}{N} (\G WG)_{ik} \\
		&- \frac{1}{N}\pair{G} ((\G G)_{ki} + (\G G)_{ik}) 
		- \frac{1}{N}\pair{\G} ((G^2)_{ki} + (G^2)_{ik}) \\&
		-\frac{2s}{q^2} \frac{1}{N} \sum_{j=1}^{N} \left[ \pair{G \odot G} \G_{jj} (G_{kj}G_{ji} + G_{ij}G_{jk} )+ \pair{\G \odot G} G_{jj} (G_{kj}G_{ji} + G_{ij} G_{jk}) \right],
	\end{split}
\end{align}
where we define
\begin{align}
	\begin{split}
		Q_{ik} = &\frac{1}{N} (\G W G)_{ik} + \frac{1}{N} \pair{G} (\G G)_{ik} + \frac{1}{N} \pair{\G} (G^2)_{ik}\\& + \frac{2s}{q^2} \pair{G \odot G} \frac{1}{N} \sum_{j=1}^{N} \G_{jj} G_{ij} G_{jk} + \frac{2s}{q^2} \pair{\G \odot G} \frac{1}{N} \sum_{j=1}^{N} G_{jj}G_{ij}G_{jk}.
	\end{split}
\end{align}
Then, we have the relation
\begin{align}
	(1+ \delta_{ik})\pd_{ik} \mc{P} = \frac{1}{N} \G_{ki} + \frac{1}{N} \G_{ik} - Q_{ki} - Q_{ik}. 
\end{align}
To this end, we can consider $Q$ as an auxiliary $N \times N$ matrix whose entries are defined above.

The quantity $Q$ was first introduced in \cite{HKR}, where it was used to prove local law for deformed Wigner matrices. Since $\mc{P}$ in our paper deals with higher order expansions, $Q$ contains terms inherited from them. We also remark that while we have only considered diagonal $B$ and diagonal $V$, the proceeding literature considered $B$ and $V$ in a more general setting where their operator norm being $O(1)$. We believe that our results can be extend to their generality, but do not pursue in that direction.

We now turn to obtaining estimates on $Q$. The following lemma is the sparse version of \cite[Lemma 3.6(ii)]{HKR}, assuming the weak local law.
\begin{lemma}\label{lem:QHKR} Fix (small) $\tau>0$. For any $x, y \in \llbracket 1, N \rrbracket$ and any $l \geq 1$, we have
	\begin{align} \label{eq:QHKR}
		Q = O_{\prec} (\Psi_1^3), \quad \pd_{xy}^{l} Q = O_{\prec} (\Psi_1^2),
	\end{align}
	for all $z \in \caD_{\tau}$, for $N$ sufficiently large. 
\end{lemma} 

The lemma can be easily obtained by altering the moment conditions throughout the proof of \cite[Lemma 3.6(ii)]{HKR}. Furthermore, its proof is roughly contained in the contents presented below, with more general and refined results. While Lemma \ref{lem:QHKR} is sufficient to prove an optimal local law without refinement to the measure, and the entrywise local law, it is not sufficient for our purposes. Through a more careful analysis, we have the following lemma.

\begin{lemma} Fix (small) $\tau$. For any $x, y \in \llbracket 1, N \rrbracket$ and integer $l \geq 1$, we have
	\begin{align*}
		Q &= O_{\prec} (N^{-1}q^{-1} + q^{-1}\Psi_2^2 +\Psi_2^3), \\
		\pd_{xy}^{l} Q &= O_{\prec } (N^{-1} + q^{-1}\Psi_2^2 +\Psi_2^3),
	\end{align*}\label{lem:optQ}
	for all $z \in \caD_{\tau}$.
\end{lemma}
The rest of the section is devoted to the proof of the lemma. Before heading into the proofs, we outline the strategy of proof. \revision{We first obtain a priori estimates using Proposition \ref{prop:weakrme}.} We obtain stronger bounds on $Q$, using cumulant expansion methods, and so the estimate on the derivatives of $Q$ play heavy role in the improvements. 

To estimate the higher order derivatives of $Q$, we observe that the first derivatives on $Q$ can be organized as follows. For any $x, y \in \llbracket 1, N \rrbracket$, we have
\begin{align}
	\begin{split}\label{eq:pd1}
		(1+\delta_{xy})\pd_{xy} Q_{ik} =& \frac{1}{N} \G_{ix} G_{ky} + \frac{1}{N} \G_{iy} G_{kx} - G_{xi} Q_{yk} - G_{xk} Q_{iy}  - G_{yi}Q_{kx} - G_{yk} Q_{ix} \\
		& - \frac{2}{N^2} (G^2)_{xy} (\G G)_{ik} - \frac{1}{N^2} [(\G G)_{xy} + (\G G)_{yx}] (G^2)_{ik}\\
		& - \frac{2s}{q^2} \left[ (\pd_{xy} \pair{G \odot G}) \frac{1}{N} \sum_{j=1}^{N} \G_{jj}G_{ij} G_{jk} + \pair{G \odot G} \frac{1}{N} \sum_{j=1}^{N} (\pd_{xy}\G_{jj}) G_{ij}G_{jk} \right]\\
		& - \frac{2s}{q^2} \left[ (\pd_{xy} \pair{\G \odot G}) \frac{1}{N} \sum_{j=1}^{N} G_{jj}G_{ij} G_{jk} + \pair{\G \odot G} \frac{1}{N} \sum_{j=1}^{N} (\pd_{xy} G_{jj})  G_{ij} G_{jk} \right],
	\end{split}
\end{align}
where we will abbreviate the last three lines by $E_{ik}^{xy}$. We can easily find \begin{align}
	\pd_{xy}^{m} E_{ik}^{xy} =O_{\prec} (\Psi_2^3),\label{eq:3line}
\end{align}
for any integer $m\geq 0$. The terms in the first line are important. Using induction, we realize that establishing bounds on $Q$ improves the bounds of their derivatives. Such interdependence allows the bounds on $Q$ and its derivatives to iteratively strengthen each other. 

To this extent, the key task is determining the boundaries of the improvement. Comparing the results in Lemma \ref{lem:QHKR} and Lemma \ref{lem:optQ}, we observe that either result contains $q^{-1} \Psi^2 + \Psi^3$ on the right hand side, where these terms could not be improved further in our exposition; refer to the treatments of the lower order expansions. On the other hand, the term $q^{-1} \Psi$ can gain extra factor of $q^{-1}$ up to a certain point. The effects of such refinement to the local law will be illustrated in the following sections.

We start by obtaining the preliminary bounds on $Q_{ik}$.
\begin{lemma}\label{lem:pdQ} For any fixed integer $l \geq 0$, for any $x, y \in \llbracket 1, N \rrbracket$, for all $z \in \caD_{\tau}$, we have
	\begin{align*}
		\pd_{xy}^{l} Q = O_{\prec} (N^{-1} +  q^{-1}\Psi_2 + \Psi_{2}^{2}).
	\end{align*}
	Furthermore, assume that $Q = O_{\prec} (N^{-1}q^{-1} + q^{-b} \Psi_2 + q^{-1} \Psi_2^2 + \Psi_2^3)$ (for $b \in [1, 3])$, then we have
	\begin{align*}
		\pd_{xy}^{l} Q = O_{\prec} (N^{-1} +  q^{-b}\Psi_2 + q^{-1}\Psi_{2}^{2} + \Psi_2^3 ),
	\end{align*}
	for all $z \in \caD_{\tau}$.
\end{lemma}
\begin{proof}
	For the first claim, we proceed by induction. For $l=0$, we have 
	\begin{align*}
		|Q_{ik}| &\leq  \left|\frac{1}{N}\sum_{j} \G_{ij}[(WG)_{jk} +\pair{G} G_{jk}]  \right| + \left| \frac{1}{N} \pair{\G} (G^2)_{ik} \right|
        + \left| \frac{2s}{Nq^2}  \sum_{j=1}^{N}(G_{ij}G_{jk})\left(\pair{G \odot G} \G_{jj}  + \pair{\G \odot G}  G_{jj} \right) \right|\\
		& \prec \Psi_1 \Psi_2+ \Psi_2^2 ,
	\end{align*}
	where we have used Ward identity and Proposition \ref{prop:weakrme}.
	
	Next, assume that the claim is true for any set of indices for all derivatives of order $0 \leq r \leq l-1$. From \eqref{eq:pd1}, we have
	\begin{align*}
		\pd_{xy}^{l} Q_{ik} =  \frac{1}{1+ \delta_{xy}} \pd_{xy}^{l-1} \left( \frac{1}{N} \G_{ix} G_{ky} +\frac{1}{N} \G_{iy} G_{kx}  - G_{xi} Q_{yk} - G_{xk} Q_{iy}  - G_{yi}Q_{kx} - G_{yk} Q_{ix}\right) + O_{\prec}(\Psii^3),
	\end{align*}
	where with the weak local law, the first two terms are of order $O_{\prec} (N^{-1})$. The succeeding four terms are of the same form, and so we only illustrate the treatment of the first. We have
	\begin{align*}
		\pd_{xy}^{l-1} (G_{xi}Q_{yk}) = \sum_{m=0}^{l-1} \binom{l-1}{m} (\pd_{xy}^{l-m} G_{xi}) (\pd_{xy}^{m} Q_{yk}) = O_{\prec} (N^{-1} + q^{-1}\Psi_2 + \Psi_2^2),
	\end{align*}
	by the induction hypothesis, and the claim is proved. 
	
	For the second claim, note that we have placed $N^{-1}$ in the control parameter due to the first two terms in the expansion of the first derivative. Proceeding using induction as in the previous claim, we obtain the desired result.
\end{proof}

Now we turn to cumulant expansion, improving the previous bounds using non-trivial cancellations. We start by writing
\begin{align*}
	\E |Q_{ik}|^{2D} =& \E \left[\frac{1}{N} \sum_{j_1 j_2} W_{j_1 j_2} G_{j_2 k} \G_{i j_1} Q_{ik}^{D-1} \overline{Q_{ik}^D} \right] + \E \left[ \Bigg(\frac{1}{N} \pair{G} (\G G)_{ik} + \frac{1}{N} \pair{\G} (G^2)_{ik} \Bigg) Q_{ik}^{D-1} \overline{Q_{ik}^D} \right]\\
	&+ \E \left[ \frac{2s}{q^2} \Bigg( \pair{G \odot G} \frac{1}{N} \sum_{j=1}^{N} \G_{jj} G_{ij} G_{jk}  + \pair{\G \odot G} \frac{1}{N} \sum_{j=1}^{N} G_{jj}G_{ij}G_{jk} \Bigg) Q_{ik}^{D-1} \overline{Q_{ik}^D} \right],
\end{align*}
for any (large) positive integer $D$. The terms in the last line are easily bounded by $O_{\prec} (q^{-2}\Psi_2^2  \E|Q_{ik}|^{2D-1})$. Applying \eqref{eq:genStein} to the first term, we obtain
\begin{align*}
	\E \left[\frac{1}{N} \sum_{j_1 j_2} W_{j_1 j_2} G_{j_2 k} \G_{i j_1} Q_{ik}^{D-1} \overline{Q_{ik}^D} \right] &= \sum_{r=1}^{\ell} \frac{\mc{C}_{r+1}}{r!} \E \left[\frac{1}{N} \sum_{j_1 j_2} \pd_{j_1 j_2}^r \left(  G_{j_2 k} \G_{i j_1} Q_{ik}^{D-1} \overline{Q_{ik}^D} \right)  \right] + \E \caR_{\ell}(Y),
    \\
	\caR_{\ell}(Y) &\deq \caR_{\ell} \left( \frac{1}{N}\sum_{j_1 j_2} W_{j_1 j_2} G_{j_2 k} \G_{ij_1} Q_{ik}^{D-1} \overline{Q_{ik}^D}\right),
\end{align*}
for any $D>0$, To abbreviate, we set
\begin{align}
	 Y_{r,s }  &\deq  \frac{\caC_{r+1}}{N}  \sum_{j_1 j_2} \pd_{j_1 j_2}^{r-s} \left(  G_{j_2 k} \G_{i j_1} \right) \pd_{j_1 j_2} ^{s} \left( Q_{ik}^{D-1} \overline{Q_{ik}^D} \right),\\ \quad Y_r & \deq \sum_{s=0}^{r} \binom{r}{s} Y_{r,s} =  \frac{\caC_{r+1}}{N} \sum_{j_1 j_2} \pd_{j_1 j_2}^r \left(  G_{j_2 k} \G_{i j_1} Q_{ik}^{D-1} \overline{Q_{ik}^D} \right).
\end{align}
With these notations, we may rewrite the cumulant expansion as
\begin{align}
	\E \left[\frac{1}{N} \sum_{j_1 j_2} W_{j_1 j_2} G_{j_2 k} \G_{i j_1} Q_{ik}^{D-1} \overline{Q_{ik}^D} \right]  &= \sum_{r=1}^{\ell} w_{r,s} \E Y_{r.s} + \E \caR_{\ell} (Y)
\end{align}
where $w_{r,s} \deq \binom{r}{s}/r!$. Now, we present the key lemma that underpins Lemma \ref{lem:optQ}.

\begin{lemma} \label{lem:Q} Assume that $Q = O_{\prec}(N^{-1} + q^{-b} \Psi_2 + \Psi_2^2 )$. Assume further that $\xi$ is a deterministic parameter satisfying $ Q = O_{\prec} (\xi)$ and $\xi \geq N^{-1} q^{-1} + q^{-(b+1)} \Psi_2 + q^{-1} \Psi_2^2 + \Psi_2^3 $. Then the following statements hold for all $z \in \mc{D}_{\tau}$.
	\begin{enumerate}
		\item For $r=1$, we have
		\begin{align}
			\left|\E Y_{1, 0} + \E \left[ \Bigg(\frac{1}{N} \pair{G} (\G G)_{ik} + \frac{1}{N} \pair{\G} (G^2)_{ik} \Bigg) Q_{ik}^{D-1} \overline{Q_{ik}^D} \right] \right| \prec \Psi_2^{3} \E|Q_{ik}|^{2D-1} ,\\
			|\E Y_{1, 1}| \prec  \left( \Psi_2^3 \xi + \Psi_2^6 + q^{-2} \Psi_2^4 \right) \E |Q_{ik}|^{2D-2}.
		\end{align}
		\item For $r \geq 2$, we have
		\begin{align}
			|\E Y_r | \prec \sum_{a=0}^{\min\{r, 2D\}} [(N^{-1} q^{-1} + q^{-(b+1)} \Psi_2 + q^{-1} \Psi_2^3) \xi ]^{a/2} \E |Q_{ik}|^{2D-a}
		\end{align}
		\item For any $D>0$, there exist $\ell \equiv \ell(D)$ such that $ \caR_{\ell}(Y) \prec q^{-\ell}$.
	\end{enumerate}
	Finally, combining (1), (2), (3), we have 
	\begin{align}
		Q= O_{\prec} (N^{-1}q^{-1} + q^{-(b+1)} \Psi_2 + q^{-1}\Psi_2^2 +\Psi_2^3),\label{eq:Qres}
	\end{align}
	for all $z\in \mc{D}_{\tau}$.
\end{lemma}

We postpone the proofs of (1), (2), (3) to the Appendix.
\begin{proof}[Proof of \eqref{eq:Qres}]
	
	From the results of (1), (2), (3), choosing $\ell >10D$ we have
	\begin{align}
		\E |Q_{ik}|^{2D}  \prec & \sum_{a'=1}^{2D} [(N^{-1}q^{-1} + q^{-(b+1)}\Psi_2 +q^{-1} \Psi_2^2 +\Psi_2^3) \xi]^{a'/2} \E |Q_{ik}|^{2D-a'} + q^{-10D} \\
		\prec & \sum_{a'=1}^{2D} [(N^{-1}q^{-1} + q^{-(b+1)}\Psi_2 +q^{-1} \Psi_2^2 +\Psi_2^3) \xi]^{a'/2} (\E |Q_{ik}|^{2D})^{(2D-a')/2D}.\label{eq:last}
	\end{align}
	From Young's inequality, we have $\E |Q_{ik}|^{2D} \prec [(N^{-1}q^{-1} + q^{-(b+1)}\Psi_2 +q^{-1} \Psi_2^2 +\Psi_2^3) \xi]^{2D/2}$ for large $N$, and since $D$ is arbitrary, from Markov's inequality, we deduce the implication
	\begin{align}
		Q = O_{\prec} (\xi) \quad \Longrightarrow \quad Q= O_{\prec} ([(N^{-1}q^{-1} + q^{-(b+1)}\Psi_2 +q^{-1} \Psi_2^2 +\Psi_2^3) \xi]^{1/2}).
	\end{align}
	Note that uniformity in the domain $z\in \caD_{\tau}$ can be obtained from a standard lattice argument. Starting with the a priori $ Q = O_{\prec}(N^{-1} +  q^{-b}\Psi_2 + \Psi_{2}^{2}) $ from assumption, and using the self-improving mechanism introduced in \cite{HKR}, we obtain the desired result.
\end{proof}

Finally, we justify Lemma \ref{lem:optQ}, by iteratively applying Lemma \ref{lem:pdQ} and Lemma \ref{lem:Q}. \revision{We note that only a few iterations of applying these lemmas are required, in fact, two iterations suffice in our case. }

\begin{proof}[\textit{Proof to Lemma \ref{lem:optQ}}]
	We start with $Q= O_{\prec} (N^{-1} + q^{-1}\Psi_2 + \Psi_2^2) $ from Lemma \ref{lem:pdQ}. Applying Lemma \ref{lem:Q}, we obtain $Q= O_{\prec} (N^{-1}q^{-1} + q^{-2} \Psi_2 + q^{-1}\Psi_2^2 + \Psi_2^3)$. Since this result satisfies the assumption of Lemma \ref{lem:Q} with $b=2$, its application yields the bound $Q = O_{\prec} (N^{-1}q^{-1} +q^{-3} \Psi_2  +q^{-1} \Psi_2^2 + \Psi_2^3)$. Using $q^{-2} \leq \Psi_2$, the result can be reduced into the desired form. Application of Lemma \ref{lem:pdQ} once more gives the second claim.
\end{proof}

\section{Truncated cumulant expansion of $\mathcal{P}$} \label{sec:rme}
\subsection{Preliminaries and results}
In this section, we prove Proposition \ref{prop:main}. All statements hold on the event $\Xi$, even if it is not mentioned explicitly. Applying the standard cumulant expansion method, we have
\begin{align}
	\E \left[ \frac{1}{N} \sum_{i=1}^{N} \sum_{k}^{(i)} W_{ik} \G_{ki} \mc{P}^{D-1} \overline{\mc{P} ^D}  \right] &= \sum_{r=1}^{\ell} \frac{\mc{C}_{r+1}}{r!} \E \left[\frac{1}{N} \sum_{i=1}^{N} \sum_{k}^{(i)} \pd_{ik}^{r} \left(\G_{ki} \mc{P}^{D-1} \overline{\mc{P}^D}  \right) \right] + \E \caR_{\ell} (I),\label{eq:corexp}\\
	\caR_{\ell}(I) &\deq \caR_{\ell} \left( \frac{1}{N} \sum_{i=1}^{N} \sum_{k}^{(i)} W_{ik} \G_{ki} \mc{P}^{D-1} \overline{\mc{P} ^D} \right),
\end{align}
for $D > 0$. We set
\begin{align}
	I_{r,s} &\deq \mc{C}_{r+1} \frac{1}{N} \sum_{i=1}^{N} \sum_{k}^{(i)} \left(\partial_{ik}^{r-s} \G_{ki} \right) \left(\partial_{ik}^{s}  (\mc{P}^{D-1} \overline{\mc{P}^{D}}) \right),\\ \quad I_r &\deq  \sum_{s=0}^{r} \binom{r}{s} I_{r,s}. =\mc{C}_{r+1} \frac{1}{N} \sum_{i=1}^{N} \sum_{k}^{(i)} \pd_{ik}^{r} \left(\G_{ki} \mc{P}^{D-1} \overline{\mc{P}^D} \right).
\end{align}
With these notations, we may rewrite the cumulant expansion \eqref{eq:corexp} as
\begin{align}
	\E \left[\frac{1}{N} \sum_{i=1}^{N} \sum_{k}^{(i)} W_{ik} \G_{ki} \mc{P}^{D-1} \overline{\mc{P} ^D}  \right] &= \sum_{r=1}^{\ell}\sum_{s=0}^{r} w_{r,s} \E I_{r.s} + \E \caR_{\ell} (I),
\end{align}
where $w_{r,s} \deq \binom{r}{s}/r!$.

First, we tackle $I_{1,0}$ and $I_{3,0}$, where the non negligible terms are exactly canceled out.
\begin{lemma}\label{lem:cancel} We have, for all $z \in \mathcal{D}_{\tau}$,
	\begin{align}
		\left| \E \left[ w_{1, 0} I_{1,0} + \pair{G} \pair{\G} \mc{P}^{D-1} \overline{\mc{P}^{D}} \right] \right|  &\prec \Psi_{2}^2 \E |\mc{P}|^{2D-1},\\
		\bigg|\E \bigg[ w_{3, 0} I_{3,0} + \frac{s}{q^2} \pair{G \odot G} \pair{\G \odot G} \mc{P}^{D-1} \overline{\mc{P}^D} \bigg]\bigg| & \prec \Psi_{2}^2 \E |\mc{P}|^{2D-1}.
	\end{align}
\end{lemma}
\begin{proof} The proof follows from direct computations and the Ward identity. Bounding $I_{1, 0}$ is simple,
	\begin{align}
		\left| \E \left[ w_{1, 0} I_{1,0} + \pair{G} \pair{\G} \mc{P}^{D-1} \overline{\mc{P}^{D}} \right] \right| \leq  \frac{1}{N^2} \E \left| \sum_{i}^{N} \sum_{k}^{(i)} \ G_{ik} \G_{ki} \mc{P}^{D-1} \overline{\mc{P}^D} \right | \prec \Psi_{2}^2 \E |\mc{P}|^{2D-1}.
	\end{align}
 Bounding $I_{3, 0}$ requires more calculation, where we have
	\begin{align}
			\bigg|\E &\bigg[ w_{3, 0} I_{3,0} + \frac{s}{q^2} \pair{G \odot G} \pair{\G \odot G} \mc{P}^{D-1} \overline{\mc{P}^D} \bigg]\bigg| \\ 	\begin{split}&\leq w_{3, 0} {\frac{s}{N^2q^2 }} \sum_{ik} \E \Bigg|\big[\pd_{ik} ( 2G_{ik}^2 \G_{ki} ) + \pd_{ik} (\G_{ii} G_{kk}) G_{ik} + \G_{ii} G_{kk} G_{ik}^2  + 3 \pd_{ik} (G_{ii} \G_{kk} ) G_{ik} + 3G_{ii} \G_{kk} G_{ik}^2  \\
			& \phantom{O(N^2 q^2[[[[[[[[]]]]]]]])} + (\pd_{ik} (G_{ii} G_{kk} ) + G_{ii} G_{kk} G_{ik} )(\G_{ik} + \G_{ki}) \big]  \mc{P}^{D-1} \overline{\mc{P}^D} \Bigg|	\end{split} \\
			&\prec \Psi_{2}^2 \E |\mc{P}|^{2D-1},
	\end{align}
	since all the terms inside the square bracket contain at least two off-diagonal entries containing both $i$ and $k$ in the indices.
\end{proof}

Next, we consider the higher order terms, which can be managed with a single cumulant expansion, with the aid of the estimates obtained in Section \ref{sec:Q}. 
\begin{lemma}\label{lem:high}
	Given any pair of integers $(r, s)$ with $r \geq 4, 0\leq s \leq r$, or $(r,s) = (3, 1), (3, 3),(2, 2)$, the following holds
	\begin{align}
		\E I_{r,s} = O_{\prec} \left(\sum_{a=0}^{s \wedge 2D-1} \Psi_2^{2(a+1)} \E |\mc{P}|^{2D-1-a} \right),
	\end{align}
	uniformly for all $z \in \mathcal{D}_{\tau}$. Moreover, for any $D>0$, there exist $\ell \equiv \ell(D)$ such that $ \caR_{\ell}(I) \prec q^{-\ell}$.
\end{lemma}

\revision{We note that $\ell \equiv \ell(D)$ is a positive integer independent of $N$ and may be chosen arbitrarily large. Likewise, in Lemmas \ref{lem:I11}-\ref{lem:I32} that follow, $\ell', \ell''$ are positive integers that may be chosen arbitrarily large for any fixed $D>0$ and are independent of $N$. Moreover, in what follows, summations should be understood in the sense that it runs only up to when the exponent of $|\caP|$ is non-negative, regardless of the choice of $\ell, \ell', \ell''$, as it would indicate that the term had vanished after iterative differentiations. }

The cases that have not yet been covered are $I_{1,1}$, $I_{2, 0}$, $I_{2,1}$, and $I_{3,2}$. \revision{Analysis of these terms are more intricate, where $I_{2, 0}$, and $I_{2,1}$ require successive cumulant expansions to obtain a finer estimate, while  $I_{1,1}$ is handled via a different approach to estimate $\caP$ rather than using the estimate on $Q$.} The term $I_{2,1}$ is especially challenging; an estimate using the control parameter $\Psi_{3/2}$ was proved after several successive cumulant expansions. The term $I_{3,2}$ is also controlled using $\Psi_{3/2}$, with a simpler proof without further cumulant expansion. The results are stated below.

\begin{lemma}\label{lem:I11} 
	We have, uniformly for all $z \in \caD_{\tau}$,
	\begin{align}
		|\E I_{1, 1}| \prec \Psi_1^4 \E |\caP|^{2D-2}.
	\end{align}
	\revision{Assume further that Theorem \ref{thm:strong} and Lemma \ref{lem:61} is true and recall that $\caP = B \sfP$, where $B$ is a deterministic diagonal matrix with $\norm B \norm = O(1)$. Recall that we had assumed $|\pair{G} - \pair{\wt{\MM}} |\prec \vartheta$ and further assume that $\sup_{j} \left| B_{j} - (\lambda V_j - z - \pair{G})^{-1}  \right| \prec \chi$, for some $\chi \geq q^{-2}$. Then, we have
	\begin{align}
		  |\E I_{1,1}|\prec \left( \Psi_2 ^4 + \Psi_2^2 \frac{|1 - \pair{\wt{\MM} \odot \wt{\MM} }|+\vt}{N\eta} +\chi \frac{\Im \pair{\wt{\MM}} +\vartheta}{(N \eta)^2} \right) \E |\caP|^{2D-2},
	\end{align}
	uniformly for all $z \in \caD_{\tau}$.}
\end{lemma}

\begin{lemma} \label{lem:I20} 
	We have, uniformly for all $z \in \caD_{\tau}$,
	\begin{align}
		|\E I_{2, 0}| {\prec} \sum_{r=1}^{\ell'} \sum_{a=0}^{r} \Psi_2^{2(a+1)} \E |\mc{P}|^{2D-1-a} + q^{-\ell'},
	\end{align}
	for appropriately chosen $\ell'$.
\end{lemma}
\begin{lemma}\label{lem:I21}
	We have, uniformly for all $z \in \caD_{\tau}$,
	\begin{align}
		|\E I_{2, 1}| \prec \sum_{r=0}^{\ell'} \sum_{a'=2}^{r+2} \Psi_{3/2}^{2a'} \E |\mc{P}|^{2D -a'} +\sum_{r=0}^{\ell''} \sum_{a''=2}^{r+3} \Psi_{3/2}^{2} \Psi_2^{2(a''-1)} \E|\mc{P}|^{2D-a''}+ q^{-\ell'} +q^{-\ell''},
	\end{align}
	for $\ell'$ chosen appropriately.
\end{lemma}
\begin{lemma}\label{lem:I32}
	We have
	\begin{align}
		|\E I_{3, 2}| {\prec} \Psi_{3/2}^4 \E|\mc{P}|^{2D-2},
	\end{align}
	uniformly for all $z \in \caD_{\tau}$.
\end{lemma}
We defer the proofs of Lemma \ref{lem:high} and Lemma \ref{lem:I32} to the Appendix, where the bounds are obtained without more successive cumulant expansions. To obtain the bounds on $I_{1, 1}$, $I_{2, 0}$, and $I_{2,1}$, further expansions are required. In Section \ref{sec:I11}, we obtain the estimates on $I_{1, 1}$ (Lemma \ref{lem:I11}), and in Section \ref{sec:I20}, we obtain the estimates on $I_{2,0}$ (Lemma \ref{lem:I20}). Finally, in Section \ref{sec:I21}, we obtain estimates on $I_{2,1}$ (Lemma \ref{lem:I21}).

Culminating these bounds, we are ready to prove Proposition \ref{prop:m1} and \ref{prop:main}.

\begin{proof}[Proof of Propositions \ref{prop:m1} and \ref{prop:main}]
	From Lemmas \ref{lem:cancel}--\ref{lem:I32}, by choosing $\ell, \ell', \ell'' >8D $, we deduce
	\begin{align}
		\E |\mc{P}|^{2D} &\prec \sum_{a=1}^{2D} \Psi_{1}^{2a} \E |\mc{P}|^{2D-a} + q^{-8D}.
	\end{align}
	From Young's inequality, we obtain  $\E |\mc{P}|^{2D} \prec (\Psi_{1}^2   )^{2D}$ for large $N$. Since $D$ is arbitrary,  from Markov's inequality, for a fixed $z \in \caD_{\tau}$, we obtain $|\mc{P} | {\prec}\Psi_{1}^{2} $. 
	
	Next, assuming Theorem \ref{thm:strong}, the estimate in Lemma \ref{lem:I11} is strengthened, and we have
\revision{\begin{align}
		\E |\mc{P}|^{2D} &\prec \sum_{a=1}^{2D} \Psi_{3/2}^{2a} \E |\mc{P}|^{2D-a} +\left( \Psi_2 ^4 + \Psi_2^2 \frac{|1 - \pair{\wt{\MM} \odot \wt{\MM} }|+\vt}{N\eta} +\chi \frac{\Im \pair{\wt{\MM}} +\vartheta}{(N \eta)^2} \right) \E |\caP|^{2D-2}+ q^{-8D}\\
        &\prec \sum_{a=1}^{2D} \Psi_{3/2}^{2a} (\E |\mc{P}|^{2D})^{(2D-a)/2D} + \left[ \Psi_2^4 + \Psi_2^2 \left(\sqrt{\frac{|1- \pair{\wt{\MM} \odot \wt{\MM}  }| + \chi}{N\eta}}  \right)^2 \right]\E |\caP|^{2D-2} + q^{-8D},
	\end{align}
    }
	\revision{where Young's inequality was used in the second line. Next, applying Young's inequality again, we obtain $\E |\mc{P}|^{2D} \prec [\Psi_{3/2}^2 + \Psi_2( (|1- \pair{\wt{\MM} \odot {\wt{\MM}}}|+\chi)/({N\eta})) ^{1/2} ]^{2D}$ for large $N$. } Since $D$ is arbitrary, we have \revision{
	\begin{align}
		|\mc{P} | \prec \Psi_{3/2}^{2} + \Psi_2 \left(\sqrt{\frac{|1- \pair{\wt{\MM} \odot \wt{\MM}}|+ \chi}{N\eta}  }  \right).
	\end{align}}
	Uniformity in both cases can be obtained by a standard ``lattice'' argument.
\end{proof}

Finally, we introduce a criterion that efficiently confirms whether the terms satisfy the desired estimates.
\begin{lemma}\label{lem:neglect}
	For positive numbers (most of the times being integers) $n_d, n_m, n_{p}, n_{q}$, the following holds
	\begin{align*}
		\frac{1}{q^{n_q}}  (N^{-1} +q^{-1} \Psi_2^2 + \Psi_2^3)^{n_d}  \Psi_2^{n_{p}}\prec \Psi_2^{2n_m} ,
	\end{align*}
	if 
	\begin{align*}
		{\frf} \equiv \frf(n_q, n_d, n_p, n_m) \deq \min\{n_q + 2n_d, 4n_d, (n_q + 5n_d)/2 \} +n_p -2n_m  \geq 0.
	\end{align*} 
	We will refer to $\frf$ as \emph{power index}.
\end{lemma}
\begin{proof}
	The proof is a straightforward application of the properties of stochastic domination.
\end{proof}

\subsection{Estimates on $I_{1, 1}$} \label{sec:I11}
\begin{proof}[Proof of Lemma \ref{lem:I11}] We begin by expanding
	\begin{align}
    \begin{split}
		\E I_{1,1} &= \bigO{\frac{1}{N^2}} \E \left[\sum_{i} \sum_{k}^{(i)} \G_{ki} (\pd_{ik} \mc{P}) \mc{P}^{2D-2} \right] = \bigO{\frac{1}{N^2}} \E \left[ \sum_{i} \sum_{k} \G_{ki} \left(\frac{1}{N} \G_{ki} + \frac{1}{N} \G_{ik} - Q_{ki} - Q_{ik} \right) \mc{P}^{2D-2} \right]\\
		&=  \bigO{\frac{1}{N^2}} \E \left[ \sum_{i} \sum_{k} \G_{ki} \left( - Q_{ki} - Q_{ik} \right) \mc{P}^{2D-2} \right] + O_{\prec} (N^{-1} \Psi^2 \E |\mc{P}|^{2D-2}),\label{eq:I11first}
        \end{split}
	\end{align}
	where one can easily find the first claim by using Lemma \ref{lem:optQ}.
	Now, let us assume that Theorem \ref{thm:strong} holds. \revision{To obtain a refined estimate, we need a different method to cope with $\pd_{ik} \caP$. Explicitly writing down $\pd_{ik} \pair{B \sfP}$, for arbitrary $i, k \in \llbracket 1, N \rrbracket$, we have}{
    \begin{align}
        \pd_{ik} \pair{B \sfP}  =& \pd_{ik} \left[ \frac{1}{N} \sum_{j} B_{j} \left(1 + z G_{jj} + \pair{G} G_{jj}  - \lambda V_j G_{jj} + \frac{s}{q^2} \pair{G \odot G} G_{jj}^2 \right) \right]\notag \\
        \begin{split}
        =& \frac{1}{N} \sum_{j} \left[\frac{1}{\lambda V_j -z - \pair{G}} \left( z + \pair{G} - \lambda V_j + \frac{s}{q^2} \pair{G \odot G} G_{jj}  \right) +1 \right] \cdot \pd_{ik} G_{jj} -\frac{1}{N} \sum_j \pd_{ik} G_{jj}\\
        & - \frac{1}{N} \sum_j \left[ \left(\frac{1}{\lambda V_j -z - \pair{G}} - B_{j} \right) \left( z + \pair{G} - \lambda V_j + \frac{s}{q^2} \pair{G \odot G} G_{jj} \right) \right] \cdot \pd_{ik} G_{jj} \\
        & + \frac{1}{N} \sum_n \left[\frac{1}{\lambda V_n -z - \pair{G}} \left( G_{nn} \pd_{ik} \pair{G} + \frac{s}{q^2} G_{nn} \pd_{ik} \left( \pair{G \odot G} G_{nn} \right)\right) \right] \\
        &- \frac{1}{N} \sum_{j} \left[ \left(\frac{1}{\lambda V_j -z - \pair{G}} -B_{j} \right) \left( G_{jj} \pd_{ik} \pair{G} + \frac{s}{q^2} G_{jj} \pd_{ik} \left( \pair{G \odot G} G_{jj} \right)\right) \right] 
        \end{split}\notag\\
        \begin{split}
        =& \frac{1}{N} \left[ \sum_{j} \frac{s}{q^2} \pair{G \odot G} G_{jj} \frac{1}{\lambda V_j - z - \pair{G}} \right] \cdot \pd_{ik} G_{jj} \\
        &+ \frac{1}{N} \sum_{i} \pd_{ik} G_{jj} \left( -1+ \frac{1}{N} \sum_{n} {G_{nn} \frac{1}{\lambda V_n - z- \pair{G}}} \right) + \frac{1}{N} \sum_{n} \frac{s}{q^2} \frac{ G_{nn}}{\lambda V_n -z - \pair{G}} \pd_{ik} (\pair{G \odot G} G_{nn}) \\
        &- \frac{1}{N} \sum_j \left[\left( \frac{1}{\lambda V_j - z- \pair{G}} - B_{j} \right) \left( z + \pair{G} - \lambda V_j + \frac{s}{q^2} \pair{G \odot G} G_{jj} \right) \right] \cdot \pd_{ik} G_{jj} \\
        &- \frac{1}{N} \sum_{j} \left[ \left( \frac{1}{\lambda V_j - z- \pair{G}} - B_{j} \right) \left( G_{jj} \pd_{ik} \pair{G} + \frac{s}{q^2} G_{jj} \pd_{ik} \left( \pair{G \odot G} G_{jj} \right)\right) \right] ,\label{eq:xylong}
        \end{split}
    \end{align} 
    where we denote each term in \eqref{eq:xylong} by $Z_{xy}^{(m)}$, with $m=1, 2, 3, 4, 5$.}


    \revision{We then turn to estimating each term. We will often use that $(\lambda V_i - z- \pair{G})^{-1} = O_{\prec} (1) $, which can be obtained by combining the weak local law and stability bounds; see also proof of Lemma \ref{lem:61}. The first term and the third term are handled in a similar manner, where we have}
    \revision{
    \begin{align}
    \left| \E \left[ \frac{1}{N^2} \sum_{ik} \G_{ki} Z_{ik}^{(1)} \caP^{2D-2} \right] \right| =& \left|\E \left[ \frac{1}{N^3} \sum_{ijk} \G_{ki}\frac{s}{q^2} \pair{G \odot G} G_{jj} \frac{1}{\lambda V_j - z - \pair{G}} (-2G_{ij }  G_{jk}) \right] \caP^{2D-2} \right| \\ \prec &q^{-2} \Psi_2^3 \E |\caP|^{2D-2} 
    \end{align}
    \begin{align}
      \left| \E \left[ \frac{1}{N^2} \sum_{ik} \G_{ki} Z_{ik}^{(3)} \caP^{2D-2} \right] \right| =& \E \left[ \frac{1}{N^3} \sum_{ijk}\G_{ki}  \frac{1}{N} \sum_{n} \frac{s}{q^2} \frac{ G_{nn}}{\lambda V_n -z - \pair{G}} \pd_{ik} (\pair{G \odot G} G_{nn})  \caP^{2D-2}  \right]  \\ \prec &q^{-2} \Psi_2^3 \E |\caP|^{2D-2}.
    \end{align}
    The second term require little more work
    \begin{align}
    \begin{split}
    \left| \E \left[ \frac{1}{N^2} \sum_{ik} \G_{ki} Z_{ik}^{(2)} \caP^{2D-2} \right] \right| =& \left| \E \left[ \frac{1}{N^3} \sum_{ijk}\G_{ki}  (-2G_{ij }  G_{jk}) \left( -1+ \frac{1}{N} \sum_{n} {G_{nn} \frac{1}{\lambda V_n - z- \pair{G}}} \right) \caP^{2D-2}\right] \right| \\
     \prec& \E \left[  \left| \frac{1}{N^3}\sum_{ijk}\G_{ki}  (-2G_{ij }  G_{jk}) \right|  \left|1- \frac{1}{N}\sum_n \frac{1}{(\lambda V_n - z - \pair{G} )^2 } \right| |\caP|^{2D-2} \right]\\
    &+ \E \left[ \left| \frac{1}{N^3}\sum_{ijk}\G_{ki}  (-2G_{ij }  G_{jk}) \right| \right.\\
    & \phantom{.............}\times \left. \left| \frac{1}{N} \sum_n \left(G_{nn} -\frac{1}{\lambda V_n - z - \pair{G} }\right)  \frac{1}{\lambda V_n - z - \pair{G}}  \right||\caP|^{2D-2} \right]\\
    \prec& \frac{\Im \pair{\wt{\MM}} + \vartheta }{(N\eta)^2} \E \left[\left|1- \frac{1}{N} \sum_{n=1} \frac{1}{(\lambda V_n - z - \pair{G})^2} \right||\caP|^{2D-2} \right] + \Psi_2^3 \Psi_1^2 \E|\caP|^{2D-2}\\
     \prec & \Psi_2^4\E |\caP|^{2D-2} + (|1 - \pair{\wt{\MM } \odot \wt{\MM} }|+ \vartheta )\frac{\Im \pair{\wt{\MM}} + \vartheta} {(N\eta)^2} \E | \caP|^{2D-2},\label{eq:revZ2}
    \end{split}
    \end{align}
    where we have used the second statement of Lemma \ref{lem:61} and spectral decomposition; see also (3.14) in \cite{LS18}. The last step also involves bounding $(\lambda V_n - z- \pair{G})^{-1}$ The other two terms, which involve $B$, can be bounded in the same bound. Bounding $m=4$ is slightly more delicate, where we illustrate the process as follows. Consider a quantity $A= (A_j)$ such that $\sup_j |A_j| \prec \chi$, then consider
\begin{align*}
        \frac{1}{N^3} \left| \sum_{ijk} G_{ij} G_{jk} G_{ki} A_j B_i \right|&\leq \sup_r |A_r| \sup_j \left|\frac{1}{N^2} \sum_{ik} G_{ij} G_{jk}G_{ki} B_i  \right| \\
        &\prec \chi \frac{1}{N\eta} \sup_j\left[ \left(\frac{1}{N} \sum_{i} |B_i G_{ij}|^2 \right)^{1/2} \right] \sup_{j'} \left[\left(\frac{1}{N} \sum_{k} |G_{j'k}|^2 \right)^{1/2} \right],
    \end{align*}
    since the $L_2$ norm of $G$ is bounded by $1/\eta$. (See also proof of Proposition 2.14 in \cite{HY}.) Then, identifying appropriate $A$, and using the Ward identity along with Lemma \ref{lem:wardd},  we obtain}
    \revision{
    \begin{align}
        \left| \E \left[ \frac{1}{N^2} \sum_{ik} \G_{ki} Z_{ik}^{(m)} \caP^{2D-2} \right] \right| & \prec \left( \chi \frac{\Im \pair{\wt{\MM}} +\vartheta}{(N \eta)^2} + \Psi_2^4 \right)  \E |\caP|^{2D-2},
    \end{align}
    for $m=4, 5$. Note that the bound holds even when $B$ is random. In summary, we have
    \begin{align}
      | \E I_{1, 1}| \prec  \left( \Psi_2 ^4 + \Psi_2^2 \frac{|1 - \pair{\wt{\MM} \odot \wt{\MM} }|+\vt}{N\eta} +\chi \frac{\Im \pair{\wt{\MM}} +\vartheta}{(N \eta)^2} \right) \E |\caP|^{2D-2},
    \end{align}
    and we obtain the desired lemma.}     
\end{proof}
\revision{We remark that when replace $\chi $ by $q^{-2} + \vt $, which will be accomplished in Lemma \ref{lem:gminusrg}, the improved bound coincides with the bounds in sparse random matrix literature, where $|1- N^{-1}\sum_n (\lambda V_n - z- \pair{G})^{-2} |$ roughly corresponds to $|\caP'|$ in their settings. See also Section 6.3 of \cite{LS18}.}

\subsection{Estimates on $I_{2, 0}$} \label{sec:I20}
In this section, we prove the estimates on $I_{2, 0}$. 
\begin{proof}[Proof of Lemma \ref{lem:I20}]
	We begin by explicitly expanding
	\begin{align}
		\E I_{2,0} &= N \mc{C}_{3} \E \Bigg[\frac{1}{N^2}\sum_{i} \sum_{k}^{(i)} \pd_{ik}^2 \G_{ki} \mc{P}^{D-1} \overline{\mc{P}^D} \Bigg] \\
		&= N \mc{C}_3 \E \Bigg[\frac{1}{N^2} \sum_{i} \sum_{k}^{(i)} \bigg(2 G_{ik}^2 \G_{ki} + G_{ik} G_{kk} \G_{ii} + 2 G_{ii} G_{kk} \G_{ki} + 3 G_{ik} G_{ii} \G_{kk} \bigg)\mc{P}^{D-1} \overline{\mc{P}^D} \Bigg],\label{eq:I20}
	\end{align}
	where the term without any diagonal entries are easily bounded by $\Psi_2^2 \E |\mc{P}|^{2D-1}$. On the other hand, the terms with diagonal entries require extra care. Conducting an extra cumulant expansion on one of those terms, we deduce
	\begin{align*}
		N \mc{C}_3 \E \Bigg[ \frac{1}{N^2} \sum_{i \not= k} G_{ki} G_{kk} \G_{ii} \mc{P}^{D-1} \overline{\mc{P}^D
		} \Bigg] &=N \mc{C}_3 \E \Bigg[ \frac{1}{N^2} \sum_{i \not = k} \sum_{j} W_{kj} G_{ji}^{(k)} G_{kk}^2 \G_{ii} \mc{P}^{D-1} \overline{\mc{P}^D
		} \Bigg]\\
		&= N \mc{C}_3 \sum_{r=1}^{\ell'} \frac{\mc{C}_{r+1} }{r!}\frac{1}{N^2} \E \left[\sum_{ijk}  G_{ji}^{(k)} \pd_{kj}^{r} (G_{kk}^2 G_{ii} \mc{P}^{D-1} \overline{\mc{P}^{D}} )\right]  + O_{\prec} (q^{-\ell'}) \\
		& =  \frac{1}{N^3q } \E \left[\sum_{ijk}  \left(G_{ji} - \frac{G_{jk} G_{ki}}{G_{kk}} \right)\pd_{kj}(G_{kk}^2 G_{ii} ) \mc{P}^{D-1} \overline{\mc{P}^D}  \right] \\
		&\phantom{=}+ \frac{1}{N^3q } \E \left[\sum_{ijk}  \left(G_{ji} - \frac{G_{jk} G_{ki}}{G_{kk}} \right) G_{kk}^2 G_{ii}  \pd_{kj} (\mc{P}^{D-1} \overline{\mc{P}^D}) \right] \\
		& \phantom{=} + \sum_{r=2}^{\ell'} \bigO{\frac{1}{N^{3} q^{r} }} \E \left[ \sum_{ijk} \left(G_{ji} - \frac{G_{jk} G_{ki}}{G_{kk}} \right) \pd_{kj}^{r} (G_{kk}^2 G_{ii} \mc{P}^{D-1} \overline{\mc{P}^{D}} )\right]\\&\phantom{=} +O_{\prec} (q^{-\ell'} )
		\\
		& = O_{\prec} \left(\sum_{r=1}^{\ell'} \sum_{a=0}^{r} \Psi_2^{2(a+1)} \E |\mc{P}|^{2D-1-a} \right) + O_{\prec} (q^{-\ell'}),
	\end{align*}
	where the last step follows from the following observations. For the lower order terms, we have
	\begin{align*}
		\left| \frac{1}{N^3q } \E \left[ \sum_{ijk}  \left(G_{ji} - \frac{G_{jk} G_{ki}}{G_{kk}} \right)\pd_{kj}(G_{kk}^2 G_{ii} ) \mc{P}^{D-1} \overline{\mc{P}^D}  \right] \right| &\prec q^{-1} \Psi_2^2 \E|\mc{P}|^{2D-1}\\
		\left|\frac{1}{N^3q } \E \left[\sum_{ijk}  \left(G_{ji} - \frac{G_{jk} G_{ki}}{G_{kk}} \right) G_{kk}^2 G_{ii}  \pd_{kj} (\mc{P}^{D-1} \overline{\mc{P}^D}) \right] \right| 
		&\prec (q^{-1} \Psi_2 (N^{-1} + q^{-1} \Psi_2^2 + \Psi_2^3) \E |\mc{P}|^{2D-2}\prec \Psi_{2}^{4} \E |\mc{P}|^{2D-2}.
	\end{align*}
	For the higher order terms $r \geq 2$, analogous argument for fixed $r$ yields
	\begin{align*}
		\Bigg|\bigO{\frac{1}{N^{3} q^{r} }}\E \Bigg[ \sum_{ijk}  \left(G_{ji} - \frac{G_{jk} G_{ki}}{G_{kk}} \right) \pd_{kj}^{r} (G_{kk}^2 G_{ii} \mc{P}^{D-1} \overline{\mc{P}^{D}} )\Bigg] \Bigg|& \\ \prec \sum_{a=0}^{r} \frac{1}{q^r}  \Bigg(\frac{1}{N} + \frac{1}{q} \Psi_2^2 + \Psi_2^3 \Bigg)^a\Psi_2 \E |\mc{P}|^{2D-1-a} &\prec \sum_{a=0}^{r} \Psi_2^{2(a+1)} \E |\mc{P}|^{2D-1-a},
	\end{align*}
	where the last step is verified by computing the power index
	\begin{align}
		\frf(n_q = r, n_d = a , n_p =1, n_m = a+1) = \min \left\{r-1, 2a-1, \frac{r+a-2}{2} \right\},
	\end{align}
	which is non-negative for $r \geq 2$, unless $a = 0$. When $a=0$, the last step clearly holds given $r \geq 2$.
	
	The other terms in \eqref{eq:I20} are handled identically with trivial modifications. In conclusion, we have
	\begin{align}
		\E I_{2, 0} = O_{\prec} \left(\sum_{r=1}^{\ell'} \sum_{a=0}^{r} \Psi_2^{2(a+1)} \E |\mc{P}|^{2D-1-a} \right) + O_{\prec} (q^{-\ell'}),
	\end{align}
	for appropriately chosen $\ell'$ for all $z \in \mc{D}_{\tau}$.
\end{proof}

\subsection{Estimates on $I_{2, 1}$} \label{sec:I21}

In this section, we prove the estimates on $I_{2, 1}$. 
\begin{proof}[Proof of Lemma \ref{lem:I21}] 
	We begin by expanding the derivatives
	\begin{align}
		\E I_{2,1} &= N\mc{C}_3 \E \Bigg[ \frac{1}{N^2} \sum_{i} \sum_{k}^{(i)}  (\pd_{ik} \G_{ki}) (\pd_{ik} \mc{P}^{2D-1})  \Bigg] \\
		&= N \mc{C}_3 \E \left[ \frac{1}{N^2} \sum_{i} \sum_{k}^{(i)} (G_{ik}\G_{ki} + G_{kk} \G_{ii}) \left(\frac{1}{N} \G_{ki} + \frac{1}{N} \G_{ik} - Q_{ki} - Q_{ik} \right)  \mc{P}^{2D-2} \right],
	\end{align}
	where the terms with at least one off-diagonal resolvent entries are handled easily, dominated by $ ((Nq)^{-1} \Psi_2 + q^{-1}\Psi_2 (N^{-1} q^{-1} + q^{-1} \Psi_2^2 + \Psi_2^3)) \E |\mc{P}|^{2D-2} \prec \Psi_2^4 \E|\mc{P}|^{2D-2}$. To handle the remaining term, more extensive cumulant expansions are required to extract extra factor of $q^{-1}$. Consider
	\begin{align*}
		&\frac{1}{N^2 q} \E \Bigg[ \sum_{i} \sum_{k}^{(i)} Q_{ik}  G_{kk} \G_{ii}  \mc{P}^{2D-2} \Bigg] \\ 
		=&\frac{1}{N^3 q} \E \Bigg[ \sum_{i k j_1 j_2} W_{j_1 j_2} G_{j_2 k} \G_{ij_1}  G_{kk} \G_{ii}  \mc{P}^{2D-2} \Bigg] +  \frac{1}{N^3 q} \E \Bigg[  \sum_{ik}\Bigg( \pair{G} (\G G)_{ik} +  \pair{\G} (G^2)_{ik} \Bigg) G_{kk} \G_{ii} \mc{P}^{2D-2 } \Bigg]\\
		&+ \frac{1}{N^2q} \E \Bigg[ \sum_{ik} \Bigg( \frac{2s}{q^2 }\pair{G \odot G} \frac{1}{N} \sum_{j=1}^{N} \G_{jj} G_{ij} G_{jk}  + \pair{\G \odot G} \frac{1}{N} \sum_{j=1}^{N} G_{jj}G_{ij}G_{jk} \Bigg) G_{kk} \G_{ii} \mc{P}^{2D-2} \Bigg],
	\end{align*}
	where the last line is of order $O_{\prec} (q^{-3} \Psi_2^2 \E |\mc{P}|^{2D-2})$, which can be absorbed into $O_{\prec} (\Psi_{3/2}^2 \E|\mc{P}|^{2D-2})$. To deal with the first line, we expand further, which yields
	\begin{align}
		\begin{split}
			&\frac{1}{N^4 q} \E\Bigg[ \sum_{i k j_1 j_2} (\pd_{j_1 j_2} G_{j_2 k} \G_{ij_1} G_{kk} \G_{ii} ) \mc{P}^{2D-2} \Bigg] + \frac{1}{N^3 q} \E \Bigg[  \sum_{ik}\Bigg( \pair{G} (\G G)_{ik} +  \pair{\G} (G^2)_{ik} \Bigg) G_{kk} \G_{ii} \mc{P}^{2D-2 } \Bigg]\\
			+& \frac{1}{N^4 q} \E \Bigg[\sum_{i k j_1 j_2}  G_{j_2 k} \G_{ij_1} G_{kk} \G_{ii}  ( \pd_{j_1 j_2}\mc{P}^{2D-2}) \Bigg] + \sum_{r=2}^{\ell'} \frac{\mc{C}_{r+1}}{r!} \frac{1}{N^3 q}  \E \Bigg[\sum_{i k j_1 j_2}  \pd_{j_1 j_2}^r (G_{j_2 k} \G_{ij_1} G_{kk} \G_{ii}  \mc{P}^{2D-2}) \Bigg] \\
			+& O_{\prec} (q^{-\ell'}),\label{eq:long}
		\end{split}
	\end{align}
	where we deal with each terms separately. For the terms in the first line, we observe that if $\pd_{j_1 j_2}$ hits the diagonal entries one or more times, we gain two off-diagonal entries and become negligible. Aided with cancellation that was observed for $I_{1, 0}$ and $Y_{1, 0}$, the terms that have not yet been bounded are
	\begin{align}
		\frac{1}{N^4 q} \E\Bigg[ \sum_{i k j_1 j_2} (G_{j_1 j_2 } G_{j_2 k} \G_{ij_1} G_{kk} \G_{ii}  + G_{j_2k} \G_{j_2 j_1} G_{j_1 i} G_{kk} \G_{ii}  ) \mc{P}^{2D-2} \Bigg].\label{eq:llast10}
	\end{align}
	Note that the two terms are intrinsically of the same form, and we will obtain estimate for
	\begin{align}
		J_1 \deq \frac{1}{N^4 q} \E\Bigg[ \sum_{i k j_1 j_2} G_{j_1 j_2 } G_{j_2 k} G_{ij_1} G_{kk} G_{ii}  \mc{P}^{2D-2} \Bigg]. \label{eq:last10}
	\end{align}
	For the first term in the second line of \eqref{eq:long}, it is enough to consider
	\begin{align}
		J_2 \deq \frac{1}{N^4 q} \E \Bigg[\sum_{i k j_1 j_2} G_{j_2k} G_{ij_1} G_{kk} \G_{ii} Q_{j_1j_2} \mc{P}^{2D-3} \Bigg].\label{eq:last11}
	\end{align}
	We remark that from this point onward in the proof, we do not distinguish between $G$ and $\G$ as this assumption can be dropped without any consequences. We postpone the treatment of $J_{1}$ and $J_{2}$ to the end of the section.
	
	For the higher order terms, noticing that $\pd_{j_1j_2}^{r-s} (G_{j_2 k} \G_{ij_1} G_{kk} \G_{ii})$ contain at least two off-diagonal entries, for each fixed $r$ we have
	\begin{align}
		&\Bigg| \frac{1}{N^4 q^r} \E \Bigg[ \sum_{ik j_1j_2} \pd_{j_1j_2}^r (G_{j_2k } \G_{ij_1} G_{kk} \G_{ii} \mc{P}^{2D-2}) \Bigg] \Bigg| \\
		=& \Bigg| \sum_{s=0}^{r} \binom{r}{s} \frac{1}{N^4 q^r} \E \Bigg[ \sum_{ik j_1j_2} \pd_{j_1j_2}^{r-s} (G_{j_2k } \G_{ij_1} G_{kk} \G_{ii}) (\pd_{j_1j_2}^{s} \mc{P}^{2D-2}) \Bigg] \Bigg|\\
		\prec& \sum_{a =0}^{r} \frac{1}{q^{r}} \Psi_2^2 (N^{-1} +q^{-1} \Psi_2^2 + \Psi_2^3)^a \E|\mc{P}|^{2D-2 -a} \prec  \sum_{a'=2}^{r+2} \Psi_{3/2}^{2a'} \E |\mc{P}|^{2D -a'},\label{eq:Nq}
	\end{align}
	when $r \geq 3 $ or $r=2$ with $s\not=0$, where we have used the following facts.
	\begin{itemize}
		\item For $r \geq 2$, since $q^{a+1-r} \leq N^{(a-1)/2}$, $N^{-a}q^{-r} \leq (N^{-1/4} q^{-1/2})^{2(a+1)}$.
		\item We have $q^{-a -r} \Psi_{2a+2} \prec \Psi_{3/2}^{2a+4}$ when $a+r \geq 3$. That is, it is always true for $r \geq 3$ and true for $r=2 $ with $a \not=0$. The latter condition is equivalent to $r =2$ with $s \not =0$.
		\item We have $q^{-r} \Psi_2^{3a+2} \prec \Psi_{2}^{2a+2} \Psi_{3/2}^{2}$ for $r\geq 3$, immediately. For $r=2$, the statement also holds when $a \geq 1$, which is equivalent to $s \not =0$.
	\end{itemize}
	It remains to consider when $r=2$ and $s = 0$, $J_1$, and $J_2$ cases, where another cumulant expansion process is required to bound
	\begin{align} \label{eq:4}
		\frac{1}{N^4 q^2} \E \Bigg[\sum_{ik j_1j_2} \pd_{j_1 j_2}^2 (G_{j_2 k } \G_{ij_1} G_{kk} \G_{ii}) \mc{P}^{2D-2} \Bigg].
	\end{align}
	Before proceeding further, we study the property of $\pd_{j_1 j_2}^{2} G_{j_2k} \G_{ij_1} G_{kk} \G_{ii}  $. While the term is consisted of sum of product of six resolvent entries, those terms that contain at least three off-diagonal entries become negligible upon normalized summation, being of order $O_{\prec} (q^{-2} \Psi_2^3 \E |\mc{P}|^{2D-2})$. The non-negligible terms in the derivative, which contain only two off-diagonal resolvent entry, then must attain the form
	\begin{align}
		G_{ii} G_{kk} G_{i j_x } G_{k j_y} G_{j_x j_x} G_{j_y j_y},\label{eq:form}
	\end{align}
	where $(x,y) =(1, 2)$ or $(x, y) =(2, 1)$. Without loss of generality, it is enough to consider \eqref{eq:4} after replacing the derivative by \eqref{eq:form} with $(x, y) = (1, 2)$. That is, the only non-negligible term in \eqref{eq:4} is of the form
	\begin{align}
		J_3 \deq \frac{1}{N^4 q^2} \E \Bigg[\sum_{ik j_1j_2} G_{ij_1} G_{k j_2} G_{j_1 j_1} G_{j_2 j_2 } G_{ii} G_{kk} \mc{P}^{2D-2} \Bigg] \label{eq:last20},
	\end{align}
	up to trivial modifications. 
	
	In order to save work and space, we denote
	\begin{align}\label{eq:J}
		J_u \deq   \frac{1}{N^4 q} \E \Bigg[\sum_{ik j_1 j_2} G_{ij_1} G_{kj_2 } G_{ii}G_{kk} T_{j_1 j_2}^{(u)} \mc{P} ^{2D- t^{(u)}} \Bigg],
	\end{align}
	where $T_{j_1 j_2}^{(u)}$ and $t^{(u)}$ corresponds to 
	\begin{align*}
		T_{j_1 j_2}^{(1)} = G_{j_1 j_2}, \quad t^{(1)} =2, \quad T_{j_1 j_2}^{(2)} = Q_{j_1 j_2}, \quad t^{(2)} = 3, \quad T_{j_1 j_2}^{(3)} = q^{-1} G_{j_1 j_1} G_{j_2 j_2}, \quad t^{(3)} =2,
	\end{align*}
	and we will often abbreviate $J \equiv J_{u} $, $T_{j_1 j_2} \equiv T_{j_1 j_2}^{(u)}$, and $t \equiv t^{(u)}$ unless necessary. 
	
	Applying the resolvent identity and standard cumulant expansion method to \eqref{eq:J}, we have
	\begin{align}
		J= &\frac{1}{N^4 q} \E \Bigg[-\sum_{ik j_1 j_2} G_{ii} \sum_{n}^{(i)} W_{i n} G_{nj_1}^{(i)} G_{kj_2 } G_{ii}G_{kk} T_{j_1 j_2} \mc{P} ^{2D- t} \Bigg] \\
		\begin{split}
			=&-\frac{1}{N^5 q} \E \Bigg[ \sum_{\substack{ k j_1 j_2 \\ n \not =i}} \Bigg(G_{n j_1} - \frac{G_{n i} G_{ij_1}}{G_{ii}} \Bigg) \pd_{in} \big(G_{kj_2} T_{j_1 j_2} G_{ii}^{2} G_{kk} \mc{P}^{2D-t} \big) \Bigg] \\
			&-\sum_{r=2}^{\ell''} \frac{\mc{C}_{r+1}}{r!}  \frac{1}{N^4q} \E \Bigg[ \sum_{\substack{ k j_1 j_2 \\ n \not =i}}  \Bigg(G_{n j_1} - \frac{G_{n i} G_{ij_1}}{G_{ii}} \Bigg) \pd_{in}^r \big(G_{kj_2} T_{j_1 j_2} G_{ii}^{2} G_{kk} \mc{P}^{2D-t} \big) \Bigg] + O_{\prec} (q^{-\ell'}),\label{eq:ll}
		\end{split}
	\end{align}
	where we consider the first line and the second line of \eqref{eq:ll} separately. The first line is consisted of the following two terms, which are
	\begin{align}
		\frac{1}{N^5 q} \E \Bigg[ -\sum_{ik j_1 j_2} \sum_{n}^{(i)} \Bigg(G_{n j_1} - \frac{G_{n i} G_{ij_1}}{G_{ii}} \Bigg) \pd_{in}\big( G_{kj_2} T_{j_1 j_2} G_{ii}^{2} G_{kk}\big) \mc{P}^{2D-t}  \Bigg],\label{eq:ll1}\\
		\frac{1}{N^5 q} \E \Bigg[ -\sum_{ik j_1 j_2} \sum_{n}^{(i)} \Bigg(G_{n j_1} - \frac{G_{n i} G_{ij_1}}{G_{ii}} \Bigg)  G_{kj_2} T_{j_1 j_2} G_{ii}^{2} G_{kk} \big(\pd_{in} \mc{P}^{2D-t} ) \Bigg]. \label{eq:ll2}
	\end{align}
	For \eqref{eq:ll1}, for each cases we argue as follows. 
	\begin{itemize}
		\item $J_1$: Since the term with derivative $\pd_{in}$ always contain at least three off-diagonal entries, there are at least four off-diagonal entries, leading to the bound $O_{\prec} (q^{-1} \Psi_2^4 \E |\mc{P}|^{2D-2})$, which is negligible.
		\item $J_2$: If $\pd_{in}$ hits entries other than $T_{j_1 j_2}$ at least one additional off-diagonal entries are obtained, leading to the bound $O_{\prec} (q^{-1} \Psi_2^3 (N^{-1} + q^{-1} \Psi_2^2 + \Psi_2^3)\E |\mc{P}|^{2D-3})$, which can be absorbed into the bound $O_{\prec} (\Psi_2^6 \E|\mc{P}|^{2D-3})$. The case becomes subtle when $\pd_{in}$ hits $T_{j_1j_2}$, and an explicit expansion is required. To this end, we have
		\begin{align*}
			\eqref{eq:ll1} =& \frac{1}{N^5 q} \E \Bigg[ -\sum_{ik j_1 j_2} \sum_{n}^{(i)} \Bigg(G_{n j_1} - \frac{G_{n i} G_{ij_1}}{G_{ii}} \Bigg)  G_{kj_2} G_{ii}^{2} G_{kk}  \Bigg( \frac{1}{N} \G_{j_1i} G_{j_2n} + \frac{1}{N} \G_{j_1n} G_{j_2i} \\
			&\phantom{\sum \sum \sum \sum \sum \sum \sum}- G_{ij_1} Q_{nj_2} - G_{ij_2} Q_{j_1n}  - G_{nj_1}Q_{j_2i} - G_{nj_2} Q_{j_1i} \Bigg)\mc{P}^{2D-3}  \Bigg],\\
			& + O_{\prec} (q^{-1}\Psi_2^2 ( q^{-2} \Psi_2^2 + \Psi^4)  \E|\mc{P}|^{2D-3}),
		\end{align*}
		where we have used \eqref{eq:3line}. We deal with the first and the second line separately. The terms in the first line have four off-diagonal entry and from the extra factors we obtain the bound $O_{\prec}(N^{-1}q^{-1} \Psi_2^4 \E|\mc{P}|^{2D-3}) $, which is negligible. For the terms in the second line, there are three off-diagonal entries and from the bound on $Q$, they attain the bound $O_{\prec} (q^{-1} \Psi_2^3 (N^{-1} + q^{-1} \Psi_2^2 + \Psi_2^3) \E |\mc{P}|^{2D-3} )$, which is also absorbed into the desired bound.
		\item $J_3$: The term with derivative $\pd_{in}$ always contain at least two off-diagonal entries, and with the aid of an extra factor $q^{-1}$ in $T_{j_1 j_2}$, the bound $O_{\prec} (q^{-2} \Psi_2^3 \E|\mc{P}|^{2D-2})$, which is negligible. In conclusion, it is bounded by $O_{\prec} (\Psi_{3/2}^{3} \E |\caP|^{2D-3})$.
	\end{itemize}
	The term \eqref{eq:ll2} can be considered along with the second line in \eqref{eq:ll}. Let us assume $\pd_{xy}^{l} T_{j_1 j_2} = O_{\prec} (\varsigma)$, for any $x, y \in \llbracket 1, N \rrbracket$ and any non-negative integer $l$. Noticing that the minimum number of off-diagonal entry of $(\pd_{in} G_{kj_2} T_{j_1 j_2} G_{ii}^2 G_{kk})$ is $2$ for $J_1$ while $1$ for $J_2$ and $J_3$, with a slight abuse of notation, we simply pass a factor $\Psi_2$ to $\varsigma$ for $J_1$. Then, we write the bound
	\begin{align}
    \begin{split}
		&\frac{\mc{C}_{r+1}}{r!}  \frac{1}{N^4q} \E \Bigg[ -\sum_{ik j_1 j_2} \sum_{n}^{(i)} \Bigg(G_{n j_1} - \frac{G_{n i} G_{ij_1}}{G_{ii}} \Bigg) \pd_{in}^r \big(G_{kj_2} T_{j_1 j_2} G_{ii}^{2} G_{kk} \mc{P}^{2D-t} \big) \Bigg] \\
		\prec & \sum_{a=0}^{r} \frac{1}{q^r} \Psi_2^2 \varsigma (N^{-1} + q^{-1} \Psi_2^2 + \Psi_2^3)^a \E |\mc{P}|^{2D- t-a} \prec \sum_{a=0}^{r} \Psi_{3/2}^{2} \Psi_2^{2(a+t-1)} \E|\mc{P}|^{2D-t-a},\label{eq:153}
        \end{split}
	\end{align}
	for any fixed $r$ satisfying $r \geq 2$ or $r=1$ assuming $a \not =0$. The latter condition is equivalent to $r=1$ with the derivative acting on $\mc{P}^{2D-t}$. We have used the following facts in the last step.
	
	\begin{itemize}
		\item  $J_1$: In this case, $\vs = \Psi_2 $.  \\
		For $r \geq 2$, since $q^{a+1-r} \leq N^{(a-1)/2}$, $N^{-a}q^{-r} \Psi_2 \varsigma \leq (N^{-1/4} q^{-1/2})^{2(a+1)} \Psi_2^2$.\\
		For the other terms, we use that $q^{-a -r} \Psi_2^{2a+2} \varsigma \prec \Psi_2^{2a+4} $ when $a +r \geq 2$. Also, $q^{-r} \Psi_2^{3a+2} \varsigma \prec \Psi_2^{2a+4}$ when $r \geq 2$ or $r=1$ with $a \not =0$.\\
		For $(r,a)=(1,1)$, a separate treatment is required; the derivative of $\mc{P}$ exactly attains the bound for an entry of $Q = O_{\prec} (N^{-1} q^{-1} + q^{-1} \Psi_2^2 + \Psi_2^3)$, the improved factor $(Nq)^{-1}$ resolves the issue. 
		\item $J_2$: In this case, $\vs = (N^{-1} + q^{-1} \Psi_2^2 +\Psi_2^3)$, and so we replace the power of the term behind $\vs$ into $a+1$, and simply remove $\vs$. \\
		For $r \geq 2$, since $q^{a+2-r} \leq N^{a/2}$, $N^{-a-1}q^{-r} \Psi_2^2\leq (N^{-1/4} q^{-1/2})^{2(a+2)} \Psi_2^2$. \\
		For the other terms, we use that $q^{-a -r-1} \Psi_2^{2a+4}  \prec \Psi_{3/2}^{2} \Psi_2^{2a+4} $ when $a +r +1 \geq 3$. Also, $q^{-r} \Psi_2^{3a+5} \prec \Psi_{3/2} \Psi_2^{2a+5}$ when $r \geq 2$, or $r=1$ with $a \not =0$\\
		For $r=1$ and $a=1$, treatment analogous to the case of $J_1$ settles the problem. 
		\item $J_3$: In this case, $\vs = q^{-1}$, and so it is equivalent to replacing the factor $1/q^r$ by $1/q^{r+1}$ and simply removing $\vs$. \\
		For $r \geq 2$, since $q^{a+1-r} \leq N^{(a-1)/2}$, $N^{-a}q^{-r-1} \Psi_2^2\leq (N^{-1/4} q^{-1/2})^{2a+2} \Psi_2^2$.\\
		For the other terms, we use that $q^{-a -r -1} \Psi_2^{2a}\prec \Psi_{3/2}^{2} \Psi_2^{2a} $ when $a +r +1 \geq 3$. Also, $q^{-r-1} \Psi_2^{3a+2} \prec \Psi_{3/2}^2 \Psi_2^{2a+2}$ when $r+1 \geq 3$, or $r=1$ with $a \not =0$.\\
		For $r=1$ and $a=1$, treatment analogous to the case of $J_1$ settles the problem. 
	\end{itemize}
	
	In conclusion, we have
	\begin{align*}
		\Bigg|\frac{1}{N^2 q} \E \Bigg[ \sum_{i \not = k} Q_{ik}  G_{kk} \G_{ii}  \mc{P}^{2D-2} \Bigg] \Bigg| \prec &  \sum_{r=0}^{\ell'} \sum_{a'=2}^{r+2} \Psi_{3/2}^{2a'} \E |\mc{P}|^{2D -a'} +\sum_{r=0}^{\ell''} \sum_{a''=2}^{r+3} \Psi_{3/2}^{2} \Psi_2^{2(a''-1)} \E|\mc{P}|^{2D-a''}+ q^{-\ell'} +q^{-\ell''},
	\end{align*}
	uniformly on $z \in \caD_{\tau}$.
\end{proof}

\begin{remark}\label{rem:import}
	The main difference in the recursive moment estimate for the deformed model is in the bound on the derivative of $\caP$. For sparse random matrices, $\caP$ is a polynomial on $z, g$, where $|\caP'| $ is close to $\Im \wt{m}$, serving as an effective factor at the spectral edges. 
	
	Let us compare  Lemma \ref{lem:I21} to the known results for sparse random matrices. The main task is bounding the term
	\begin{align}
		-N \caC_{3}\E \left[\frac{2}{N^2} \sum_{ik} \pd_{ik} (G_{ki}) \pd_{ij} (\caP^{D-1} \ovl{\caP^{D}}) \right],\label{eq:problem}
	\end{align}
	where different types of $\caP$ have been used. In \cite{LS18}, a fourth order (deterministic) polynomial was considered in place of $\caP$. By conducting another cumulant expansion using one of off-diagonal entry generated from $\pd_{ik} \caP(g)$, an extra factor $q^{-1}$ was gained, settling the issue. 
	
	Although we have also considered the deterministic shift up to order of $q^{-2}$, the extra $q^{-1}$ factor from cumulant expansion was not sufficient to prove the local law up to the bound given for sparse random matrices. It is due to the weaker estimate on $Q $ in our case, where the bound $q^{-1}\Psi_2^2$ could not be improved further, at least in our exposition. \revision{On the other hand, it is natural to conjecture that deterministic shift at the edge is sufficient to recover edge universality for $q\gg N^{1/6}$. This can also be observed in our analysis by comparing \eqref{eq:problem} with the random fluctuation term studied in sparse random matrix literature. In \cite{HK, HLY20, HY, Le} the random polynomial $\caP_r (z, g)$ defined by}
	\begin{align}
		\caP_r (z, g) &=  \caP_{0}(z, g) + \caZ g^2,\\
		\caP_0(z, g) &=1 + z g + g^2  + \frac{1}{q^2} ( a_2 g^4 + a_3 g^6 + \cdots),\quad \caZ = \frac{1}{N}\sum_{ij} \left(W_{ij}^2 - \frac{1}{N} \right),
	\end{align}
	was used in place of $\caP$, to exploit the eigenvalue behaviors of sparse random matrices up to $q>N^{\delta}$ for any $\delta >0$. It is known that $\sqrt{N}q \caZ$ is asymptotically Gaussian with mean $0$ and variance $Nq^2 \caC_{4}$. By applying cumulant expansion to the $\caZ$ portion, we encounter
	\begin{align}
		N\caC_3 \E \left[ \frac{2}{N} \sum_{ik} \pd_{ik}(W_{ik}) \pd_{ik} (g^2 \caP_r^{D-1} \ovl{\caP_r^{D}}) \right],
	\end{align}
	which cancels out \eqref{eq:problem} up to negligible terms, using the mechanism replacing the diagonal resolvent entries with its normalized trace. See also \cite[Proposition 4.2]{Le}. The intricate terms in $I_{3,2} $ could be coped using the same idea, for sparse random matrices. \revision{Then, it is natural to expect that the estimates of $\E I_{2.1}$ or $\E I_{3.2}$ should not exceed the fluctuation size given by the random fluctuation terms.}
	
	\revision{A different method that was used to deal with the derivatives of $\caP$, used in $\E I_{1,1}$ (Lemma \ref{lem:I11}), could be a method to obtain optimality, as it extracts terms of size roughly $|1- \pair{\wt{\MM} \odot \wt{\MM}} | \sim \Im \wt{m}_{\fc}$ that corresponds to $|\caP'|$. However, as it can be glimpsed in its proof, it requires careful sum over indices that also originates from terms other than $\caP$, hence its adaption to $\E I_{2,1}$ or $\E I_{3.2}$ remain non  trivial. In the direction of achieving edge universality removing all deterministic and random fluctuations as in \cite{HY}, the random fluctuation term would not be represented in a simple polynomial in our model, due to its anisotropic nature. That is, preparing an appropriate $\caP$ with a suitable corresponding equilibrium measure is a complex problem itself.}
\end{remark}

\section{Proof of local laws}\label{sec:locallaw}
\subsection{Preliminaries}
A first step in proving the local law is to transform Proposition \ref{prop:m1} and Proposition \ref{prop:main} into a strong self-consistent equations. 

In this section again, we work on $\Xi$, and hence the entries of $V$ are fixed. 

To prove the local law with respect to the refined measure $\wt{\MM}$, we obtain a bound on $\wt{\mathsf{P}} = (\wt{\mathsf{P}}_{ij})$, a slightly modified version of $\mathsf{P}$, defined through
\begin{align}
	\wt{\mathsf{P}}_{ij} \equiv \wt{\mathsf{P}}_{ij} (z) \deq \delta_{ij} + z G_{ij} + \pair{G} G_{ij} - \lambda V_i G_{ij} + \frac{s}{q^2} \frac{G_{ij}}{\lambda V_i - z - \pair{G}},  
\end{align}
which resembles the structure of $\wt{\MM}$. Furthermore, for some arbitrary (deterministic) matrix $B \in \bbC^{N \times N}$, we denote $\wt{\caP} \deq \pair{B\wt{\mathsf{P}}}$. It is also convenient to introduce $R\equiv R(\cdot, z) $ and $\wt{R}\equiv \wt{R}(\cdot , z )$, defined through
\begin{align}
	R(u) &\equiv R(u, z) = (R(u , z)_{ij}) = \delta_{ij}(\lambda V_i - z - u(z) )^{-1}, \\
	\wt{R}(u)&\equiv \wt{R}(u, z) =(\wt{R}(u, z)_{ij})  = \delta_{ij} \left(\lambda V_i - z - u(z) -\frac{s}{q^2} \frac{1} {\lambda V_i - z -u(z) }\right)^{-1},
\end{align}
and note that $R(\pair{\wh{M}}) = \wh{M}$ and  $\wt{R}(\pair{\wt{\MM}}) = \wt{\MM} $.

The following estimates are modifications of Proposition \ref{prop:m1} and Proposition \ref{prop:main}.
\begin{lemma}\label{lem:61}
	On the event $\Xi$, we have
	\begin{align}
		\max_{ij} |\wt{\mathsf{P}}_{ij}| \prec \Psi_1, \quad |\wt{\caP} | \prec \Psi_1^2 , 
	\end{align}
	uniformly for all $z \in \caD_{\tau} $. \revision{Moreover, the statements still hold with $\wt{\mathsf{P}}$ replaced by $\wt{R}(\pair{G}) \wt{\mathsf{P}} $ or by  $R(\pair{G})\wt{R}(\pair{G})  \wt{\mathsf{P}} $ }.
\end{lemma}
\begin{proof}
    \revision{First, we extend stability results to the terms with randomness. Observe that
	\begin{align}
		{R}(\pair{G}) = {R}(\pair{\wt{\MM}}) + {R} (\pair{G}) {R} (\pair{\wt{\MM}})  \left( \pair{G} - \pair{\wt{\MM}}  \right),
	\end{align}
	and from \revision{weak local law and the stability bounds}, we have $R(\pair{G}) = O_{\prec}(1)$. Using a similar formalism, we deduce that $\wt{R}(\pair{G}) = O_{\prec}(1)$. Then we have, for arbitrary $i, j \in \llbracket 1, N \rrbracket$
	\begin{align*}
		|\sfP_{ij} - \wt{\sfP}_{ij}| &= \frac{s}{q^2} \left| \frac{G_{ij} }{\lambda V_i -z  - \pair{G}} \right|  \left| \pair{G \odot G} G_{ii} (\lambda V_i - z-\pair{G}) -1  \right| \\
		& \prec s{q^{-2}} \left||\pair{G \odot G}||\mathsf{P}_{ii}| + |\pair{G \odot G} - \pair{\wt{\MM} \odot \wt{\MM}}| + |1- \pair{\wt{\MM} \odot \wt{\MM}}| \right| + O(q^{-4}),
	\end{align*}
	where we end up in a trivial bound $\max_{ij} |\sfP_{ij} - \wt{\sfP}_{ij}| \prec q^{-2}$, and we have the first claim.}

    \revision{The second claim originates from \cite[Lemma 4.6]{HKR}, Firstly, we complete the lattice argument. Let $\Omega_r$ be an event of high probability such that on $\Omega_r$, $\norm {R}(\pair{G}) \norm = O(1)$. Then, on $\Omega_r$, $\wt{R}(\pair{G}) = O(1)$.}
				
				\revision{Define the set $\caU \deq \{ g =\pair{G} = \frac{1}{N} \Tr \frac{1}{ H - z \mathrm{I}} :H \in \Omega_r, z \in \caD_{\tau} \} \subset \bbC$. Since $ |g| = |\pair{G}| \leq N$, $\caU$ is contained in the ball of radius $N$ around the origin. Let $\overline{\caU} \subset \caU$ be an $N^{-4}$ net of $\caU$. That is, $|\overline{\caU}| = O(N^{10})$ and for each $u \in \caU$, there exist a $\overline{u}$ such that $|u - \overline{u}| \leq N^{-4}$.} 
				
				\revision{From union bound, and applying the result of Lemma \ref{lem:61} or Lemma \ref{lem:62}, it follows that, uniformly for any $i, j  \in \llbracket1, N \rrbracket$,
				\begin{align}\label{eq:64}
					\sup_{\overline{u} \in \overline{\caU}} |(\wt{R}(\overline{u}) \mathsf{P})_{ij}| \prec \Psi_{1}.
				\end{align}
				Then for each $g \in\caU$, and $\overline{g} \in \overline{\caU} $ such that $|g - \overline{g}|\leq N^{-4}$, we have for any $i, j \in \llbracket 1, N \rrbracket$
				\begin{align*}
					|(\wt{R}(g)\mathsf{P})_{ij}| &\leq |(\wt{R}(\overline{g})\mathsf{P})_{ij}| + |((\wt{R}({g}) - \wt{R}(\overline{g} )\mathsf{P}))_{ij}| \\
					&\leq |(\wt{R}(\overline{g})\mathsf{P})_{ij}|  + |g  - \overline{g}|  \norm(1 + {s}q^{-2} R(\overline{g}) R(g) )\wt{R}(\overline{g})  \wt{R}(g)\norm \norm \mathsf{P} \norm \\
					& \prec \Psi_{1} + N^{-2},
				\end{align*}
				where \eqref{eq:64} and the trivial bound $\norm \mathsf{P} \norm = O(N^2)$ was used in the last line.}
				
				\revision{Analogous argument holds for normalized trace. Notice that the corresponding estimate \eqref{eq:64} for normalized trace would be bounded by $\Psi_{1}^2 $. Moreover, argument on $R(g) \wt{R}(g) \mathsf{P}$ can be achieved through similar expansions, completing the proof.}
\end{proof}

\begin{lemma}\label{lem:62}
	On the event $\Xi$, \revision{assume that Theorem \ref{thm:strong} and Lemma \ref{lem:61} holds. Assume that $\pair{G} - \pair{\wt{\MM}} = O_{\prec} (\vt)$ on $\caD_{\tau}$, for some deterministic $\vt \in [N^{-1}, N^{\tau/10}]$. Consider arbitrary deterministic diagonal matrix $B$ with $\norm B \norm = O(1)$. Furthermore assume that $\sup_j|B_{jj} - (\lambda V_{j} - z - \pair{G})^{-1}| \prec \chi$ for some $\chi \in [q^{-2} , N^{\tau/10}]$. Then, on the event $\Xi$, we have 
    \revision{
	\begin{align}
		|\wt{\caP} |\deq |\pair{B\wt{\mathsf{P}}}| \prec \Psi_{3/2}^{2} + \Psi_2 \left(\sqrt{\frac{|1- \pair{\wt{\MM} \odot \wt{\MM}}| + \chi}{N\eta}  }  \right) + \frac{|1- \pair{\wt{\MM} \odot \wt{\MM}}|}{q^2}  ,
	\end{align}}
	uniformly on $z = E+\ii \eta \in \caD_{\tau}$.}
\end{lemma}
\begin{proof}
	 Assuming the strong entrywise local law further, we have the bound
	\begin{align}
		\max_{ij} |\sfP_{ij} - \wt{\sfP}_{ij}|\prec   \frac{1}{q^2} \left(\frac{1}{q} + \sqrt{\frac{\Im \pair{\wt{\MM}}}{N\eta}} + |1-\pair{\wt{\MM} \odot \wt{\MM}}| \right) \prec \Psi_{3/2}^2 + q^{-2}  |1- \pair{\wt{\MM} \odot \wt{\MM}}|,
	\end{align}
	\revision{and the claim follows upon combining Proposition \ref{prop:main}.}
\end{proof}

\revision{Finally, we replace $B$ by $\wt{R}({\pair{G}})$ in Lemma \ref{lem:62}.}

\begin{lemma}\label{lem:gminusrg} \revision{On the event $\Xi$, assume that Lemma \ref{lem:61} and Theorem \ref{thm:strong} hold. Then, we  have
\begin{align}
    \left| \pair{G} - \pair{\wt{R} (\pair{G})} \right| \prec \Psi_{3/2}^2 + \Psi_2 \sqrt{\frac{|1- \pair{\wt{\MM}\odot\wt{\MM}}| +\vartheta}{N\eta}}+ \frac{|1- \pair{\wt{\MM} \odot \wt{\MM}}|}{q^2}
\end{align}}
\end{lemma}
\begin{proof}
    \revision{The lattice argument is similar to Lemma \ref{lem:61}. We work on the high probability event $\Omega_r$, and recall that we assume throughout that for $u\in \Omega_r$, $|u - \wt{m}_{\fc}| \prec \vartheta$. Consider a lattice $\overline{\caU}$ as in Lemma \ref{lem:61}, then we have
					\begin{align}
						\sup_{\overline{u} \in \overline{\caU}} \norm{\wt{R} (\overline{u}) - R({g})}\norm \leq\sup_{\overline{u} \in \overline{\caU}} \norm \wt{R} (\overline{u}) - R({\wt{m}_{\fc}}) \norm + \norm R(\wt{m}_{\fc} ) - R(g) \norm \prec \vartheta + q^{-2},
					\end{align} 
					which can be obtained by utilizing the stability bound and a union bound. Then, by a union bound, we have
                    \begin{align}
                       \sup_{\overline{u} \in \overline{\caU}}|\pair{\wt{R} (\overline{u}) \mathsf{P}(g)}| \prec \Psi_{3/2}^2 + \Psi_2 \sqrt{\frac{|1- \pair{\wt{\MM}\odot\wt{\MM}}| +\vartheta}{N\eta}}+ \frac{|1- \pair{\wt{\MM} \odot \wt{\MM}}|}{q^2} ,
                    \end{align}
                    where we used the assumption $q^{-2} \leq \vartheta$. Since for each $g \in \caU$, there exist $\overline{g} \in \overline{\caU}$ such that $|g- \overline{g}| \leq N^{-4}$, we may write
                      \begin{align}
                        |\pair{G} - \pair{\wt{R}(\pair{G})}| = |\pair{\wt{R}(g) \wt{\mathsf{P}} (g) }| &\leq | \pair{\wt{R} (\overline{g} ) \wt{\sfP}(g)}| + \norm{\wt{R}(\overline{g} ) - \wt{R}(g)}\norm \norm{\sfP}\norm \\
                        &\prec \Psi_{3/2}^2 + \Psi_2 \sqrt{\frac{|1- \pair{\wt{\MM}\odot\wt{\MM}}| +\vartheta}{N\eta}}+ \frac{|1- \pair{\wt{\MM} \odot \wt{\MM}}|}{q^2}  + N^{-2},
                    \end{align}}
                \revision{and the desired estimate follows.}
\end{proof}


We will denote
\begin{align}
	\wt{\mathsf{R}}^{(n)}(z) \equiv \wt{\mathsf{R}}^{(n)} \deq  \int \frac{1}{ \left(\lambda v - z - \wt{m}_{\fc}(z) \right)^{n}} \d \wt{\nu}_N(v) , 
\end{align}
for each $n \in \bbN$. Comparison of $\wt{\mathsf{R}}^{(2)}$ with $\pair{\wt{\MM} \odot \wt{\MM}}$ and $\wt{\mathsf{R}}^{(3)}$ with $\pair{\wt{\MM} \odot \wt{\MM} \odot \wt{\MM}}$ are important in particular. We have 
\begin{align}\label{eq:r2mm}
	\left|\wt{\mathsf{R}}^{(2)} -  \pair{ \wt{\MM} \odot \wt{\MM}} \right| = \frac{s}{q^2} \left| \int \frac{1}{( \left( \lambda v - z - \wt{m}_{\fc} \right)^2 - sq^{-2} )^{2}} \d \wh{\nu}(v) \right| = O(q^{-2}),\\
\label{eq:r3mmm}
	\left|\wt{\mathsf{R}}^{(3)} -  \pair{ \wt{\MM} \odot \wt{\MM} \odot \wt{\MM}} \right| = \frac{3s}{q^2} \left|\int \frac{\lambda v - z - \wt{m}_{\fc}}{\left((\lambda v - z - \wt{m}_{\fc})^2 - sq^{-2} \right)^3} \d \wh{\nu} (v)  \right| = O(q^{-2}),
\end{align}
by utilizing the stability bounds.

Now, we derive the self-consistent equations. For convenience, we denote $[\mathrm{v}] \deq \pair{G} - \pair{\wt{\MM}} $ and $\Lambda\equiv \Lambda(z) \deq |\pair{G(z)} - \pair{\wt{\MM}(z)}|$.

\begin{lemma}[Strong self-consistent equation] \label{lem:self1} Fix arbitrary $\epsilon \in (0, \frac{\tau}{100})$. Assume that Lemma \ref{thm:weak} and Lemma \ref{lem:61} holds. Then, there exist a constant $C>0$ such that 
	\begin{align}
		\left|(1 - \wt{\mathsf{R}}^{(2)} ) [\mathrm{v}]  \revision{-}  \wt{\mathsf{R}}^{(3)} [\mathrm{v}]^2 \right| \leq C (q^{-2} \Lambda + o(1)\Lambda^2 ) + N^{\frac{\epsilon}{5}} \left( \frac{1}{q^2} + \frac{1}{N^{1/2} q} +  \frac{|1 - \wt{\mathsf{R}}^{(2)}|}{q^2} + \frac{\Im \pair{\wt{\MM}} + \vartheta }{N\eta}\right),
	\end{align}
	uniformly on $z \in \caD_{\tau}$, for $N \geq N_0 (\epsilon)$ large enough, with high probability.
\end{lemma}

\begin{lemma}[Strong self-consistent equation] \label{lem:self2}
	Fix arbitrary $\epsilon \in (0, \frac{\tau}{100})$. Assume that Lemma \ref{thm:weak}, Lemma \ref{lem:61}, and Theorem \ref{thm:strong}, Lemma \ref{lem:62}  holds. Then, there exist a constant $C>0$ such that
	\begin{align} \label{eq:scone}
		\left|(1 - \wt{\mathsf{R}}^{(2)} ) [\mathrm{v}]  \revision{-}  \wt{\mathsf{R}}^{(3)} [\mathrm{v}]^2 \right| \leq C (q^{-2} \Lambda + o(1) \Lambda^2 ) + N^{\frac{\epsilon}{5}} \left( \frac{1}{q^3} + \frac{1}{N^{1/2} q} +  \frac{|1 - \wt{\mathsf{R}}^{(2)}|}{q^2} + \frac{\Im \pair{\wt{\MM}} + \vartheta }{N\eta}\right),
	\end{align}
	uniformly on $z \in \caD_{\tau}$, for $N \geq N_0 (\epsilon)$ large enough, with high probability.
\end{lemma}

\begin{proof}[Proof of Lemma \ref{lem:self1} and \ref{lem:self2}] \revision{From Lemma \ref{lem:61}}, we have \begin{align}\label{eq:ooo}
		\left|\pair{G} - \pair{\wt{R} (\pair{G})} \right| \prec \Psi_1^2,
	\end{align}
	uniformly on $ z \in \caD_{\tau}$.
	Third order Taylor expansion of $\wt{{R}}(G)$ around $\wt{R}(\wt{\MM})$ yields
		\begin{align}
		\begin{split} \label{eq:sum2}
			\Bigg|\wt{{R}}_i (G) - \wt{{R}}_i(\wt{\MM}) &- [\mathrm{v}] \Bigg( 1+ \frac{s}{q^2} \frac{1}{(\lambda V_i - z- \pair{\wt{\MM}})^2} \Bigg)  \wt{{R}}_i(\wt{\MM})^2- [\mathrm{v}]^2 \left( 1 + \frac{3s}{q^2} \frac{1}{(\lambda V_i - z - \pair{\wt{\MM}})^2 } \right) \wt{{R}}_i(\wt{\MM})^3 \Bigg | \leq C \Lambda^3,
		\end{split}
	\end{align}
	for some constant $C$. Using the estimates \eqref{eq:r2mm} and \eqref{eq:r3mmm} and noticing that $\wt{{R}}_i (\wt{\MM}) = \wt{\MM}_i$, we obtain Lemma \ref{lem:self1} after summing up \eqref{eq:ooo} and \eqref{eq:sum2}. 
	
	\revision{Assuming further Theorem \ref{thm:strong} and Lemma \ref{lem:61}, the right hand side of \eqref{eq:ooo} becomes $\Psi_{3/2}^2 + q^{-2} |1- \wt{\mathsf{R}}^{(2)} |$, and we obtain Lemma \ref{lem:self2}.}
\end{proof}

A key input to the proof of the strong local law is the following deterministic lemma, which is an adaptation of \cite[Lemma 4.6]{LS} to the current setting.

\begin{lemma} \label{lem:det} Fix arbitrary $\epsilon \in (0, \frac{\tau}{100})$, and let $\frb \in [1, 3/2]$. Assume that a function $\vartheta(z)$ satisfying
	\begin{align}
		\vartheta(z) \leq N^{\epsilon} \left(\frac{1}{q^{\frb}} + \frac{1}{N^{1/4} q^{1/2}} + \frac{1}{N\eta} \right)^{1-\upsilon},
	\end{align}
	also satisfies $\Lambda(z) \leq \vartheta(z)$, for all $z \in \mathcal{D}_{\tau}$. Furthermore, for all $z \in \caD_{\tau}$, 
	\begin{align}
		|(1- \wt{\mathsf{R}}^{(2)} ) [\mathrm{v}] -  \wt{\mathsf{R}}^{(3)}  [\mathrm{v}]^2 | \leq  C(q^{-2} \Lambda  + o(1) \Lambda^2) +N^{\frac{\epsilon}{5}} \left(\frac{1}{q^{2\frb}} + \frac{1}{N^{1/2} q} +\frac{|1- \wt{\mathsf{R}}^{(2)} |}{q^2}+ \frac{ |1- \wt{\mathsf{R}}^{(2)} | + \vartheta}{N \eta} \right),
	\end{align}
	Finally, assume that $\Lambda \prec 1$ if $\eta \sim 1$. Then, we have
	\begin{align}
		\Lambda \leq N^{\epsilon} \left(\frac{1}{q^{\frb}} + \frac{1}{N^{1/4} q^{1/2}} + \frac{1}{N\eta} \right)^{1 - \upsilon/2},
	\end{align}
	uniformly on the domain $\caD_{\tau}$, for $N$ sufficiently large.
\end{lemma}
The proof is based on a dichotomy argument, and we defer the details to the Appendix.
\begin{remark}
	Using stochastic domination, the statement of the lemma can be reduced as follows. On the event that Lemma \ref{lem:self1} or Lemma \ref{lem:self2} holds, and when $\Lambda \prec 1 $ holds if $\eta \sim 1$, we have the implication 
	\begin{align}
		\Lambda \prec \left(\frac{1}{q^{\frb}} + \frac{1}{N^{1/4} q^{1/2}} + \frac{1}{N \eta} \right)^{1 - \upsilon} \Longrightarrow  \Lambda \prec \left(\frac{1}{q^{\frb}}+ \frac{1}{N^{1/4}q^{1/2}} + \frac{1}{N \eta} \right)^{1-{\upsilon}/2 },\label{eq:improve}
	\end{align}
	uniformly on the domain $z \in \caD_{\tau}$. 
\end{remark}

\subsection{Proof of Theorem \ref{thm:refine} and Theorem \ref{thm:strong}}

We prove the entrywise local law by combining Lemma \ref{lem:self1} and Lemma \ref{lem:det}. Then, having established Theorem \ref{thm:strong}, we are free to use Lemma \ref{lem:self2} and using the improved bounds, we prove Theorem \ref{thm:refine}.

\begin{proof}[Proof of Theorem \ref{thm:strong}]
	From the weak local law, we obtain an initial bound, where we have
	\begin{align}\label{eq:weakapriori}
		\Lambda \prec \left(\frac{1}{q^2} + \frac{1}{N^{1/4} q^{1/2}} + \frac{1}{N \eta} \right)^{1-5/6}.
	\end{align}
	Since we have $\Lambda \prec 1$ for $\eta \sim 1$, we can apply Lemma \ref{lem:det} along with Lemma \ref{lem:61}, and from Lemma \ref{lem:self}, we deduce that
	\begin{align}
		\Lambda \prec \frac{1}{q} + \frac{1}{N^{1/4} q^{1/2}} + \frac{1}{N\eta},
	\end{align}
	uniformly for $z \in \caD_{\tau}$. For the entrywise bound,
	\begin{align}
		\wt{R}_i(\pair{G})  -  \wt{R}_i({\pair{\wt{\MM}}})  = (\pair{G} - {\pair{\wt\MM}}) (\wt{R}_i(\pair{G}) \wt{R}_i ({\pair{\wt{\MM}}}) + O(q^{-2}) ),
	\end{align}
	\revision{and from Lemma \ref{lem:61}, the claim is proved.}
\end{proof}

\begin{proof}[Proof of Theorem \ref{thm:refine}] Since we have established Theorem \ref{thm:strong}, we are free to use Lemma \ref{lem:self2}. From the weak local law, we have the initial bound in \eqref{eq:weakapriori}, and applying the self-improving high probability bound from the implication \eqref{eq:improve}, we obtain that for any fixed $\epsilon \in (0, \frac{\tau}{100})$, 
	\begin{align} \label{eq:sl0}
		\Lambda \leq  N^{\epsilon} \left( \frac{1}{q^{3/2}} + \frac{1}{N^{1/4}q^{1/2}} +\frac{1}{N\eta} \right),
	\end{align}
	with high probability. It remains to consider the reigme
	\begin{align}
		\frac{1}{q^{3/2}} + \frac{1}{N^{1/4} q^{1/2}} \geq N^{\epsilon} \frac{1}{q^2 } \frac{1}{|1 - \wt{\mathsf{R}}^{(2)} |}, \quad  \frac{1}{q^{3/2}} + \frac{1}{N^{1/4} q^{1/2}} \geq  \frac{1}{N\eta },\label{eq:600}
	\end{align}
	where then
	\begin{align}
		N^{-\epsilon} |1- \wt{\mathsf{R}}^{(2)}| \geq \frac{1}{q^{3/2}} + \frac{1}{N^{1/4} q^{1/2}} \geq  \frac{1}{N\eta }.\label{eq:6000}
	\end{align}
	holds. Feeding back into Lemma \ref{lem:self2} by bounding $\Lambda$ and $\Lambda^2$ using \eqref{eq:600} and \eqref{eq:6000}, and writing $\alpha = |1-\wt{\mathsf{R}}^{(2)}| $ we have
	\begin{align}
		(1- o(1)) \alpha  \Lambda  \leq N^{\frac{\epsilon}{5}} \left( \frac{1}{q^3} + \frac{1}{\sqrt{N} q} + \frac{\alpha}{q^2} + \frac{\Im \pair{\wt{\mathsf{M}}} + \alpha}{N\eta} \right),
	\end{align}
	with high probability. Clearly, we have
	\begin{align}\label{eq:sl1}
		\Lambda \prec  \frac{1}{q^2 \alpha } +\frac{1}{N\eta}.
	\end{align}
	and using a classic lattice argument of the bounds \eqref{eq:sl0} and \eqref{eq:sl1} we obtain the theorem.
\end{proof}

\section{Eigenvalue behaviors}\label{sec:behavior}
\subsection{Extremal eigenvalue bound and density of states} 

In this section, we collect results on extremal eigenvalue bound and density of states on the event $\Xi$, using the refined local law Theorem \ref{thm:refine}. The proofs are based on method in \cite{EY, EKYY1, LS}. The proofs can be found in the Appendix.

\begin{lemma}\label{thm:opnorm}
	On the event $\Xi$, for any (small) $\epsilon>0$ and (large) $D \geq 1$, we have
	\begin{align}
		\bbP \left( \norm H \norm \geq  \max\{\wt{L}_+, |\wt{L}_{-}|\} + N^{\epsilon} \left( \frac{1}{N^{2/3}}  + \frac{1}{q^3 } \right) \right) \leq N^{-D},
	\end{align}
	for any $N \geq N_0 (\epsilon, D)$.
\end{lemma}

For $\mu_{\alpha}$, the eigenvalues of $H=W+ \lambda V$ in the definition of Theorem \ref{thm:strong}, and real numbers $E_1 <E_2$, we define the counting functions 
\begin{align}
	\frn(E_1, E_2) \deq \frac{1}{N} | \{ \alpha : E_1 < \mu_{\alpha} <E_2 \}  |, \quad \wt{n}_{\fc}(E_1, E_2)  \deq \int_{E_1}^{E_2} \wt{\rho}_{\fc} (x) \d x,
\end{align}
where in this section $\wt{\rho}_{\fc}$ denotes the density of the free convolution measure. We also define the integrated density of states and the distribution function of the deformed semicircle density through
\begin{align}
	\frn (E) \deq \frac{1}{N} \{ \alpha : \mu_{\alpha} \leq E \}, \quad \wt{n}_{\fc} \deq \int_{-\infty}^{E} \wt{\rho}_{\fc} (x) \d x.
\end{align}

We have the following results on the local density of states and integrated density of states, where the proof is based on the Helffer-Sj\"{o}strand formula.

\begin{proposition}[Local density of states]\label{prop:loc} Fix any $\epsilon>0$. On the event $\Xi$, for any $E_1, E_2$ satisfying $-E_0 \leq E_1 <E_2 \leq E_0$, we have
	\begin{align}
		|\frn (E_1 , E_2) - \wt{n}_{\fc} (E_1, E_2) | \prec \frac{N^{\epsilon}}{N} +   \min \left\{ (E_2 - E_1)\left(\frac{1}{q^{3/2} } + \frac{1}{N^{1/4}q^{1/2}} \right) ,  \frac{1}{q^2} \frac{E_2 -E_1}{\sqrt{\vk + E_2 - E_1 }} \right \} ,
	\end{align}
	where $\vk \deq \min  \{ \vk_{E_1} , \vk_{E_2} \}$
\end{proposition}

\begin{proposition}[Integrated density of states]\label{prop:int}  On the event $\Xi$, for any $E \in [-E_0, E_0]$, for arbitrary fixed $\epsilon>0$ and $D>0$, we have
	\begin{align}
		\sup_{|E| \leq E_0 } |\frn(E) -\wt{n}_{\fc}(E)| \leq  N^{ {\epsilon}} &\left[ \frac{1}{N} + \frac{1}{q^{9/2}}  + \min \left\{ \vk_E \sqrt{\frac{1}{q^{3}} + \frac{1}{N^{2/3}}} , \frac{1}{q^2} \sqrt{\vk_E + \frac{1}{q^{3} } + \frac{1}{N^{2/3}}}  \right\} \right],
	\end{align}
	with probability at least $1-N^{-D}$, for $N \geq N_0 (\epsilon, D)$.
\end{proposition}
\begin{remark} \label{rem:doscases}
	Proposition \ref{prop:int} can be interpreted conveniently as follows.
	\begin{enumerate}
		\item For any $E$ such that $\vk_E \leq N^{\epsilon} \left(q^{-3} + N^{-2/3} \right)$, we have
		\begin{align}
			\left|\frn(E) -\wt{n}_{\fc}(E) \right| \prec \frac{1}{q^{9/2}} + \frac{1}{N}.
		\end{align}
		\item There exists an unique $E_*$ that depends on $N$, where $\vk_{E_*} = O(q^{-4} (q^{-3} + N^{-2/3})^{-1} + q^{-2})$ such that for any $E$ satisfying  $\vk_E \leq N^{\epsilon} \vk_{E_*}$, we have
		\begin{align}
			\left|\frn(E) -\wt{n}_{\fc}(E) \right| \prec \frac{1}{N} + \frac{1}{q^{9/2} }+ \vk_E \sqrt{\frac{1}{q^{3}} + \frac{1}{N^{2/3}} },
		\end{align}
		while in the complementary case, we have a cruder bound
		\begin{align}
			\left|\frn(E) -\wt{n}_{\fc}(E) \right| \prec \frac{1}{N} + \frac{1}{q^2} \sqrt{\vk_E + \frac{1}{q^3} + \frac{1}{N^{2/3}}}.
		\end{align}
		When $q \lesssim N^{2/9}$, $\vk_{E_*} = O(q^{-1})$, and when $q \gtrsim N^{2/9}$, $\vk_{E_*} = O(N^{2/3} q^{-4} + q^{-2})$.
	\end{enumerate}
\end{remark}

\subsection{Proof of Theorem \ref{thm:rigidity}}

In this section, we prove Theorem \ref{thm:rigidity}, which concerns eigenvalue location and rigidity.
\begin{proof}
	We focus on the eigenvalues $\mu_1, \dotsc, \mu_{N/2} $. Other eigenvalues can be treated in a similar way. Define an event $\Omega_*$ as the intersection of $\Xi$ and the events on which the estimates 
	\begin{align}
		\norm H \norm &<  \max\{\wt{L}_+, |\wt{L}_{-}|\} + N^{\frac{\epsilon}{10}} \left( \frac{1}{N^{2/3}}+ \frac{1}{q^3 } \right), \\
		\sup_{|E| \leq E_0 } |\frn(E) -\wt{n}_{\fc}(E)| &\leq  N^{ \frac{\epsilon}{10}} \left[ \frac{1}{N} + \frac{1}{q^{9/2}}  + \min \left\{ \vk_E \sqrt{\frac{1}{q^{3}} + \frac{1}{N^{2/3}}} , \frac{1}{q^2} \sqrt{\vk_E + \frac{1}{q^{3} } + \frac{1}{N^{2/3}}}  \right\} \right],
	\end{align}
	hold. (See 	
	(see Lemma \ref{thm:opnorm}) and Proposition \ref{prop:int}.) The event $\Omega_*$ is indeed a high probability event, and on $\Omega_*$, we have $\mu_{N/2} \leq K$, for some positive $K \leq \wt{L}_{+}$. Then, consider the dyadic decomposition 
	\begin{align}
		\{1, \dotsc, N/2 \} &= \bigcup_{k=0}^{2 \log N} U_k, \\
		U_0 &\deq \left\{j \leq \frac{N}{2}: |\wt{L}_{-}| + \max \{ \mu_j , \wg_j \} \leq 2 N^{\frac{4\epsilon}{10} }(q^{-3} + N^{-2/3} ) \right\}, \\
		U_k &\deq \left\{j \leq \frac{N}{2}:  |\wt{L}_{-}| + \max \{ \mu_j , \wg_j \}\in \left[2^k N^{\frac{4\epsilon}{10} } (q^{-3}+ N^{-2/3} ), 2^{k+1} N^{\frac{4\epsilon}{10} } (q^{-3}+ N^{-2/3}) \right] \right\},
	\end{align}
	Here, there exists a maximal integer $k_* <2 \log N$ such that for any $k \leq k_*$,    
	\begin{align}
		|\wt{L}_{-}| + \max \{ \mu_j , \wg_j \} \leq N^{\frac{4\epsilon}{10}} \vk_{E_*},
	\end{align} holds for all $j \in U_k$. The threshold $k_*$ will be used to distinguish the latter two cases in the statement of the theorem.
	
	Recall $E_*$ defined in Remark \ref{rem:doscases}. By the definition of $U_0$ and the extremal eigenvalue bounds, for $j \in U_0$, we have 
	\begin{align}
		|\mu_j - \wg_j| \leq N^{\epsilon} \left(  \frac{1}{q^{3}}  + \frac{1}{N^{2/3}} \right),
	\end{align}
	on the event $\Omega_*$, proving case (1).
	
	For $k \in \llbracket 1, k_* \rrbracket$, on the event $\Omega_*$ we have
	\begin{align}\label{eq:71}
		\frac{j}{N} = \wt{n}_{\fc} (\wg_j) = \frn (\mu_j)  = \wt{n}_{\fc}(\mu_j) + N^{\frac{ \epsilon }{10}} O\left( \frac{1}{N} + \frac{1}{q^{9/2}}  + \frac{{\vk_{\mu_j}}}{q^{3/2}} + \frac{{\vk_{\mu_j}}}{N^{1/3}}\right).
	\end{align}
	On $\Omega_*$, for $j \in U_k$, we have the bounds
	\begin{align}
		\frac{{\vk_{\mu_j}}}{q^{3/2}}  \leq  2^{k+1} N^{\frac{4\epsilon}{10}} \left(\frac{1}{q^{9/2}} + \frac{1}{ q^{3/2} N^{2/3}} \right), \quad \frac{{\vk_{\mu_j}}}{N^{1/3}} \leq  2^{k+1} N^{\frac{4\epsilon}{10}} \left(\frac{1}{q^{3} N^{1/3}} + \frac{1}{ N} \right),
	\end{align}
	where we have used $\vk_{\mu_j} \leq |\wt{L}_-| +\mu_j$. Furthermore, since $\wt{n}_{\fc} ( -|\wt{L}_-| + x ) \sim x^{3/2} $, for $x \in [0, |\wt{L}_-| +K]$, we have
	\begin{align}
		\wt{n}_{\fc} (\wg_{j}) + \wt{n}_{\fc}(\mu_j) \geq  c 2^{\frac{3k+2}{2}} N^{\frac{6\epsilon}{10}} \left(\frac{1}{q^{9/2}} + \frac{1}{N}  \right).
	\end{align}
	Thus,
	\begin{align}
		N^{\frac{ \epsilon }{10}} O\left( \frac{1}{N} + \frac{1}{q^{9/2}} + \frac{{\vk_{\mu_j}}}{q^{3/2}}  + \frac{{\vk_{\mu_j}}}{N^{1/3}} \right) \ll  \wt{n}_{\fc} (\wg_{j}) + \wt{n}_{\fc}(\mu_j) ,
	\end{align}
	and feeding the relations to \eqref{eq:71}, we get $\wt{n}_{\fc} (\mu_j) = \wt{n}_{\fc} (\wg_j) ( 1+o(1))$ on $\Omega_*$. 
	
	Meanwhile, for $k > k_*$, on the event $\Omega_*$ we have
	\begin{align}\label{eq:7..1}
		\frac{j}{N} = \wt{n}_{\fc} (\wg_j) = \frn (\mu_j)  = \wt{n}_{\fc}(\mu_j) + N^{\frac{ \epsilon }{10}} O\left( \frac{1}{N} + \frac{1}{q^2} \sqrt{\vk_{\mu_j} + \frac{1}{q^3 } + \frac{1}{N^{2/3}} } \right).
	\end{align}
	Similarly to the previous case, on $\Omega_*$, for $j \in U_k$, we have the bound
	\begin{align}\frac{1}{q^2} \sqrt{\vk_{\mu_j} + \frac{1}{q^3} + \frac{1}{N^{2/3}}}  &\leq  2^{\frac{k+3}{2}} N^{\frac{2\epsilon}{10}} \left(\frac{1}{q^{7/2}} + \frac{1}{ q^{2} N^{1/3}} \right)\leq  3 \cdot 2^{\frac{k+2}{2}} N^{\frac{2\epsilon}{10}} \left(\frac{1}{q^{3}} + \frac{1}{N} \right), \\
		\wt{n}_{\fc} (\wg_{j}) + \wt{n}_{\fc}(\mu_j) &\geq  c N^{\frac{6\epsilon}{10}} \left(   \frac{1}{q^{3/2}} \wedge \left( \frac{N}{q^6 } + \frac{1}{q^{3}} \right)   +  2^{\frac{3k+2}{2}} \left( \frac{1}{q^{9/2}} + \frac{1}{N} \right) \right).\label{eq:casebycase}
	\end{align}
	We obtain lower bound on \eqref{eq:casebycase} by dividing into $q\lesssim N^{2/9}$ and $N^{2/9} \lesssim q$. When $q\lesssim N^{2/9}$, we easily find 
	\begin{align}
		\wt{n}_{\fc} (\wg_{j}) + \wt{n}_{\fc}(\mu_j) &\geq c N^{\frac{6 \epsilon}{10}} \left( \frac{2^{\frac{3k+2}{6} -1} }{q^{5/2}}  +2^{\frac{3k+2}{2}}  \frac{1}{N} \right),
	\end{align}
	by applying Young's inequality. When $N^{2/9} \lesssim q $, we have
	\begin{align*}
		\wt{n}_{\fc} (\wg_{j}) + \wt{n}_{\fc}(\mu_j) &\geq  c N^{\frac{6\epsilon}{10}} \left(     \frac{1}{q^3}   +  (2^{\frac{3k+2}{2}}-1) \left( \frac{1}{q^{9/2}} + \frac{1}{N} \right) \right) \geq c N^{\frac{6 \epsilon}{10}} \left( 2^{\frac{3k+2}{6} -1} \left(  \frac{1 }{q^{7/2}}+ \frac{1}{q^{2}N^{1/3}} \right)  +2^{\frac{3k+2}{2}}  \frac{1}{N} \right).
	\end{align*}
	In conclusion, we have
	\begin{align}
		N^{\frac{ \epsilon }{10}} O\left( \frac{1}{q^2} \sqrt{\vk_{\mu_j} + \frac{1}{q^3} + \frac{1}{N^{2/3}}}\right) \ll  \wt{n}_{\fc} (\wg_{j}) + \wt{n}_{\fc}(\mu_j)  ,
	\end{align}
	and feeding the relations to \eqref{eq:7..1}, we get $\wt{n}_{\fc} (\mu_j) = n_{\fc} (\wt{\gamma}_j) ( 1+o(1))$ on $\Omega_*$. Using that $\wt{n}_{\fc}'(x) \sim (\wt{n}_{\fc}(x))^{1/3} \sim (|\wt{L}_-| +x)^{1/2}$, for $x \in [\wt{L}_{-},K ]$, we have $|\wt{L}_{-}|+\wg_j \sim |\wt{L}_{-} | + \mu_j $. Hence, we have $\wt{n}'_{\fc} (x) \sim \wt{n}'_{\fc} (\wt{\gamma}_j)$ for any $x$ between $\wt{\gamma}_j$ and $\mu_j$. 
	
	To prove case (2), recalling that the density $\wt{\rho}_{\fc}$ is continuous, on the event $\Omega_*$, for $j \in U_{k}$ with $k \leq k_*$, we have
	\begin{align}
		|\mu_j - \wg_j| & \leq C \frac{|\wt{n}_{\fc} ({\mu}_{j}) - \wt{n}_{\fc} (\wg_j)|}{\wt{n}'_{\fc} (\wt{\gamma}_j) }\leq  \frac{CN^{\frac{\epsilon}{10}}}{(j/N)^{1/3}} \left( \frac{1}{N} + \frac{1}{q^{9/2}}  + \frac{\vk_{\mu_j}}{q^{3/2}}+ \frac{{\vk_{\mu_j}}}{N^{1/3}}  \right) \\
		&\leq \frac{CN^{\frac{\epsilon}{10}}}{j^{1/3}} \Bigg(\frac{1}{N^{2/3}} + \frac{N^{1/3}}{q^{9/2}} + \frac{j^{2/3}}{N^{1/3}} \left( \frac{1}{q^{3/2}} + \frac{1}{N^{1/3}} \right) + \left( 1+ \frac{N^{1/3}}{q^{3/2}} \right)|\mu_j - \wt{\gamma}_j| \Bigg) ,\label{eq:inter}
	\end{align}
	where the mean value theorem was used in the first step, and $\vk_{\mu_j} \leq \vk_{\wt{{\gamma}}_j} + |\mu_j -\wt{\gamma}_j|$ and $\vk_{\wt{\gamma}_j} \sim (j/N)^{2/3}$ was used in the last step. Next, since $j = N \wt{n}_{\fc} (\wt{\gamma}_j) \sim N(|\wt{L}_-| + \wt{\gamma}_j  )^{3/2} $, for $j \in U_k$ with $k\in \llbracket 1, k_* \rrbracket$, we find
	\begin{align}
		j \geq c N \left( 2^{k} N^{\frac{4\epsilon }{10}} \left( \frac{1}{q^3} + \frac{1}{N^{2/3}} \right) \right)^{3/2} \gg N^{\frac{\epsilon}{2}} \left(\frac{N}{q^{9/2}} +1 \right).
	\end{align}
	Hence, we can absorb the last term on the right hand side of \eqref{eq:inter} into the left hand side and we obtain
	\begin{align}
		|\mu_j - \wg_j| \leq N^{{\epsilon}} \left[ \frac{1}{j^{1/3}} \left(\frac{1}{ N^{2/3}} +\frac{N^{1/3}}{q^{9/2}} \right)+ \frac{j^{1/3}}{N^{1/3}} \left( \frac{1}{q^{3/2}} + \frac{1}{N^{1/3}}  \right) \right],
	\end{align}
	on the event $\Omega_*$. This proves \eqref{eq:each1}.
	
	Likewise, for case (3), on the event $\Omega_*$, for $j \in U_k$ with $k > k_*$, we have
	\begin{align}
		|\mu_j - \wg_j|  \leq C \frac{|\wt{n}_{\fc} ({\mu}_{j}) - \wt{n}_{\fc} (\wg_j)|}{\wt{n}'_{\fc} (\wt{\gamma}_j) }
		&\leq  \frac{CN^{\frac{\epsilon}{10}}}{(j/N)^{1/3}} \left( \frac{1}{N} + \frac{1}{q^2} \sqrt{\vk_E + \frac{1}{q^3} + \frac{1}{N^{2/3}}} \right) \\
		&\leq  \frac{CN^{\frac{\epsilon}{10}}}{j^{1/3}} \left(\frac{1}{N^{2/3}} + \frac{N^{1/3}}{q^{7/2}} + \frac{j^{1/3}}{q^2} + \frac{N^{1/3} \sqrt{|\mu_j - \wg_j |}}{q^2} \right).\label{eq:7100}
	\end{align}
	For $k \geq k_*$, we have $j \gg Nq^{-3/2}$ when $q \lesssim N^{2/9}$, while $j \gg N^{2}q^{-6} + N^{-1} + Nq^{-3} \geq 2\sqrt{2} N^{1/2} q^{-3}$ when $q \gtrsim N^{2/9}$. In either cases, we can absorb the last term on the right side of \eqref{eq:7100} into the left side. Then, we achieve the bound
	\begin{align}
		|\mu_j - \wg_{j} | &\leq N^{\epsilon}  \left[ \frac{1}{j^{1/3}} \left( \frac{1}{N^{2/3}} + \frac{N^{1/3}}{q^{7/2}}  \right) + \frac{1}{q^2} \right],
	\end{align}
	which is \eqref{eq:each2} with minor adjustments.
\end{proof}

\subsection{Proof of Theorem \ref{thm:gauss}}\label{sec:gaussian}

By following the methods in \cite[Lemma 4.8]{LS14}, we prove Gaussian fluctuation of the endpoint of $\wt{L}_+$ in case $V$ is random, with $\lambda \gg q^{-1}$. Note that when $V$ is a random matrix, $\wt{m}_{\fc}$ is a random variable, while when $V$ is a deterministic matrix, $\wt{m}_{\fc}$ is also deterministic.
\begin{lemma} \label{lem:gaussian} If $\lambda \gg q^{-1}$, the rescaled fluctuation $N^{1/2} \lambda^{-1} (\wt{L}_+ - \breve{L}_+)$ converges in distribution to a Gaussian random variable with mean $0$ and variance 
	\begin{align}
		\lim_{N \to \infty} \lambda^{-2} (1 - \mfc(L_+)^2 ),
	\end{align}
	as $N \to \infty$. In particular, if $\beta>0$, the variance is $m^{(2)}(\nu)$.
\end{lemma}
\begin{proof}
	Following the proof in \cite{LS}, we find that $\wt{L}_{\pm}, \breve{L}_{\pm}, L_{\pm}$ are the solutions to the equations
	\begin{align}
		\wt{m}_{\fc} (\wt{L}_{\pm}) = \int \frac{\d \wt{\nu}(v)}{\lambda v - \wt{L}_{\pm} - \wt{m}_{\fc}(\wt{L}_{\pm})}, \quad & 1  = \int \frac{\d \wt{\nu}(v)}{(\lambda v - \wt{L}_{\pm} - \wt{m}_{\fc}(\wt{L}_{\pm}))^2 }, \\
		\breve{m}_{\fc} (\breve{L}_{\pm}) = \int \frac{\d \breve{\nu}(v)}{\lambda v - \breve{L}_{\pm} - {m}_{\fc}(\breve{L}_{\pm})}, \quad & 1  = \int \frac{\d \breve{\nu}(v)}{(\lambda v - \breve{L}_{\pm} - {m}_{\fc}(\breve{L}_{\pm}))^2 } \label{eq:int1}\\
		{m}_{\fc} ({L}_{\pm}) = \int \frac{\d {\nu}(v)}{\lambda v - {L}_{\pm} - {m}_{\fc}({L}_{\pm})}, \quad &  1  = \int \frac{\d {\nu}(v)}{(\lambda v - {L}_{\pm} - {m}_{\fc}({L}_{\pm}))^2 }. \label{eq:int11}
	\end{align}
	We focus on the upper edge. Let 
	\begin{align}
		\wt{\zeta}_{+} = \wt{L}_+ + \wt{m}_{\fc} (L_+), \quad   \breve{\zeta}_{+} \deq \breve{L}_{+} + \breve{m}_{\fc} (\breve{L}_+) , \quad \zeta_+ \deq L_+ + \mfc(L_+).
	\end{align}
	Recall that for $V$ random, we are working on $\Xi_N (1/2 -\epsilon_0)$, where then we have the bounds
	\begin{align}
		\int \frac{\d \breve{\nu}(v)}{(\lambda v - \breve{\zeta}_{+})^2 } - \int \frac{\d \wt{\nu}(v)}{(\lambda v- \breve{\zeta}_{+})^2 } = O_{\prec} \left(\frac{\lambda}{\sqrt{N}} \right) ,\quad
		\breve{\zeta}_{+} - \wt{\zeta}_{+} = O_{\prec} (\lambda N^{-1/2} ).
	\end{align}
	For proofs, we refer to the proof of Lemma \ref{lem:stabbound} in the Appendix regarding the remainder, and \cite[Lemma 4.8]{LS14}. Notice that the bounds are stronger than the deterministic case, due to the additional assumption given for random $V$ on the event $\Xi$.
	Finally, we can write 
	\begin{align}
		\wt{\zeta}_{+} -\wt{L}_+  =  \wt{m}_{\fc} (\wt{L}_+)  = \int \frac{\d \wt{\nu}(v)}{\lambda v - \wt{\zeta}_{+}}&  = \int  \frac{ \d\wt{\nu}(v) }{\lambda v - \breve{\zeta}_{+}} + \int  \frac{ (\wt{\zeta}_{+} - \breve{\zeta}_{+}) }{(\lambda v - \breve{\zeta}_{+})^2 }\d\wt{\nu}(v)  + O_{\prec} (\lambda^2 N^{-1}) \\
		&=( \breve{\zeta}_{+} - \breve{L}_+) + \caX + (\wt{\zeta}_{+} - \breve{\zeta}_{+} ) + O_{\prec} (\lambda^2 N^{-1}) ,
	\end{align}
	on $\Xi$, where we define the random variable $\caX$ by
	\begin{align}
		\begin{split}
			\caX \deq \int \frac{\d \wt{\nu}(v)}{\lambda v - \breve{\zeta}_{+}} -  \int \frac{\d \breve{\nu}(v)}{\lambda v - \breve{\zeta}_{+}}= \frac{1}{N}& \sum_{i=1}^{N}\frac{1}{2}  \Bigg[ \frac{1}{\lambda V_i - \breve{\zeta}_{+} - \sqrt{s}q^{-1} }+ \frac{1}{\lambda V_i - \breve{\zeta}_{+} + \sqrt{s}q^{-1} } \Bigg. \\ &  \Bigg.- \E \left(\frac{1}{\lambda V_i - \breve{\zeta}_{+} - \sqrt{s}q^{-1} } +  \frac{1}{\lambda V_i - \breve{\zeta}_{+} + \sqrt{s}q^{-1} } \right)  \Bigg].
		\end{split}
	\end{align}
	By the central limit theorem, $\caX$ converges to a Gaussian random variable with mean $0$ and variance $N^{-1} \sigma_N^2$ with 
	\begin{align}
		\sigma _N^2 &= \int \left[ \frac{1}{2} \left(\frac{1}{\lambda v - \breve{\zeta}_{+} -\sqrt{s}q^{-1} } + \frac{1}{\lambda v - \breve{\zeta}_{+} + \sqrt{s}q^{-1} }  \right) \right] ^2 \d \nu(v) - \left[ \int \frac{1}{\lambda v - \breve{\zeta}_{+}} \d \breve{\nu}(v) \right]^2  \\
		&= \frac{1}{2} \int \frac{1}{(\lambda v - \breve{\zeta}_{+})^2 }\d \breve{\nu}(v) +\frac{1}{2} \int \frac{1}{(\lambda v - \breve{\zeta}_{+})^2 - sq^{-2}} \d {\nu}(v) - \breve{m}_{\fc}(\breve{L}_+)^2  = 1 - \bv{m}_{\fc}(\bv{L}_+)^2 + O( q^{-2}).
	\end{align}
	It remains to prove the given variance coincides with the limit given in the theorem, and we provide the details in the Appendix.
\end{proof}

\begin{proof}[Proof of Theorem \ref{thm:gauss}]
	By combining the results in Lemma \ref{thm:opnorm} and Lemma \ref{lem:gaussian}, when 
	\begin{align}
		\lambda \gg \max\{q^{-1}, N^{-1/6} , \sqrt{N}q^{-3}\} ,
	\end{align} the rescaled fluctuation $N^{1/2}{\lambda^{-1}} (\mu_N - \breve{L}_+)$ converges in distribution to the Gaussian random variable in Lemma \ref{lem:gaussian}.
\end{proof}
\section*{Acknowledgements}
The work of Ji Oon Lee was partially supported by National Research Foundation of Korea under grant number NRF-2019R1A5A1028324 and RS-2023-NR076695.

\appendix
\section{The high probability event $\Xi$ and stability of the refined measures}
\subsection{Proof of Lemma \ref{lem:JL}}\label{Appendix_Xi}
\begin{proof}[Proof of Lemma \ref{lem:JL}]

The pointwise convergence rate requirement, which holds for each $z \in \bbC^+$ can be easily proved referring to the Chernoff bounds in \cite[Appendix E]{JL}. Therefore, we will only focus on the uniform bound on each domain $\caD$.

Let us first fix an integer $N>0$ and compute the probability of the event on which (17) does not hold. Fix a sequence $\delta_j = \epsilon_0/j$. Define the event $\mathsf{E}_N^{(j)} $ as follows: on the largest compact set $\caD_* \subset \caD' \subset \bbC^+$ that satisfies $\mathrm{dist}\, (\supp \nu, \caD) + \Im z \geq N^{-\fra}$, there exists an integer $N_*$ greater than or equal to $N$ such that,  
\begin{align}
    \sup_{z \in \caD_*} |m_{\wh{\nu}_{N_*}} - m_{\nu} (z)| \geq 
    {N_*}^{-1/2 + \epsilon_{j}},
\end{align}
where $\epsilon_{j} \deq \epsilon_0 - \delta_j$. Clearly, $\mathsf{E}_N^{j}$ is a decreasing event and $\mathsf{E}_N = \cap_j \mathsf{E}_N^{(j)}$ is the desired complementary event.

Let us define the event $\mathsf{E}_N^{(j)}(M)$ for $M \geq N$ as the complementary event stated above with $N_* =M$, then clearly, $\mathsf{E}_N^{(j)} \subset \cup_{M = N}^{\infty} \mathsf{E}_N^{(j)}(M) $.

Notice that 
\begin{align}
    |m_{\wh{\nu}}'(z)| \leq \int \frac{1}{|x- z|^2} \d \wh{\nu} (x) \leq \frac{1}{\mathrm{dist}\, (\supp \nu, \caD)^2 + (\Im z)^2} \leq \frac{1}{2} 
    {M^{2\fra}},
\end{align}
and likewise $|m'_{\nu}(z) | \leq \frac{1}{2} M^{2 \fra}$, so that
\begin{align}
    |m_{\wh{\nu}}(z) - m_{\nu}(z)|\leq \frac{M^{2 \fra}}{2} |z - z'|.
\end{align}

Let $\overline{\caD_N}$ be a lattice in $\bbC^+$ with $|\overline{\caD_N}| \leq C_2 N^{1 - \epsilon_{j } +2\fra}$ (where $C_2$ is chosen an absolute constant) such that for any $z \in \caD$, $\inf_{z' \in \overline{\caD}_N} |z- z'| \leq \frac{2}{3}M^{-2\fra  } M^{-1/2+\epsilon_{j}} $. Then,
\begin{align}
     \bbP \left[ \sup_{z\in \caD_* }|m _{\wh{\nu}} (z) - m_{\nu} (z) | \geq {M^{ -1/2 + \epsilon_{j} }}   \right] \leq \sum_{z \in \overline{\caD_N}} \bbP \left[ |m_{\wh{\nu} }(z) -m_{\nu}(z)| \geq \frac{M^{-1/2 + \epsilon_{j}}}{3}  \right]
\end{align}

For each fixed $z \in \caD$, applying the Chernoff inequality, we obtain

\begin{align}
    \bbP \left[ |m _{\wh{\nu}} (z) - m_{\nu} (z) | \geq \frac{M^{ -1/2 + \epsilon_{j} }}{3}   \right] = \bbP \left[ \left|\sum_{i=1}^M \left( \frac{1}{V_i -z}   - \E \left[ \frac{1}{V_i -z} \right]  \right) \right| \geq \frac{M^{ -1/2 + \epsilon_{j} }}{3} \right] \leq 4\exp \left( -\frac{ M^{  2\epsilon_{j}} }{S } \right) ,
\end{align}
where $S$ is an absolute constant that depends on the compact set $\caD_*$ (and the $M$-independent constant that bounds $V$).

Therefore, we obtain
\begin{align}
    \bbP \left[ \sup_{z \in \caD_*} |m _{\wh{\nu}} (z) - m_{\nu} (z) | \geq {M^{ -1/2 + \epsilon_{j} }}  \right] \leq C M^{1- \epsilon_{j} + 2 \fra}\exp \left(-\frac{M^{2 \epsilon_{j}}}{S} \right) \leq CM^{2} \exp \left(-\frac{M^{2 \epsilon_{j}}}{S} \right) 
\end{align}

Finally, we compute
\begin{align}
    \bbP (\mathsf{E}_N^{(j)} ) \leq \bbP\left(\bigcup_{M = N}^{\infty} \mathsf{E}_{N}^{(j)}(M)  \right)\leq \sum_{M=N}^{\infty} \bbP(\mathsf{E}_N^{(j)} (M)) \leq \sum_{M=N}^{\infty} C M^{2} \exp(-M^{2\epsilon_j}/S) \leq C_jN^{-\frt}, 
\end{align}
and from the continuity of measure from above, $\bbP(\mathsf{E}_N) = \lim_{j\to \infty} \bbP(\mathsf{E}_N^{(j)}) \leq CN^{-\frt} $, for some $\frt >0$.
\end{proof}

	\subsection{Stability of the refined measures}\label{appendix1}
	In this section, we prove Proposition \ref{prop:refv} \eqref{item:prop1}, which asserts that the refined measure also satisfies stability. We note that Lemma \ref{lem:stabbound} is not invoked in the proof.
	\begin{proof}[Proof of Proposition \ref{prop:refv} \eqref{item:prop1}] In what follows, denote $\wt{m}_{\wh{\nu}_N}$ as the Stieltjes transform of the probability measure $\wt{\nu}_N^{\lambda}$. {Recall that we work on the event $\Xi$ and so the entries of $V$ are fixed.} Let the lower and upper endpoints of $I_{\wh{\nu}_N}$ be $\wh{r}_N^{-}$ and $\wh{r}_N^{+}$, respectively. {Without loss of generality, we may assume that $\limsup_{N \to \infty} \max \{|\wh{r}_N^{-}|, |\wh{r}_N^{+}| \} \leq 1$,} and that $\wh{\nu}_N$ and $\nu$ are supported on a compact set. Then, $\wt{\nu}_N$ is clearly supported on $ I_{\wt{\nu}_N} = \wt{I}_{\lambda, q} \deq \cup_{j =-1 , 0, 1} ( I_{\wh{\nu}_N} + j \sqrt{s} q^{-1}\lambda^{-1})$, where the endpoints of the interval are $\wh{r}_N^{-} - \sqrt{s}q^{-1} \lambda^{-1}$ and $ \wh{r}_N^{+} + \sqrt{s}q^{-1} \lambda^{-1}$. We partition $N$ into two classes
		\begin{align}
			S^{(1)} &\deq \{ N\in \bbN: (\wh{r}_N^{+} - \wh{r}_N^{-})/2 \geq \sqrt{s}q^{-1} \lambda^{-1} \},\\
			S^{(2)} & \deq \{ N \in \bbN: (\wh{r}_N^{+} - \wh{r}_N^{-})/2 < \sqrt{s}q^{-1} \lambda^{-1}\}.
		\end{align}
		Since $\wh{r}_N^+ - \wh{r}_N^{-} = O(1)  $, if $q^{-1} \lambda^{-1} \ll 1 $, only finitely many $N$ belong to $S^{(2)}$. Therefore, it suffices to consider $N$ that belongs to $S^{(2)}$ when $q^{-1} \lambda^{-1} $ does not converge to $0$. In what follows, $N$ is the index of the increasing sequence belonging to $S^{(1)}$ or $S^{(2)}$.
		
		\noindent \textbullet \, Case 1: $N$ belonging to $S^{(1)}$
		
		We partition $\wt{I}_{\lambda, q}$ into three intervals; $\wt{I}^{-}_{\lambda, q} = [\wh{r}_N^{-} - \sqrt{s}q^{-1} \lambda^{-1}, \lambda \wh{r}_N^{-} + \sqrt{s}q^{-1}\lambda^{-1}]$, $\wt{I}^{0}_{\lambda, q} = [ \wh{r}_N^{-} + \sqrt{s}q^{-1} \lambda^{-1}, \lambda \wh{r}_N^+ - \sqrt{s}q^{-1}\lambda^{-1}]$, and $\wt{I}^{+}_{\lambda, q} = [ \wh{r}_N^+ - \sqrt{s}q^{-1} \lambda^{-1}, \wh{r}_N^+ + \sqrt{s}q^{-1}\lambda^{-1}]$. 
		
		From Assumption \ref{assm:stableint}, we know that there exist $\varpi>0$ such that 
		\begin{align}
			\inf_{x \in I_{\wh{\nu}_N}+\frac{\sqrt{s}}{q\lambda}} \int \frac{\d \nu(v - \sqrt{s} q^{-1} \lambda^{-1} )}{\lambda^2( v - x)^2}  \geq 1+ \varpi, \quad  \inf_{x \in I_{\wh{\nu}_N}-\frac{\sqrt{s}}{q\lambda}} \int \frac{\d \nu(v + \sqrt{s} q^{-1} \lambda^{-1} )}{\lambda^2 ( v - x)^2}  \geq 1+ \varpi,
		\end{align} and so we easily verify
		\begin{align}
			\inf_{x \in \wt{I}^{0}_{\lambda, q}} \int \frac{\d \wt{\nu}_N(v) }{\lambda^2 ( v - x)^2} \geq 1+ \varpi.
		\end{align}
		Turning to the other intervals, we bound the non-trivial portion by
		\begin{align} \label{eq:inf}
			\inf_{x \in I_{\wt{\nu}}  \cap  \wt{I}_{\lambda, q}^+} \int \frac{\d \wh{\nu}_N (v + \sqrt{s} q^{-1} \lambda^{-1} )}{(\lambda v - x)^2} &\geq \frac{1}{\lambda^2} \int \frac{\d \wh{\nu}_N(v)}{(v- \wh{r}_N^+ - 2 \sqrt{s} q^{-1} \lambda ^{-1} )^2 } \\
			&\geq \frac{1}{\lambda^2} \frac{1}{(\int (v - \wh{r}_N^+ - 2 \sqrt{s} q^{-1} \lambda^{-1} ) d \wh{\nu}_N(v))^2  } \\
			& = \frac{1}{\lambda^2} \frac{1}{(\wh{r}_N^+ + 2 \sqrt{s}q^{-1} \lambda^{-1})^{2} },
		\end{align}
		where Jensen's inequality was used in the second line.
		If $\beta >0$, the last line becomes arbitrarily large as $N \to \infty$. If $\beta =0 $, taking $N \to \infty$, the last line converges to $(\lambda r^{+})^{-2} \geq \lambda^{-2} \geq 1$. Finally, if $\lambda \geq 1$, we have that $\sqrt{s}q^{-1} \lambda^{-1}$ converges to $0$ as $N \to \infty$. 
		From the continuity of 
		\begin{align}
			{H}(z) \deq \int \frac{\d \nu(v)}{(z-v)^2}, \quad z \in \bbC^+ \cup (\bbR \setminus \supp \nu), 
		\end{align} we can find $\varepsilon$ (independent of $N$) such that whenever $u \in (0, 2\varepsilon)$, we have
		\begin{align} \label{eq:cont}
			\int \frac{\d \nu(v)}{\lambda^2 (r^{+}+u -v)^2} > 1, \quad \int \frac{\d \nu(v)}{\lambda^2 (r^{-}-u -v)^2} > 1.
		\end{align}
		Then, there exist $N_0$ such that for any $N > N_0$, $\wh{r}_N^{+} + 2\sqrt{s}q^{-1} \lambda^{-1} - r^+ < \varepsilon$. Returning to \eqref{eq:inf}, we obtain
		\begin{align}
			\frac{1}{\lambda^2} \int \frac{\d \wh{\nu}_N(v)}{(v- \wh{r}_N^+ - 2 \sqrt{s} q^{-1} \lambda ^{-1} )^2 } \geq \frac{1}{\lambda^2} \int \frac{\d \wh{\nu}_N(v)}{(r^+ + \varepsilon - v )^2 }.
		\end{align}
		Notice that $v \mapsto \lambda^{-2} (r^{+} + \varepsilon -v)^{-2}$ is continuous and bounded on $\supp \wh{\nu}_N$, and from the definition of weak convergence, we have 
		\begin{align}
			\int \frac{\d \wh{\nu}_N (v) }{\lambda^2 (r^{+} + \varepsilon -v)^2 } \rightarrow \int \frac{\d {\nu} (v) }{\lambda^2 (r^{+} + \varepsilon -v)^2 }, \quad \mathrm{as} \,\, N \to \infty.
		\end{align}
		It immediately follows from \eqref{eq:cont} that for large enough $N$, we have
		\begin{align}
			\frac{1}{\lambda^2} \int \frac{\d \wh{\nu}_N(v)}{(v- \wh{r}_N^+ - 2 \sqrt{s} q^{-1} \lambda ^{-1} )^2 } >1. 
		\end{align}
		In all these cases, we deduce that there exist some ${\varpi}' > 0$ such that 
		\begin{align}
			\inf_{x \in \wt{I}^{+}_{\lambda, q}} \int \frac{\d \wt{\nu}_N(v) }{\lambda^2 ( v - x)^2} =  \inf_{x \in \wt{I}^{+}_{\lambda, q}} \left[ \frac{1}{2} \int \frac{ \d \wh{\nu}_N\left(v - \frac{\sqrt{s}}{q \lambda} \right) }{\lambda^2 ( v - x)^2} +  \frac{1}{2}\int \frac{ \d \wh{\nu}_N\left(v + \frac{\sqrt{s}}{q \lambda} \right)}{\lambda^2 ( v - x)^2} \right] \geq 1+ {\varpi}',
		\end{align}
		for $N$ sufficiently large. 
		
		Analogous statement holds for the lower endpoint of the support.
		
		\noindent \textbullet \, Case 2: $N$ belonging to $S^{(2)}$
		
		In this subsequence, it is only required to consider $\lambda \leq \frac{1}{2} \sqrt{s}  (\wh{r}_N^{+}-\wh{r}_N^{-}) q^{-1}$, which implies that in $N \to \infty$ limit, $\lambda$ must converge to $0$, that is, $\delta >0$. While it is possible to provide a better bound, the following crude bound suffices. Since there are no overlap between the supports $\wh{\nu}_N(v- \sqrt{s} q^{-1} \lambda^{-1})$ and $\wh{\nu}_N (v + \sqrt{s} q^{-1} \lambda^{-1})$, for any $x \in \wt{I}_{N}$, $x$ is not contained in at least one of $\wh{\nu}_N(v - \sqrt{s}q^{-1}\lambda^{-1})$ or $\wh{\nu}_N(v+\sqrt{s}q^{-1} \lambda^{-1})$; without loss of generality let us assume that it is not contained in the former. Then, we have
		\begin{align}
			\int \frac{\d \wt{\nu}_N}{\lambda^2 (v-x)^2} \geq \frac{1}{2} \int \frac{\d\wh{\nu}_N(v - \sqrt{s}q^{-1}\lambda^{-1})}{\lambda^2(v-x)^2} \geq  \frac{1}{2 \lambda^2 (\wh{r}_N^+ - \wh{r}_N^{-} + 2 \sqrt{s}{q}^{-1} \lambda^{-1})^2},
		\end{align}
		where the right hand side diverges in the limit $N \to \infty$, and so for sufficiently large $N$, Assumption \ref{assm:stableint} is satisfied.
		
	{The statement for $\breve{\nu}$ also follows the proof for $\wt{\nu}$ presented above.}
	\end{proof}
	
	Next, we provide the details in the proof of Lemma \ref{lem:stabbound}. The proof follows the proof of Lemma A.1 in \cite{LSSY}.
	\begin{proof}[Proof of Lemma \ref{lem:stabbound}] 
	 Recall that in Proposition \ref{prop:refv}, we have proved that, on the event $\Xi$, there exist $\varpi' > 0$ such that for all $N$ sufficiently large, 
		\begin{align}
			\inf_{x \in I_{\wt{\nu}^{\lambda}}} \int \frac{\d \wt{\nu} (v) }{(\lambda v - x)^2} \geq 1 + \varpi',
		\end{align}
		uniformly in $\lambda \in \Theta_{\varpi}$ and $s \in \Theta_{s^{(4)}}$. Consider the function
		\begin{align}
			\wt{H}(\zeta) \deq \int \frac{\d \wt{\nu}(v)}{(\lambda v - {\zeta})^2},
		\end{align} 
		which is a decreasing continuous function outside $I_{\wt{\nu}^{\lambda}}$. There are two real solutions $\wt{\zeta}_{\pm}$ to $\wt{H} (\zeta) =1$ outside of the interval $I_{\wt{\nu}}$. Moreover, there exist a constant $\frg >0$, depending only on $\nu$ (thus uniform on $\lambda \in \Theta_{\varpi}$, the choice of $s = Nq^{-2} \caC_4\in \Theta_{s^{(4)}}$ and $N$), such that
		\begin{align}
			\inf_{\lambda \in\Theta_{\varpi}, s \in \Theta_{s^{(4)}}} \mathrm{dist} (\{\wt{\zeta}_{\pm}\} , {I}_{\wt{\nu}_N^{\lambda}}) \geq \frg,\label{eq:g}
		\end{align}
		for $N$ sufficiently large. 
		
		The assertions \eqref{eq:stab1}, \eqref{eq:stab2}, \eqref{eq:stab3}, \eqref{eq:stab4} hold when replacing $\nu$ by $\wt{\nu}$ and $\mfc$ with $\wt{m}_{\fc}$, and other quantities appropriately. The same argument works for $\bv{\nu}$ and $\bv{m}_{\fc}$.  The proofs rely on the relation $\eqref{eq:g}$. For more details we refer to the Appendix of \cite{LS, LSSY}.
		
		It remains to prove \eqref{eq:eqeq}
		Define $\wt{\zeta } \deq z +\wt{m}_{\fc}(z) $ and $\breve{\zeta} \deq z + \breve{m}_{\fc}(z) $. Furthermore, define for $w \in \bbC^+$,
		\begin{align}
			\wt{\FF}(w) &\deq w-\int \frac{\d \wt{\nu}(v)}{\lambda v - w} = w- \int \frac{\d \wh{\nu}(v)}{\lambda v - w - \frac{s}{q^2} \frac{1}{\lambda v - w}} , \\
			\bv{\FF}(w) &\deq w-\int \frac{\d \bv{\nu}(v)}{\lambda v - w} = w- \int \frac{\d {\nu}(v)}{\lambda v - w - \frac{s}{q^2} \frac{1}{\lambda v - w}},
		\end{align}
		where we have $\wt{\FF}(\wt{\zeta}) =\bv{\FF}(\bv{\zeta}) = z$ by definition. Using the stability bounds and \eqref{eq:bound} in the definition of $\Xi_N$, and assuming that $|\wt{\zeta} - \bv{\zeta}| \ll 1 $, we have
		\begin{align}
			[\bv{\FF}(\bv{\zeta}) + O(N^{-\alpha_0}) ] (\wt{\zeta} - \bv{\zeta}) + \frac{\bv{F} (\bv{\zeta})}{2} (\wt{\zeta} - \bv{\zeta})^2 = o(1) (\wt{\zeta} - \bv{\zeta})^2 + O(N^{-\fra_0}), 
		\end{align}
		uniformly in $\lambda \in \Theta_{\varpi}$, {and in $Nq^{2} \caC_4 = s^{(4)}\in \Theta_{s^{(4)}}$, for all $z \in \caD'$.} Then, from $\bv{\FF}( \bv{\zeta})  \sim \sqrt{\vk + \eta} $ and $\bv{{\FF}}''(\bv{\zeta}) \leq C$, {for all $z \in \caD'$}. In what follows, we abbreviate $\Gamma \deq |\wt{\zeta} - \bv{\zeta}|$. 
		
		First, we consider $z = E+ \ii \eta \in \caD_{\tau}$, such that ${\vk_{E} + \eta} > N^{-\epsilon}$, for some small $\epsilon >0$, (with $\epsilon < \alpha_0)$. For such $z$, we obtain from the expansion that $\Gamma \leq CN^{\epsilon} (\Gamma^2 + N^{-\fra_0})$. Then, either $\Gamma \leq C_0 N^{\epsilon} N^{-\alpha}$ or $\Gamma \geq CN^{-\epsilon}$, for some constant $C_0$. To extend the bound for all $\eta$, we use the Lipschitz continuity of $\breve{\zeta}$ and $\wt{\zeta}$. Since $\pd_{z} \bv{\zeta} = (\bv{\FF}' (\wt{\zeta}))^{-1}$ and $\pd_{z} \wt{\zeta} = (\wt{\FF}' (\wt{\zeta}))^{-1}$, holds, we infer that the Lipschitz constant of $\bv{\zeta}$ and $\wt{\zeta}$ is $N^{\epsilon/2}$ for $z \in \caD_{\tau}$ satisfying $\vk_{E} + \eta > N^{-\epsilon}$. Then, by bootstrapping, we obtain
		\begin{align}\label{eq:ap1}
			|\wt{\zeta}(z) - \bv{\zeta} (z)| \leq CN^{\epsilon} N^{-\fra_0},
		\end{align}
		on $\Xi$ for $N$ sufficiently large, for all $z$ satisfying $\vk_{E} + \eta >N^{-\epsilon} $. 
		
		To control $\wt{\zeta} - \zeta$ for $z \in \caD'$ with $\vk_E + \eta \leq N^{-\epsilon}$, we first estimate
		\begin{align}
			\begin{split}
				\wt{\zeta} - \bv{\zeta} = & ( \wt{\zeta} - \bv{\zeta} ) \int \frac{1 + \frac{s}{q^2} \frac{1}{(\lambda v - \wt{\zeta}) (\lambda v- \bv{\zeta})}}{(\lambda v - \wt{\zeta} - \frac{s}{q^2} \frac{1}{\lambda v- \wt{\zeta}})(\lambda v - \wt{\zeta} - \frac{s}{q^2} \frac{1}{\lambda v- \bv{\zeta}})} \d \wh{\nu}(v)\\
				&+ \left[ \int \frac{\d \wh{\nu}(v) }{\lambda v - \bv{\zeta} - \frac{\sqrt{s}}{q} }  +  \int \frac{\d \wh{\nu}(v) }{\lambda v - \bv{\zeta} +\frac{\sqrt{s}}{q} }  \right] -\left[ \int \frac{\d {\nu}(v) }{\lambda v - \bv{\zeta} - \frac{\sqrt{s}}{q} } +  \int \frac{\d {\nu}(v) }{\lambda v - \bv{\zeta} + \frac{\sqrt{s}}{q} }  \right],
			\end{split}
		\end{align}
		where the last line is of order $O(N^{-\fra_0}) $ from \eqref{eq:bound}. Since $\eta =2$, and $\Im \wt{\zeta} , \Im \bv{\zeta} \geq \eta$, we obtain $\Gamma \leq CN^{-\fra_0}$, for $\eta =2$. Writing $\wt{\zeta}$ and $\bv{\zeta}$ in the neighborhood of $\wt{\zeta}_+$ and $\bv{\zeta}_+$, respectively, we have
		\begin{align}
			\wt{\zeta} - \wt{\zeta}_+ &= \wh{c}_+ \sqrt{z - \wt{\zeta}_+} (1 + O(z - \wt{L}_+)),\\
			\bv{\zeta} - \bv{\zeta}_+ &= {c}_+ \sqrt{z - \wt{\zeta}_+} (1 + O(z - \bv{L}_+)),
		\end{align}
		where notice that the constant in the first equation may depend on $V$. 
		
		\begin{align}
			\int \frac{\d \breve{\nu}(v)}{(\lambda v - \breve{\zeta}_{+})^2 } - \int \frac{\d \wt{\nu}(v)}{(\lambda v- \breve{\zeta}_{+})^2 }   &= \sum_{j=-1, 1} \frac{1}{2} \left[ \int \frac{\d {\nu}(v)}{(\lambda v- \breve{\zeta}_{+} + j\frac{\sqrt{s}}{q})^2} - \int \frac{\d \wh{\nu}(v)}{(\lambda v- \breve{\zeta}_{+} +j  \frac{\sqrt{s}}{q})^2 }\right],
		\end{align}
		which is of order $ O\left( N^{- \fra_0} \right)$, and $O_{\prec} (\lambda N^{-1/2})$ when $V$ is random. We illustrate the proof for $V$ is deterministic, the case when $V$ is random can be dealt with similarly; see also \cite[Lemma 4.8]{LS14}, and the proof of Lemma \ref{lem:gaussian}.
		\begin{align}
			0 = \int \frac{\d \wt{\nu}(v)}{(\lambda v - \wt{\zeta}_{+})^2} -1 &=  \int \frac{\d \wt{\nu}(v)}{(\lambda v - \wt{\zeta}_{+})^2}  - \int \frac{\d \wt{\nu}(v)}{(\lambda v - \breve{\zeta}_{+})^2} + O\left( N^{-\fra_0} \right)\\
			&= \int \frac{(-2 \lambda v +\bv{\zeta}_+ + \wt{\zeta}_{+}) (\zeta_+ - \wt{\zeta}_{+})}{(\lambda v - \bv{\zeta}_+)^2 (\lambda v - \wt{\zeta}_{+})^2 } \d \wt{\nu}(v) +  O\left( N^{-\fra_0} \right),
		\end{align}
		and since $v \in \supp \wt{\nu}$, we have
		\begin{align}
			-2 \lambda v + \bv{\zeta}_+ + \wt{\zeta}_{+} \geq 0.
		\end{align}
		From the definition of $\bv{\zeta}_+, \wt{\zeta}_+$, we easily find that they are bounded above by $1+\lambda$. Along with the bound \eqref{eq:g}, we have
		\begin{align}
			\int  \frac{- 2 \lambda v + \breve{\zeta}_{+} + \wt{\zeta}_{+} }{(\lambda v - \breve{\zeta}_{+})^2 (\lambda v- \wt{\zeta}_{+})^2} \d \wt{\nu}(v)>c ' >0, 
		\end{align}
		for some constant $c' >0$ independent of $N$. Then, we deduce that
		\begin{align}
			|\breve{\zeta}_{+} - \wt{\zeta}_{+}|  \leq C N^{-\fra_0} ,\quad |\bv{\zeta}_+ - \wt{\zeta}_+ | \leq CN^{{-\fra_0}},
		\end{align}
		on $\Xi_N(\fra_0)$, for $N$ sufficiently large. Analogous statement holds for the lower edge.
		
		Finally, 
		\begin{align}
			|\wt{\zeta}(z) - \bv{\zeta}(z)| \leq C \sqrt{\vk_E + \eta}+ CN^{-\fra_0 /2},\label{eq:ap2}
		\end{align}
		where the constant can be chosen uniformly in {$\lambda \in \Tvpi$}. By choosing, for example, $\epsilon = \fra_0/4$, from \eqref{eq:ap1} and \eqref{eq:ap2} we get the desired result.
	\end{proof}

	\section{A note on the proof of weak local law}
    \subsection{Proof of Lemma \ref{thm:weak}}\label{Appendix_text}
	In this section, we provide a short note on regarding the weak local law (Lemma \ref{thm:weak}) based on Lemma \ref{lem:cumulantweak} which requires establishing a stochastic continuity argument. While our main reference is \cite{HKR}, there are many established literature in the proof of weak local law, including but not limited to, \cite{EKYY1, LS, BEKYY, EY}.

We first present a lemma that formalizes the stability analysis (refer to \cite[Definition 4.9]{HKR}.) We present the proof for completeness, as the proof is omitted in \cite{HKR}. The proof below is an adaptation of the proof of an analogous statement on Marchenko-Pastur law of sample covariance matrices in Section 4 of \cite{BEKYY}. Finally, we work on $\Xi$ throughout and we fix $\omega \in (0, \frac{\tau}{100})$

\begin{lemma}\label{lem:half}
   Choose arbitrary $E_0 \in \bbR$, and define a line $\caL \deq \{ z = E_0+ \ii \eta : z \in \caD_{\tau} , \eta \in  [a, b] \}$, for some appropriate $a , b \in[N^{-1+\tau}, \tau^{-1}] $. Suppose that $\scrE: \caL \to [N^{-1}, N^{-\omega/2}]$ be continuous such that $\eta \mapsto \scrE (E + \ii \eta) $ is decreasing on $[a, b]$, and Lipschitz continuous with Lipschitz constant $N$. Suppose that
    \begin{align}
        \left|g - \int \frac{d \wh{\nu}(v)}{\lambda v - z - g}  \right| \leq \scrE,
    \end{align}for all $z \in \caL$, and 
    \begin{align}
        |g - \wh{m}_{\fc} | \leq \frac{C \scrE}{\Im \wh{m}_{\fc} + \sqrt{\scrE} },\label{eq:sqrtA}
    \end{align}
    for $z = E_0 + \ii b$. Then, (\ref{eq:sqrtA}) holds for all $z \in \caL$.
\end{lemma}

\begin{proof}
   Consider the parametrization $\check{\caL} = \{z_0, \dotsc, z_{L} \}$, where $\Im z_0 =b$, $\Im z_L =a$, and $\Im z_{l+1} < \Im z_l$, with $|z_{l+1} - z_l| \leq N^{-5}$. Assume that at $z_l $ (\ref{eq:sqrtA}) holds, which implies $|g-\wh{m}_{\fc}| \leq N^{-\omega/4}$. Then, we can obtain the self consistent equation
   \begin{align}
       |(1-R_2)\vvv - R_3 \vvv^2 | \leq C \Lambda^3 + \scrE,
   \end{align}
   where we follow similar notations used in Section \ref{sec:locallaw}, while the argument is for standard deformed semicircle law. That is, $ \vvv \deq g - \wh{m}_{\fc} $ and $\Lambda \deq |g - \wh{m}_{\fc}|$, and $R_{k} \deq \int (\lambda  v - z - \wh{m}_{\fc})^{-k} d \wh{\nu}(v) $. Furthermore, the equation holds for all $z \in \caL$. Let us assume that at a point $z \in \caL$, $\Lambda \ll o(1)$. We will divide the cases comparing the sizes of $\kappa + \eta$ and $4K^2 C' \scrE$. 

\noindent \textbullet \, Case (1): $\kappa + \eta \geq 4 K^2 C'\scrE$ (Bulk estimate)

In this case, we have the inequality
\begin{align}
    \frac{\sqrt{\kappa+\eta}}{K} \Lambda \leq \alpha \Lambda \leq |R_3|\Lambda^2 + o(1) \Lambda^2 + \scrE \leq C'\Lambda^2 +\scrE,
\end{align}
which implies
\begin{align}\label{eq:c1}
   \Lambda \leq \frac{2K \scrE}{\sqrt{\kappa+\eta} + \sqrt{\kappa +\eta -4C'K^2 \scrE}} \leq \frac{ K\scrE}{\sqrt{\kappa +\eta} }\quad \mathrm{or} \quad\Lambda > \frac{\sqrt{\kappa+\eta} - \sqrt{\kappa + \eta -4C'K^2 \scrE}}{2C'K} ,
\end{align}

\noindent \textbullet \, Case(2): $\kappa + \eta < 4 K^2 C \scrE$

In this case, we have from the stability bound
\begin{align}
    c\Lambda^2 \leq |R_3| \Lambda^2 \leq |1-R_2|\Lambda + o(1) \Lambda^2 + \scrE \leq K \sqrt{\kappa+\eta} \Lambda + o(1)\Lambda^2 +\scrE, 
\end{align}
which then implies
\begin{align}\label{eq:c2}
    \Lambda \leq \frac{2\scrE}{K\sqrt{\kappa+\eta} + \sqrt{K^2 (\kappa +\eta) + 4c' \scrE}} \leq \frac{K \scrE}{\sqrt{\kappa + \eta } +\sqrt{\scrE}} ,
\end{align}
sine $K >1$.

We will now proceed by a continuity argument and induction to prove that we have
\begin{align}\label{eq:induction}
    |g(z_l) - \wh{m}_{\fc}(z_l)|\leq  \frac{K\scrE}{\sqrt{\kappa +\eta} + \sqrt{\scrE}}\Bigg|_{z_l},
\end{align}
on $l = 1, \dotsc, L$. Let us assume that (\ref{eq:induction}) hold on $z_l$. For a fixed $E_0$, since $\scrE$ is decreasing in $\eta$ and $\kappa + \eta$ is increasing in $\eta$, there is no crossover between the regimes, or there is a crossover at a single $\eta$, that is,  for some $l_0$, we have, when $l\leq l_0$, the case fall under case (1), and when $l > l_0$, the case fall under case (2). Then, we have
\begin{align} \label{eq:pri}
    |g(z_{l+1}) - \wh{m}_{\fc} (z_{l+1}) |\leq 2N^2 N^{-5} + |g(z_l) - \wh{m}_{\fc}(z_l)| \leq CN^{-3} + \frac{K \scrE}{\sqrt{\kappa+\eta} + \sqrt{\scrE}}\Bigg|_{z_{l+1}} \leq \frac{2K \scrE}{\sqrt{\kappa+\eta} + \sqrt{\scrE}}\Bigg|_{z_{l+1}}  ,
\end{align}
where the first inequality follows from $N^2$-Lipschitz continuity of $g$ and $\wh{m}_{\fc}$ and the second inequality follows from the Lipschitz continuity of $\scrE$ and the bound $\eta \geq N^{-1}$, and the last bound $\scrE \geq N^{-1}$. (\ref{eq:pri}) implies $|g(z_{l+1}) - \wh{m}_{\fc}(z_{l+1}) | = o(1)$, and so we have either (\ref{eq:c1}) or (\ref{eq:c2}). In case $\kappa + \eta \geq 4K^2 C' \scrE$, the bound (\ref{eq:pri}) yields (\ref{eq:induction}) at $z_{l+1}$. 

Since we have assumed (\ref{eq:induction}) at $z_1$, by induction we deduce the bound on all $z \in \check{\caL}$. From $N^2$-Lipschitz property of $g, \wh{m}_{\fc}$ and $\scrE$, we establish the bound on all $\caL$.
\end{proof}

Equipped with Lemma \ref{lem:half}, the proof follows almost verbatim the stochastic continuity argument presented in Section 4 of \cite{HKR}. The continuity argument for weak local law is more intricate than those in Section \ref{sec:locallaw} due to the lack of a priori bounds, although its error bound is non-optimal.

	\subsection{Estimates on the remainder term in cumulant expansion}\label{sec:remainder}
	We illustrate the treatment to $F_{ki} = G_{ki} \caP^{D-1} \overline{\caP^{D}}$, proving that $\caR_{\ell} (W_{ik}F_{ki}) \prec q^{-K}$ for large $N$, for any $K$ given. We follow the analysis in \cite{HKR, LS18}, and we do not assume the weak local law in the proof. Other remainder terms are expected to be dealt with analogously.
	
	Define $E^{[ik]}$ the $N \times N$ matrix whose entries are
	\begin{align}
		(E^{[ik]})_{xy} = \frac{1}{1+ \delta_{im}} (\delta_{ix} \delta_{ky} + \delta_{ky} \delta_{ix}).
	\end{align}
	Define $S^{[ik]} = H_{ik} E^{[ik]}$. For each pair of indices $(i, k)$, we define the matrix $H^{(ik)}$ from $H$ through the decomposition
	\begin{align}
		H = H^{(ik)} + S^{[ik]}.
	\end{align}
	Denote $G^{(ik)} =(H^{(ik)} - z \mathrm{I} )^{-1} $. Then we have the following resolvent identities
	\begin{align}
		G^{(ik)} = G + (G S^{[ik]})G + (G S^{[ik]})^2 G &+ \cdots +(G S^{[ik]})^ aG^{(ik)},\label{eq:rsv1} \\
		G = G^{(ik)}- (G^{(ik)} S^{[ik]})G^{(ik)} &+ \cdots +(-1)^{a} (G^{(ik)}S^{[ik]})^a G,\label{eq:rsv2}
	\end{align}
	for any integer $a$. To ensure $(q^{-1} N^{3\tau/10})^a N <1$, we choose $a$ to satisfy $a> 4\delta^{-1}$. Recall the assumptions $q\geq N^{\delta}$ and $N^{\tau} < q^{20/9}$.
	
	We choose $\frh = q^{-1} N^{\tau/5}$. Note that $S^{[ik]}$ only consists of two entries that are nonzero and they are stochastically dominated by $q^{-1}$. Furthermore, we have the trivial bound $\max_{nm} |G^{(ik)}_{nm}| \leq \eta^{-1} \leq N^{1- \tau}$. With $\theta \leq N^{\tau/10} $, substituting these bounds into \eqref{eq:rsv1} gives
	\begin{align}
		\max_{nm} |G^{(ik)}_{nm}| \prec (1+\theta) \sum_{b = 0 }^{a-1} \left(\frac{1+\theta}{q}\right)^b +  \left(\frac{1+ \theta}{q}\right)^a N^{1- \tau},
	\end{align}
	which from our choice of $a$ yields $\max_{nm} |G^{(ik)}_{nm}| \prec 1+ \theta$. With this bound in hand, iterating the same process with the trivial bound $\max_{nm} \sup_{|H_{ik}| \leq \frh}|G_{nm}| \leq \eta^{-1} \leq N^{1- \tau}$ and relation \eqref{eq:rsv2}, we obtain
	\begin{align}
		\max_{nm} \sup_{|H_{ik}| \leq \frh} |G_{nm}| \prec 1+ \theta.
	\end{align}
	Then a simple induction implies
	\begin{align}
		\max_{nm} \sup_{|H_{ik}| \leq \frh} |\pd_{ik}^{l} G_{nm}| \prec (1+ \theta)^{l+1},
	\end{align}
	for any fixed $l \in \mathbb{N}$. 
	To estimate the second term of the right hand side of \eqref{eq:Remainder}, we can easily show that 
	\begin{align}
		\sup_{ |t| \leq \frh }\left|\pd_{ik}^{(\ell+1)} f(H^{(ik)} + t E^{[ik]}) \right| \prec (1 + \theta )^{\ell +1 } \cdot N^{100D},
	\end{align}
	and so we can find $\ell \equiv \ell(K, D) \geq 1$ such that
	\begin{align}
		\E |H_{ik}|^{\ell+2} \cdot  \E \left[ \sup_{ |t| \leq \frh }\left|\pd_{ik}^{(\ell+1)} f(H^{(ik)} + t E^{[ik]}) \right|\right]  \leq  N^\epsilon q^{-K},
	\end{align}
	for $N$ sufficiently large, since $\E |H_{ik}|^{\ell+2}  = \bigO{N^{-1}q^{-\ell}}$ and $q > N^{\phi}$ for some $\phi>0$.
	Finally, we estimate the first term of the right hand side of \eqref{eq:Remainder}. We easily obtain for sufficiently large $N$ that
	\begin{align}
		\E \sup_{|t| \leq |H_{ik}|} \left| \pd_{ik}^{(\ell +1)} f(H^{(ik)} + t E^{[ik]}) \right|^2 \leq N^{\epsilon/2} q^{c_{D, \ell}},
	\end{align}
	where $c_{D, \ell}$ is some positive constant. Refer to the proof of \cite[Lemma 3.4]{HKR} for details. Furthermore, while we have $\max_{ij} |H_{ik}| \prec q^{-1}$, from Cauhcy-Schwarz inequality we have
	\begin{align}
		\E \left|H_{ik}^{2\ell +4} \mathbbm{1} (|H_{ik}|> \frh) \right| \leq N^{\epsilon/2} N^{-1/2} q^{-(2 \ell+3)}  \bbP({|H_{ik}|>\frh}) ,
	\end{align}
	which can be bounded by $N^{\epsilon/2} q^{-2K - c_{D, \ell}}$. Thus there exist $\ell$ that bounds the first term by $N^{\epsilon} q^{-K}$ for large $N$ for any $K$.

	\section{Cumulant expansion of $Q$} \label{appendix2}
	
	In this section, we provide the details of the cumulant expansion of the quantity $Q$, proving Lemma \ref{lem:Q}.
	\begin{proof}[Proof of Lemma \ref{lem:Q} (1)] The first assertion can be proved by explicit expansions. Unless when exact cancellation is required, we may drop the complex conjugates for they play no role in subsequent analysis. First, from explicit expansions, we have
		\begin{align*}
			\E Y_{1, 0} &+ \E \left[ \Bigg(\frac{1}{N} \pair{G} (\G G)_{ik} + \frac{1}{N} \pair{\G} (G^2)_{ik} \Bigg) Q_{ik}^{D-1} \overline{Q_{ik}^D} \right] \\&= \frac{1}{N^2} \E \Bigg[ \sum_{j_1j_2}(-G_{j_2k} G_{j_1i} \G_{j_2j_1} - G_{j_2k} G_{j_2 i} \G_{j_1 j_1} - G_{j_1 j_2} G_{j_2 k} \G_{ij_1} - G_{j_2j_2} G_{j_1k}  \G_{ij_1}) Q_{ik}^{D-1} \overline{Q_{ik}^D} \Bigg]\\
			& \phantom{=}+ \E \left[ \Bigg(\frac{1}{N} \pair{G} (\G G)_{ik} + \frac{1}{N} \pair{\G} (G^2)_{ik} \Bigg) Q_{ik}^{D-1} \overline{Q_{ik}^D} \right]
			\\
			&= \frac{1}{N^2}\E \Bigg[ \sum_{j_1 j_2}(-G_{j_2k} G_{j_1i} \G_{j_2j_1}  - G_{j_1 j_2} G_{j_2 k} \G_{ij_1} ) Q_{ik}^{D-1} \overline{Q_{ik}^D} \Bigg],
		\end{align*}
		{and by combining (3.8) in \cite{HKR} and Lemma \ref{lem:wardd}, we deduce}
		\begin{align}
			\left|\E Y_{1, 0} + \E \left[ \Bigg(\frac{1}{N} \pair{G} (\G G)_{ik} + \frac{1}{N} \pair{\G} (G^2)_{ik} \Bigg) Q_{ik}^{D-1} \overline{Q_{ik}^D} \right] \right| \prec \Psi_2^3 \E |Q_{ik}|^{2D-1}.
		\end{align}
		
		The term $Y_{1,1}$ is also handled via explicit expansions, where we have
		\begin{align*}
			|\E Y_{1, 1 }| & \leq \frac{1}{N^2 } \E \left| \sum_{j_1 j_2} G_{j_2 k} \G_{i j_1} \pd_{j_1 j_2}  \left( Q_{ik}^{D-1} \overline{Q_{ik}^D} \right) \right| \\
			& = \bigO{\frac{1}{N^2}} \E \left| \sum_{j_1j_2}G_{j_2 k} \G_{ij_1}  Q_{ik}^{2D-2} (\pd_{j_1 j_2} Q_{ik}) \right| \\
			& = \bigO{\frac{1}{N^2 }} \E \Bigg| \sum_{j_1j_2} G_{j_2 k} \G_{ij_1}  Q_{ik}^{2D-2} \Bigg[\frac{1}{N} \G_{ij_1} G_{kj_2} + \frac{1}{N} \G_{ij_2} G_{kj_1}  \\
			&\phantom{=\sum_{j_1 j_2} \E \Bigg|}- G_{j_1i} Q_{j_2k} - G_{j_1k} Q_{ij_2}  - G_{j_2i}Q_{kj_1} - G_{j_2k} Q_{ij_1} \\
			&\phantom{=\sum_{j_1 j_2} \E \Bigg|} - \frac{2}{N^2} (G^2)_{j_1j_2} (\G G)_{ik} - \frac{1}{N^2} [(\G G)_{j_1j_2} + (\G G)_{j_2j_1}] (G^2)_{ik} \\
			& \phantom{=\sum_{j_1 j_2} \E \Bigg|} - \frac{2s}{q^2} \bigg( (\pd_{j_1 j_2} \pair{G \odot G}) \frac{1}{N} \sum_{j=1}^{N} \G_{jj}G_{ij} G_{jk} + \pair{G \odot G} \frac{1}{N} \sum_{j=1}^{N} \pd_{j_1 j_2} (\G_{jj} G_{ij}G_{jk}) \bigg)\\
			& \phantom{=\sum_{j_1 j_2} \E \Bigg|} - \frac{2s}{q^2} \bigg( (\pd_{j_1 j_2} \pair{\G \odot G}) \frac{1}{N} \sum_{j=1}^{N} G_{jj}G_{ij} G_{jk} + \pair{\G \odot G} \frac{1}{N} \sum_{j=1}^{N} \pd_{j_1 j_2} (G_{jj}  G_{ij} G_{jk}) \bigg) \Bigg] \Bigg| \\
			& \prec \left(\Psi_2 ^3 \xi  + \Psi_2^2 \cdot \Psi_2^4 + \frac{1}{q^2} \Psi_2^4 \right) \E |Q_{ik}|^{2D-2},
		\end{align*}
		and we have the desired results.
	\end{proof}
	\begin{proof}[Proof of Lemma \ref{lem:Q} (2)]
		$r \geq 2$. We drop the complex conjugates of $Q_{ik}$ since it plays no role in subsequent analysis, and estimate
		\begin{align}
			\begin{split}
				|\E Y_r| &\leq  \left| \sum_{s=0}^{r} \binom{r}{s} \frac{s^{(r+1)}}{ N^2 q^{r-1}}  \sum_{j_1 j_2}  \E\left[(\pd_{j_1 j_2}^{r-s} G_{j_2 k} G_{ij_1}) (\pd_{j_1j_2}^s Q_{ik}^{D-1} \overline{Q_{ik}^D} ) \right] \right|\\
				&=  \sum_{s=0}^{r} \sum_{a = 0}^{s \wedge 2D-1} \sum_{(l_1, \cdots, l_a) \in \mc{J}_{a}^{s}} \bigO{\frac{1}{ N^2 q^{r-1}}} \sum_{j_1 j_2} \E \left| (\pd_{j_1j_2}^{r-s} G_{j_2k}G_{ij_1}) \left(\prod_{m=1}^{a} (\pd_{j_1j_2}^{l_m} Q_{ik}) \right) Q_{ik}^{2D -1 -a} \right|,
			\end{split}
		\end{align}
		where $\mc{J}_a^{s} $ is the set of all $a$-tuples of positive integers, $(l_1, \cdots, l_a)$, satisfying $l_1 + \cdots + l_a = s$. The first three summands can be suppressed by multiplying a finite constant to the remaining terms. Assuming the variables of the first three summands are fixed, the relevant prevailing term can be estimated
		\begin{align}
			&\bigO{\frac{1}{ N^2 q^{r-1}}} \sum_{j_1 j_2 } \E \left|   (\pd_{j_1 j_2} ^{r-s} G_{j_2 k} G_{ij_1}) \left(\prod_{m=1}^{a} \pd_{j_1 j_2}^{l_m} {Q_{ik}}  \right) Q_{ik}^{2D -1 -a} \right|\label{eq:relQ} \\
			\prec & \bigO{\frac{1}{q^{r-1}}}  \Psi_2^2 (N^{-1} + q^{-b} \Psi_2 + \Psi_2^2)^a \E |Q_{ik}|^{2D-1-a} \\
			\prec &  \bigO{\frac{1}{q^{r-1}}} q^{a+1} ((Nq)^{-1} + q^{-(b+1)} \Psi_2 + q^{-1}\Psi_2^2)^{a+1} \E |Q_{ik}|^{2D-1-a} \\
			\prec & (N^{-1}q^{-1} + q^{-(b+1)} \Psi_2 + q^{-1}\Psi_2^2 +\Psi_2^3 )^{a+1} \E |Q_{ik}|^{2D-1-a} , 
		\end{align}
		for $r-1 \geq a+1$ with $r \geq 2$. The remnant cases are $a > r-2$, and since $a \leq s \leq r$, we only need to consider $(a, s) = (r-1, r-1), (r-1, r), (r, r)$. We divide into the following cases.
		\begin{enumerate}
			\item $(a, s) = (r-1, r-1), (r, r)$ or $(a,s) = (r-1, r)$ with $r \geq 3$
			\item $(a, s) = (r-1, r)$ for $r=2$
		\end{enumerate}
		
		\noindent Case (i): In these cases, there exist at least one $l_i$ that is equal to $l_i =1$. Relabeling the indices of $l_i$ in descending order and expanding $\pd_{j_1j_2} Q_{ik}$, the relevant term can be written
		\begin{align}
			\begin{split}
				\bigO{\frac{1}{N^2 q^{r-1}}} &\sum_{j_1 j_2} \E \Bigg| (\pd_{j_1 j_2}^{r-s} G_{j_2 k} \G_{ij_1} ) \left(\prod_{m'=1}^{a-1} (\pd_{j_1j_2}^{l_{m'}} Q_{ik}) \right) Q_{ik}^{2D -1 -a}\Bigg[\frac{1}{N} \G_{ij_1} G_{kj_2} + \frac{1}{N} \G_{ij_2} G_{kj_1}  \\
				&\phantom{ \E \Bigg|}- G_{j_1i} Q_{j_2k} - G_{j_1k} Q_{ij_2}  - G_{j_2i}Q_{kj_1} - G_{j_2k} Q_{ij_1} \\
				&\phantom{ \E \Bigg|} - \frac{2}{N^2} (G^2)_{j_1j_2} (\G G)_{ik} - \frac{1}{N} [(\G G)_{j_1j_2} + (\G G)_{j_2j_1}] (G^2)_{ik} \\
				& \phantom{ \E \Bigg|} - \frac{2s}{q^2} \bigg( (\pd_{j_1 j_2} \pair{G \odot G}) \frac{1}{N} \sum_{j=1}^{N} \G_{jj}G_{ij} G_{jk} + \pair{G \odot G} \frac{1}{N} \sum_{j=1}^{N} \pd_{j_1 j_2}(\G_{jj} G_{ij})G_{jk} \bigg)\\
				& \phantom{ \E \Bigg|} - \frac{2s}{q^2} \bigg( (\pd_{j_1 j_2} \pair{\G \odot G}) \frac{1}{N} \sum_{j=1}^{N} G_{jj}G_{ij} G_{jk} + \pair{\G \odot G} \frac{1}{N} \sum_{j=1}^{N} \pd_{j_1 j_2} (G_{jj}  G_{ij}) G_{jk} \bigg) \Bigg] \Bigg|,\label{eq:case1}
			\end{split}
		\end{align}
		where we will consider term by term. The last three lines are simple; $N^{-2}\sum_{j_1 j_2} (\pd_{j_1j_2}^{r-s} G_{j_2 k} \G_{ij_1}) = O_{\prec} (\Psi_2^2)$, and the terms in square brackets are dominated by $O_{\prec}(\Psi_2^4 + q^{-2} \Psi_2^2)$. Hence the terms contributing from the last three lines are dominated by 
		\begin{align*}
			&O\left(1/{q^{r-1}}\right) \Psi_2^2 (N^{-1} + q^{-b}\Psi_2 + \Psi_2^2)^{a-1} (\Psi_2^4 + q^{-2} \Psi_2^2) \E |Q_{ik}|^{2D-a-1} \\
			\prec& O\left(1/{q^{r-1}} \right) q^{a-1}(N^{-1}q^{-1} + q^{-(b+1)}\Psi_2 + q^{-1}\Psi_2^2)^{a-1} (\Psi_2^6 + q^{-2} \Psi_2^4) \E |Q_{ik}|^{2D-a-1}\\
			\prec & (N^{-1}q^{-1} + q^{-(b+1)} \Psi_2 + q^{-1} \Psi_2^2 + \Psi_2^3 )^{a+1} \E |Q_{ik}|^{2D-1-a}.
		\end{align*}
		For the terms in the first and second line, the terms in each lines are generically of the same form. Thus, we only illustrate the treatment of the first ones. Moreover, we have to deal with $a=r-1$ and $a=r$ separately. The first term in the first line is
		\begin{align}
			&\bigO{\frac{1}{N^2 q^{r-1}}}  \E \Bigg| \sum_{j_1 j_2}  (\pd_{j_1 j_2}^{r-s} G_{j_2 k} \G_{ij_1} )  \left(\prod_{m'=1}^{a-1} (\pd_{j_1j_2}^{l_{m'}} Q_{ik}) \right) Q_{ik}^{2D -1 -a} \cdot \frac{1}{N} \G_{ij_1} G_{kj_2}  \Bigg| \\
			\prec & \bigO{\frac{1}{q^{r-1}}}  (N^{-1} + q^{-b} \Psi_2 + \Psi_2^2)^{a-1} 
			\frac{1}{N^2}  \sum_{j_1 j_2} \E \Bigg[  \bigg|(\pd_{j_1 j_2}^{r-s} G_{j_2 k} \G_{ij_1} ) \frac{1}{N} \G_{ij_1} G_{kj_2}  \bigg| \cdot Q_{ik}^{2D-1-a} \Bigg], \label{eq:Q11}
		\end{align}
		where if $a=r$, $ N^{-2} |\sum_{j_1 j_2}  (\pd_{j_1 j_2}^{r-s} G_{j_2 k} \G_{ij_1} ) \cdot N^{-1} \G_{ij_1} G_{kj_2}| $ is dominated by $O_{\prec}(N^{-1} \Psi_2^4) $, since $s =r $. However, if $a=r-1$, $s = r-1$ is allowed and the sum can only be dominated by $O_{\prec}(N^{-1} \Psi_2^3) $. Therefore, if $a=r$, we have
		\begin{align}
			\eqref{eq:Q11} \prec& \bigO{1/q^{r-1}} N^{-1} \Psi_2^4 (N^{-1} + q^{-b} \Psi_2 + \Psi_2^2)^{r-1}  \E |Q_{ik}|^{2D-1-r} \\
			\prec & \bigO{1/q^{r-1}}  q^{-2} \Psi_2^4 \cdot  q^{r-1} (N^{-1}q^{-1} + {q^{-(b+1)}}\Psi_2 + q^{-1}\Psi_2^2)^{r-1}  \E |Q_{ik}|^{2D-1-r}\\
			\prec & (N^{-1}q^{-1} + {q^{-(b+1)}}\Psi_2 + q^{-1}\Psi_2^2 + \Psi_2^3 )^{r+1} \E |Q_{ik}|^{2D-1-r}.
		\end{align}
		When $a=r-1$, the summation bound is coarser, but this is offset by a weaker control requirement. In detail, we have
		\begin{align}
			\eqref{eq:Q11} \prec & \bigO{1/{q^{r-1}}} N^{-1}\Psi_2^3 (N^{-1} + q^{-b} \Psi_2 + \Psi_2^2)^{a-1}  \E |Q_{ik}|^{2D-1-a} \\
			\prec & \bigO{1/q^{r-1}} \Psi_2^3 \cdot q^{a} (N^{-1}q^{-1} + {q^{-(b+1)}}\Psi_2 + q^{-1}\Psi_2^2)^{a}  \E |Q_{ik}|^{2D-1-a}\\
			\prec & (N^{-1}q^{-1} + {q^{-(b+1)}}\Psi_2 + q^{-1}\Psi_2^2 + \Psi_2^3 )^{a+1} \E |Q_{ik}|^{2D-1-a}.
		\end{align}
		Next, we turn to the terms in the second line. Assuming $Q = O_{\prec}(\xi)$, the first term in the second line can be estimated
		\begin{align}
			&\bigO{\frac{1}{N^2 q^{r-1}}}  \E \Bigg| \sum_{j_1 j_2}  (\pd_{j_1 j_2}^{r-s} G_{j_2 k} \G_{ij_1} ) \left(\prod_{m'=1}^{r-1} (\pd_{j_1j_2}^{l_{m'}} Q_{ik}) \right) Q_{ik}^{2D -1 -r} \cdot G_{j_1 i} Q_{j_2k} \Bigg|\\
			\prec & \bigO{\frac{1}{q^{r-1}}} (N^{-1} + q^{-b} \Psi_2 + \Psi_2^2)^{r-1}  \frac{1}{N^2}  \sum_{j_1 j_2} \E \left[ |(\pd_{j_1 j_2}^{r-s} G_{j_2 k} \G_{ij_1}) G_{j_1 i} |\cdot |Q_{ik}|^{2D-1-r} \right] \cdot \xi, \label{eq:Q21}
		\end{align}
		For $a=s=r$, the normalized sum over two indices always run through three off-diagonal resolvent entries and we have
		\begin{align}
			\eqref{eq:Q21} \prec & \bigO{1/q^{r-1}} \Psi_2^3 (N^{-1} + q^{-b} \Psi_2 + \Psi_2^2)^{r-1}  \xi \cdot \E |Q_{ik}|^{2D-1-r} \\
			\prec & \bigO{1/q^{r-1}} \Psi_2^3 \cdot q^{r-1} (N^{-1}q^{-1} + q^{-(b+1)} \Psi_2 + q^{-1}\Psi_2^2)^{r-1} \xi \cdot \E |Q_{ik}|^{2D-1-r}\\
			\prec & (N^{-1}q^{-1} + q^{-(b+1)} \Psi_2 + q^{-1}\Psi_2^2 + \Psi_2^3)^{r} \xi \E |Q_{ik}|^{2D-1-r}.
		\end{align} When $a = r-1$, only two off-diagonal entries are guaranteed in the summation, and we have
		\begin{align}
			\eqref{eq:Q21} \prec & \bigO{1/q^{r-1}} \Psi_2^2 (N^{-1} + q^{-b} \Psi_2 + \Psi_2^2)^{a-1} \xi \cdot \E |Q_{ik}|^{2D-1-a} \\
			\prec & \bigO{1/q^{r-1}} q^a (N^{-1}q^{-1} + q^{-(b+1)} \Psi_2 + q^{-1}\Psi_2^2)^{a} \xi \cdot  \E |Q_{ik}|^{2D-1-a}\\
			\prec & (N^{-1}q^{-1} + q^{-(b+1)} \Psi_2 + q^{-1}\Psi_2^2)^{a} \xi \E |Q_{ik}|^{2D-1-a}.
		\end{align}
		Combining these arguments, for the specified cases, where we had $a \geq 1$, we conclude
		\begin{align}
			\eqref{eq:relQ} \prec (N^{-1}q^{-1} + q^{-(b+1)}\Psi_2 +q^{-1} \Psi_2^2 +\Psi_2^3)^{a} \xi \E|Q_{ik}|^{2D-1-a}.
		\end{align}
		
		\noindent Case (ii): In this case, we consider $(a, s) = (r-1, r)$ for $r=2$. Before conducting more explicit expansions, observe that applying extra $\pd_{j_1 j_2}$ to the last three lines in the square bracket of \eqref{eq:case1} conserves the number of fresh indices or extra factors of $q^{-2}$. Thus, these terms attain the same bound to that of $r\geq 3$. Then, the case is reduced to estimating
		\begin{align}
			\begin{split}
				\eqref{eq:relQ} \prec& \bigO{\frac{1}{N^2 q}} \E \Bigg| \sum_{j_1 j_2}  G_{j_2 k} \G_{ij_1} \cdot  \Bigg[ \pd_{j_1j_2} \Bigg(\frac{1}{N} \G_{ij_1} G_{kj_2} + \frac{1}{N} \G_{ij_2} G_{kj_1}  \\
				&\phantom{\bigO{\frac{1}{N^2 q^{r-1}}} \sum_{j_1 j_2} \E \Bigg|}- G_{j_1i} Q_{j_2k} - G_{j_1k} Q_{ij_2}  - G_{j_2i}Q_{kj_1} - G_{j_2k} Q_{ij_1} \Bigg) \Bigg] Q_{ik}^{2D -2} \Bigg|\\&
				+(N^{-1}q^{-1} + q^{-(b+1)} \Psi_2 + q^{-1} \Psi_2^2 + \Psi_2^3 )^{2} \E |Q_{ik}|^{2D-2},
			\end{split}    
		\end{align}
		where again we consider each terms separately. For the terms in the first line in the square bracket, observe that even after acting $\pd_{j_1j_2}$, there is at least one off-diagonal entry containing $j_1$ or $j_2$ as indices. Hence, we have
		\begin{align}
			&\bigO{\frac{1}{N^2q}} \E \left| \sum_{j_1 j_2} G_{j_2 k} \G_{ij_1} \cdot  \Bigg[\pd_{j_1j_2}\Bigg(\frac{1}{N} \G_{ij_1} G_{kj_2} + \frac{1}{N} \G_{ij_2} G_{kj_1} \Bigg)\Bigg] Q_{ik}^{2D -2} \right| \\
			\prec & \frac{1}{Nq} \Psi_2^3 \E |Q_{ik}|^{2D-2} \prec (N^{-1}q^{-1} + q^{-(b+1)} \Psi_2 + q^{-1} \Psi_2^2 + \Psi_2^3 )^{2} \E |Q_{ik}|^{2D-2}.
		\end{align}
		To deal with the terms in the second line, observe that $\pd_{j_1j_2} G_{j_n m}  $ always contain an off diagonal entry containing $j_1$ or $j_2$ for any combinations of $n =1, 2$ and $m=i, k$. Since all these four terms can be bounded via same method, we only illustrate the first one, where we may bound
		\begin{align}
			&\bigO{\frac{1}{N^2q}} \E \left| \sum_{j_1 j_2} G_{j_2 k} \G_{ij_1} (\pd_{j_1j_2} G_{j_1 i} Q_{j_2k} ) Q_{ik}^{2D -2} \right| \\
			=&\bigO{\frac{1}{N^2q}} \sum_{j_1 j_2} \E \left|G_{j_2 k} \G_{ij_1} [(\pd_{j_1j_2} G_{j_1 i}) Q_{j_2k} + G_{j_1 i} (\pd_{j_1j_2} Q_{j_2 k}) ] Q_{ik}^{2D -2} \right|\\
			\prec & q^{-1} \Psi_2^3 (N^{-1} + q^{-b} \Psi_2 + \Psi_2^2)\E |Q_{ik}|^{2D-2} \\ \prec & (N^{-1}q^{-1} + q^{-(b+1)} \Psi_2 + q^{-1} \Psi_2^2 + \Psi_2^3 )^{2} \E |Q_{ik}|^{2D-2}.
		\end{align}
		Combining the previous results, we have
		\begin{align}
			\eqref{eq:relQ} \prec (N^{-1}q^{-1} + q^{-(b+1)}\Psi_2 +q^{-1} \Psi_2^2 +\Psi_2^3)^{a+1}  \E|Q_{ik}|^{2D-1-a}.
		\end{align}
		Therefore, we conclude 
		\begin{align}
			|\E Y_r| \prec &   (N^{-1}q^{-1} + q^{-(b+1)}\Psi_2 +q^{-1} \Psi_2^2 +\Psi_2^3) \E |Q_{ik}|^{2D-1} \\& +\sum_{a'=2}^{2D} (N^{-1}q^{-1} + q^{-(b+1)}\Psi_2 +q^{-1} \Psi_2^2 +\Psi_2^3)^{a'-1} \xi \E |Q_{ik}|^{2D-a'} \\
			\prec & \sum_{a'=1}^{2D} [(N^{-1}q^{-1} + q^{-(b+1)}\Psi_2 +q^{-1} \Psi_2^2 +\Psi_2^3) \xi]^{a'/2} \E |Q_{ik}|^{2D-a'},
		\end{align}
		and the claim is proved.
	\end{proof}
	
	The proof of Lemma \ref{lem:Q} (3) can be obtained in a similar approach that is given in Section \ref{sec:remainder}.
	\section{Proofs in Section \ref{sec:rme}}
	In this section, we provide proofs of Lemma \ref{lem:high} and Lemma \ref{lem:I32}. 

	\begin{proof}[Proof of Lemma \ref{lem:high}]
		For these terms, we may drop the complex conjugates, since they play no role in subsequent analysis. Then we have
		\begin{align}
			|\E I_r| &=\left| \sum_{s=0}^{r} \frac{s^{(r+1)}}{N^2 q^{r-1}} \sum_{ik} \E \left[ (\pd_{ik}^{r-s} \G_{ki}) (\pd_{ik}^{s}  \mc{P}^{D-1} \overline{\mc{P} ^D} )   \right] \right| =  \sum_{s=0}^{r} O \left(\frac{1}{ N^2 q^{r-1}} \right)  \E  \left| \sum_{ik} (\pd_{ik}^{r-s} \G_{ki}) (\pd_{ik}^s  \mc{P}^{D-1} \overline{\mc{P} ^D} ) \right| \\
			&=  \sum_{s=0}^{r} \sum_{a = 0}^{s \wedge 2D-1} \sum_{(l_1, \cdots, l_a) \in \mc{J}_{a}^{s}} \bigO{\frac{1}{ N^2 q^{r-1}}} \E \left| \sum_{ik}  (\pd_{ik}^{r-s} \G_{ki}) \left(\prod_{m=1}^{a} \pd_{ik}^{l_m} \mc{P}  \right) \mc{P}^{2D -1 -a} \right|,
		\end{align}
		where $w_{r,s} = \binom{r}{s}/r!$ and $\mc{J}_a^{s} $ is the set of all $a$-tuples of positive integers, $(l_1, \cdots, l_a)$, satisfying $l_1 + \cdots + l_a = s$. To simplify our arguments, we divide into when $r-s$ is odd and $r-s$ is even. If $r-s$ is even, $\pd_{ik}^{r-s} \G_{ki}$ always contain an off-diagonal entry, then by summing over $i$ or $k$, we can get a factor $\Psi_2$ from Ward identity. On the other hand, if $r-s$ is odd, we do not gain any extra factor upon summation.
		
		Let $n_p^{(r-s)} = 0$ if $r-s $ is odd, and $n_p^{(r-s)} = 1$ if $r-s$ is even. Then, we have
		\begin{align*}
			\bigO{\frac{1}{ N^2 q^{r-1}}} \E \left|  \sum_{ik}  (\pd_{ik}^{r-s} \G_{ki}) \left(\prod_{m=1}^{a} \pd_{ik}^{l_m} \mc{P}  \right) \mc{P}^{2D -1 -a} \right|\prec & \bigO{\frac{1}{q^{r-1}}}  (N^{-1} + q^{-1}\Psi_2^2 +\Psi_2^3)^a \Psi_2^{n_p^{(r-s)}} \E |\mc{P}^{2D-1-a}| \\
			\prec &  (\Psi_2^{2})^{a+1} \E |\mc{P}^{2D-1-a}|,
		\end{align*}
		where the condition for the validity of the last line is verified by computing the power index:
		\begin{align}
			\frf(n_q = r-1, n_d =a, n_p =n_p^{(r-s)}, n_m =a+1) 
			&= \min \left\{ r-3 , 2a-2  , \frac{r+a-5 }{2}\right\} +n_p^{(r-s)}.
		\end{align}
		If $r-s $ is odd, the power index is non-negative (thus the last step holds) for any pair of $(r, s)$ such that $r \geq 4$ unless $a \not =0$. If $a=0$, the last step clearly holds when $r \geq 5$. On the other hand, if $r-s$ is even, the criteria holds  for any $r \geq 2$, unless $a \not =0$. In case $a = 0$, the last step holds when $r \geq 3$.
		
		Noticing that the restriction $a \not = 0$ is equivalent to $s \not = 0$, we have proved that the given relation holds for the $(r, s)$ pairs of $r \geq 4, 0\leq s \leq r$, or $(r,s) = (3, 1), (3, 3), (2, 2)$, for all $z \in \caD_{\tau}$.
		
Refer to Section~\ref{sec:remainder} for the cumulant expansion illustration of the remainder term.
	\end{proof}
	
	Next, we prove the bound for $I_{3,2}$.
	\begin{proof}[Proof of Lemma \ref{lem:I32}]
		Explicit expansion yields
		\begin{align}
			|\E I_{3, 2}|
			&= \bigO{\frac{1}{N^2q^2 }}\left| \sum_{ik} \E (\pd_{ik} \G_{ik}) \pd_{ik}^2 ( \mc{P}^{D-1} \overline{\mc{P}^D} ) \right| \\
			\begin{split}
				&\prec \bigO{\frac{1}{N^2q^2}} \E \Bigg | \sum_{ik}  (G_{ii} \G_{kk} + G_{ik} \G_{ki}) \bigg[ \left(\frac{1}{N} \G_{ki} + \frac{1}{N} \G_{ik} - Q_{ki} - Q_{ik} \right)^2 \mc{P}^{2D-3}  
				\\ & \phantom{[[[[[[[]]]]]]]}+ \left(- \frac{1}{N}(G_{ii} \G_{kk} + G_{ik} \G_{ki} + G_{kk} \G_{ii} +G_{ki} \G_{ik} ) - \pd_{ik}Q_{ik} -\pd_{ik}Q_{ki} \right) \mc{P}^{2D-2} \bigg] \Bigg| ,
			\end{split}
		\end{align}
		where we deal with the two terms separately. First consider
		\begin{align}
			&\frac{1}{N^2 q^2}  \E \left| \sum_{ik} (G_{ii} \G_{kk} + G_{ik} \G_{ki}) \left( \frac{1}{N} \G_{ki} + \frac{1}{N} \G_{ik} - Q_{ki} - Q_{ik} \right)^2 \mc{P}^{2D-3} \right| \\\prec & \frac{1}{q^2} (N^{-1} + q^{-1} \Psi_2^2 + \Psi_2^3 )^2 \E |\mc{P}|^{2D-3},
		\end{align}
		which is of $O_{\prec}(\Psi_2^6 \E |\caP|^{2D-3})$ since $N^{-2} q^{-2} \leq (N^{-1/4} q^{-1/2})^6$. For the terms in the second line, terms without $Q_{ik}$ are bounded easily with the aid of an extra factor $N^{-1}$. For the other terms, we expand further as follows.
		\begin{align}
			&\frac{1}{N^2 q^2} \sum_{ik} \E \left| (G_{ii} \G_{kk} + G_{ik} \G_{ki}) \pd_{ik} Q_{ik} \mc{P}^{2D-2}\right| \\
			\begin{split}
				\prec & \frac{1}{N^2 q^2} \sum_{ik} \E \Bigg| (G_{ii} \G_{kk} + G_{ik} \G_{ki}) \Bigg[\frac{1}{N} \G_{ii} G_{kk} + \frac{1}{N} \G_{ik} G_{ki}  \\& \phantom{\frac{1}{N^2 q^2} \sum_{ik} \E \Bigg| ( G_{ik} \G_{ki}) \Bigg[ }- G_{ii} Q_{kk} - G_{ik} Q_{ik}  - G_{ki}Q_{ki} - G_{kk} Q_{ii} \Bigg] \mc{P}^{2D-2}\Bigg| + q^{-2} \Psi_2^3 \E|\mc{P}|^{2D-2},
			\end{split}
		\end{align}
		where we have used \eqref{eq:3line}. Those terms containg an extra factor of $N^{-1}$ or off-diagonal resolvent entries are handled easily. The tricky terms are those only with diagonal entries, where up to the results we have obtained so far, they are of the order $O_{\prec}(q^{-2} (N^{-1} + q^{-1} \Psi_2^{2} + \Psi_2^3 \E |\mc{P}|^{2D-2})$, which is dominated by $O_{\prec} (\Psi_{3/2}^4 \E|\mc{P}|^{2D-2}$).
	\end{proof}
	
	\section{Proofs in Section \ref{sec:locallaw}}

	\begin{proof}[Proof of Lemma \ref{lem:det}]
		The proof is based on a dichotomy argument. We start by defining
		\begin{align} 
			\alpha_0 (z) \deq N^{\frac{3\epsilon}{10}} \left(\frac{1}{q^{\frb}} + \frac{1}{N^{1/4} q^{1/2}} + \frac{1}{N\eta} \right)^{1-\upsilon/2}.
		\end{align}
		Then, for $N$ sufficiently large, we have
		\begin{align}
			\frac{1}{q^{2\frb}}+  \frac{1}{N^{1/2} q} + \frac{\vartheta(z)}{N \eta} \leq  \alpha_0^2 + \frac{N^{\epsilon}}{N\eta} \left(\frac{1}{q^{\frb}} + \frac{1}{N^{1/4}q^{1/2}} +  \frac{1}{N \eta} \right)^{1- \upsilon}  \leq N^{\frac{15\epsilon}{10}} \left(\frac{1}{q^{\frb}} + \frac{1}{N^{1/4}q^{1/2}} + \frac{1}{N \eta} \right)^{2- \upsilon} .
		\end{align}
		Writing $|1-\wt{\mathsf{R}}_2 (z) | \deq \alpha(z)$, we divide into $\alpha(z) \leq \alpha_0(z)$ and $\alpha(z) > \alpha_0(z)$. In the regime $\alpha \leq \alpha_0$, $\vk \ll 1$, and so $|\wt{\mathsf{R}}_3| > {c}$ for some constant ${c} >0$. Then, writing 
		\begin{align}
			|[\mathrm{v}]^2| \leq \left|[\mathrm{v}]^2 - \frac{(1-\wt{\mathsf{R}}^{(2)}) [\mathrm{v}]}{\wt{\mathsf{R}}^{(3)}} \right| + \frac{\alpha [\mathrm{v}]}{c} ,
		\end{align}
		we obtain
		\begin{align}
			\Lambda^2 \leq \frac{C}{q^2} \Lambda  + N^{\frac{15\epsilon}{10}} \left(\frac{1}{q^{\frb}} + \frac{1}{N^{1/4}q^{1/2}} + \frac{1}{N \eta} \right)^{2- \upsilon}  + N^{\frac{3\epsilon}{10}}\alpha \left(\frac{1}{q^2} + \frac{1}{N\eta}\right)  + \frac{\alpha \Lambda}{c},
		\end{align}
		for $N$ large enough. The relation can be organized into the quadratic form
		\begin{align*}
			\left(\Lambda - \frac{C}{2q^2} - \frac{\alpha}{2c} \right)^2 \leq \frac{1}{4} \left(\frac{C}{q^2} +\frac{\alpha}{c} \right)^2 + N^{\frac{15\epsilon}{10}} \left(\frac{1}{q^{\frb}} + \frac{1}{N^{1/4}q^{1/2}} + \frac{1}{N \eta} \right)^{2- \upsilon}+N^{\frac{3\epsilon}{10}}\alpha \left(\frac{1}{q^2} + \frac{1}{N\eta}\right),
		\end{align*}
		and taking square root, from $\alpha \leq \alpha_0$, we find that
		\begin{align}
			\Lambda \leq  N^{\epsilon} \left(\frac{1}{q^{\frb}} + \frac{1}{N^{1/4}q^{1/2}} + \frac{1}{N \eta} \right)^{1- \upsilon/2}.
		\end{align}
		If $\alpha > \alpha_0$, recall the bound $|\wt{\mathsf{R}}^{(3)}|< {C}_3$ for some constant ${C}_3 >0$. Then, we have
		\begin{align}
			\alpha \Lambda \leq {C}_3 \Lambda^2 + C q^{-2} \Lambda +N^{\frac{2\epsilon}{10}} \left(\frac{1}{q^{2 \frb}} + \frac{1}{N^{1/2} q} +\frac{|1- \wt{\mathsf{R}}^{(2)} |}{q^2}+ \frac{ |1- \wt{\mathsf{R}}^{(2)} | + \vartheta}{N \eta} \right),
		\end{align}
		and assuming $\Lambda \leq \alpha/(2 {C}_3)$, we obtain
		\begin{align}
			\Lambda &\leq C N^{\frac{2\epsilon}{10}} \left(\frac{1}{\alpha q^{2\frb}} + \frac{1}{\alpha N^{1/4} q^{1/2}} +\frac{1}{q^2}+ \frac{ 1}{N \eta} + \frac{\vartheta}{\alpha N\eta }\right) \leq C N^{-\frac{\epsilon}{10}} \left(2\frac{\alpha_0^2}{\alpha} + \alpha_0 \right).
		\end{align}
		Since $\alpha> \alpha_0$, it follows that if $\Lambda \leq \alpha/(2{C}_3)$, we have $\Lambda \ll \alpha/(2 {C}_3)$. This implies the dichotomy $\Lambda > \alpha/(2 {C}_3)$ or $\Lambda \ll \alpha/(2 {C}_3)$. We have assumed $\Lambda (z) \ll 1 $ for $\eta \sim 1$, and in such regimes $O(\alpha(z)) =1$. Since $\Lambda (z)$ is continuous in $\eta = \Im z$, we obtain the claim for $\alpha > \alpha_0$, and the proof is complete.
	\end{proof}
	
	\section{Proofs in Section \ref{sec:behavior}}
	\subsection{Proof of Lemma \ref{thm:opnorm}} 
	\begin{proof}[Proof of Lemma \ref{thm:opnorm}]
		We will only consider the largest eigenvalue $\mu_N$. A  bound on the lowest eigenvalue $\mu_1$ is obtained in a similar way. 
		
		First, we recall an eigenvalue bound for sparse random matrices. It was proved in \cite[Theorem 2.9]{LS18}, that for $W$ satisfying Definition \ref{def:sparse}, we have
		\begin{align}\label{eq:LS}
			\left| \norm W \norm -2 -\frac{s}{q^2} \right| \prec \frac{1}{q^4}   + \frac{1}{N^{2/3}} .
		\end{align}
		The bound will be used to cover the energies $E \geq E_0$.
		
		Fix (small) $\epsilon >0$. Let us define an event $\Omega_0$ where $\Xi$ holds, along with the bounds
		\begin{align}
			\left| \norm W \norm -2 - \frac{s}{q^2} \right| &\leq N^{\epsilon} \left(\frac{1}{q^4} + \frac{1}{N^{2/3} } \right),\\
			\Lambda(z) &\leq N^{\epsilon} \left(\frac{1}{q^{3/2}} +\frac{1}{N^{1/4} q^{1/2}} + \frac{1}{N \eta}\right),\label{eq:076}\\
			|(1- \wt{\mathsf{R}}^{(2)} )[\mathrm{v}] - \wt{\mathsf{R}}^{(3)} [\mathrm{v}]^2| &\leq o(1) \Lambda^2 +{ N^{\frac{\epsilon}{10}} \left( \frac{1}{q^3} + \frac{1}{\sqrt{N}q} + \frac{|1 - \wt{\mathsf{R}}^{(2)}|}{q^2} + \frac{\Im \wt{m}_{\fc} + \vt }{N\eta} \right)}\label{eq:077},
		\end{align}
		where $\vt = q^{-3/2} + N^{-1/4}q^{-1/2} + (N\eta)^{-1}$ and
		the \eqref{eq:076} and \eqref{eq:077} uniform in $z = E+ \ii \eta \in \caD_{\frac{\epsilon}{10}}$.  Clearly, from \eqref{eq:LS}, Theorem \ref{thm:refine}, and \eqref{eq:scone}, for any (large) $D \ge 1 $, $\Omega_0$ holds with probability at least $1-N^{-D}$, for $N \geq N_0(\epsilon, D)$. 
		
		From now on, we work on the event $\Omega_0$. For $E > \wt{L}_+$ and $\vk_E \geq \eta$, we have $\Im \wt{m}_{\fc} \sim \eta/\sqrt{\vk_E}$ and
		\begin{align}
			\alpha \deq |1- \wt{\mathsf{R}}^{(2)} | \sim \sqrt{\vk_E}.
		\end{align}
		Thus, we deduce
		\begin{align}
			|(1- \wt{\mathsf{R}}^{(2)} )[\mathrm{v}] - \wt{\mathsf{R}}^{(3)} [\mathrm{v}]^2| \leq  o(1) \Lambda^2 + N^{\frac{\epsilon}{10}} \left( \frac{1}{q^3} + \frac{1}{\sqrt{N}q} + \frac{\sqrt{\vk_E} }{q^2} + \frac{1}{N \sqrt{\vk_E}} + \frac{1}{(N \eta)^2} \right), 
		\end{align}
		for $N$ sufficiently large. For any $E$ satisfying
		\begin{align}
			\wt{L}_+ + N^{\frac{9\epsilon}{10}} \left( \frac{1}{q^3 } + \frac{1}{N^{1/2} q } +\frac{1}{N^{2/3}} \right)  \leq E \leq E_0,\label{eq:EEE}
		\end{align}
		where then $\vk_E$ satisfies
		\begin{align}\label{eq:minkap}
			\min \{N^{-1/2} \vk_E^{1/4}, N^{-1} q^3 \vk_E^{1/2} , N^{-1/2} q \vk_E^{1/2}  , \vk_E , N^{-1} q^2 \} \geq \frac{N^{\frac{3\epsilon}{10}}}{N \sqrt{\vk_E}},
		\end{align}
		For $E$ satisfying \eqref{eq:EEE}, set
		\begin{align}
			\eta \equiv \eta_E \deq  \frac{N^{\frac{2{\epsilon}}{10}}}{N \sqrt{\vk_E}}.
		\end{align}\label{eq:eat}
		Note that $z = E+ \ii \eta  \in \caD_{\frac{\epsilon}{10}}$. From \eqref{eq:minkap}, we have $\vk_E \geq \eta$. We also have 
		\begin{align} \label{eq:quadddd}
			\Im \wt{m}_{\fc} ( E + \ii \eta) \sim \frac{\eta}{\sqrt{\vk_E}} \ll \frac{1}{N\eta}, \quad \frac{1}{q^3 \sqrt{\vk_E}} \ll \frac{1}{N\eta}, \quad \frac{1}{\sqrt{N}q \sqrt{\vk_E}} \ll \frac{1}{N\eta}, \quad \frac{1}{q^2} \ll \frac{1}{N\eta}.
		\end{align}
		Since $\alpha \geq C^{-1} \sqrt{\vk_E}$ for some constant $C >1$, we have
		\begin{align}
			2 |\wt{\mathsf{R}
			}^{(3)}|\Lambda \leq C N^{\epsilon} \left(\frac{1}{q^{3/2}} + \frac{1}{N^{1/4} q^{1/2}} + \frac{1}{N\eta}  \right) \leq \alpha, \quad N^{\epsilon} \frac{1}{\alpha N \sqrt{\vk_E}} \leq  \frac{C \eta}{\sqrt{\vk_E}} \ll \frac{1}{N\eta}
		\end{align}
		Also, since $\alpha \geq 2|\RRR{3}| \Lambda$, we then have
		\begin{align}
			\Lambda \leq C N^{\epsilon} \left( \frac{1}{\alpha q^3} + \frac{1}{\alpha \sqrt{N}q}+ \frac{1}{q^2}  + \frac{1}{\alpha (N \eta)^2} + \frac{1}{\alpha N \sqrt{\vk_E}} \right) \ll \frac{1}{N\eta}.
		\end{align}
		Therefore, for any energy $E$ satisfying \eqref{eq:EEE} and $\eta$ satisfying \eqref{eq:eat} $z = E+ \ii \eta \in \caD_{\frac{\epsilon}{10}}$, we obtain
		\begin{align}
			\Im g(z) \leq \Im \wt{m}_{\fc} (z) +  \Lambda(z) \ll \frac{1}{N\eta}.
		\end{align}
		On the other hand, if there is an eigenvalue in the interval $[E-\eta, E+\eta]$, we have
		\begin{align}
			\Im g(z) = \frac{1}{N} \sum_{j =1}^{N} \frac{1}{(\mu_j - E)^2 + \eta^2} \geq \frac{C}{N\eta},
		\end{align}
		for some $C > 0$ by the spectral decomposition of $H$. Therefore, on the event $\Omega_0$, there is no eigenvalue in the interval $[E - \eta, E+\eta]$. 
		
		For the energies $E \geq E_0$, from spectral perturbation theory, it follows that $\norm H \norm \leq \norm W \norm + \lambda \norm V \norm$. The result implies that, with high probability, there is no eigenvalue of $H =W + \lambda V$ in this regime.
		
		Therefore, on the high probability event $\Omega_0$, there is no eigenvalue larger than $\wt{L}_+ + N^{\epsilon} ( N^{-2/3} + N^{-1/2} q^{-1} + q^{-3})$, and the claim follows.
	\end{proof}
	
	\subsection{Proof of Proposition \ref{prop:loc} and Proposition \ref{prop:int}}

	First, we prove Proposition \ref{prop:loc} following a standard application to the Helffer-Sj\"{o}strand formula.
	
		We denote the differences
	\begin{align}
		\rho^{\Delta} \deq \rho_N - \wt{\rho}_{\fc}, \quad m^{\Delta} \deq g - \wt{m}_{\fc}.
	\end{align}
	 We have the following intermediate lemma.
	
	\begin{lemma}\label{lem:lem7}
		Choose arbitrary (small) $\epsilon>0$. Fix two energies $E_1<E_2$ in $[-E_0, E_0]$. Set $\eta = N^{-1}$ and ${\eta} '= N^{\tau -1}$, for $\tau>0$ satisfying the assumption given for $\caD_{\tau}$. Define $f(x) \equiv f_{E_1, E_2 , {\eta}'}(x)  $ to be an indicator function of the interval $[E_1, E_2] $, smoothed out on a scale ${\eta}'$, i.e. $f(x) =1$ for $x \in [E_1 , E_2] $, and $f(x) = 0$ for $x \in [E_1 - {\eta}', E_2 + {\eta}']^{c}$; moreover $|f'(x) | \leq C {\eta}'^{-1}$ and $|f''(x) | \leq C {\eta}'^{-2}$. Assume Theorem \ref{thm:refine} on the domain $\caD_{\tau}$. Then we have
		\begin{align}
			\left| \int_{\bbR} f(x) \rho^{\Delta} (x) dx \right| \prec {\eta}'  +  \min \left\{ {\caE}\left( \frac{1}{q^{3/2}} + \frac{1}{N^{1/4}q^{1/2}} \right),\frac{1}{q^2} \frac{\caE}{\sqrt{\vk + \caE }} \right\} ,\label{eq:res7}
		\end{align}
		where we have abbreviated $\vk \deq \min  \{ \vk_{E_1} , \vk_{E_2} \}$ and $\caE \deq \max\{E_2-E_1 , {\eta}'. \}$
	\end{lemma}
	\begin{proof}
		We use the Helffer-Sj\"{o}strand formula following the methods introduced in \cite{EKYY1}. Choose a cut-off function $\chi(y)$ such that $\chi(y) =1$ on $[-\caE, \caE]$, $\chi(y) = 0$ on $[-2 \caE, 2\caE]^c$, and $|\chi'(y)| \leq C/\caE$. Starting from the Helffer-Sj\"{o}strand formula,
		\begin{align}
			f(w) = \frac{1}{2\pi} \int_{\bbR^2} \frac{\ii y f''(y) \chi(y) + \ii (f(x) + \ii y f'(x) )\chi'(y) }{w - x - \ii y}\d x \d y,
		\end{align}
		we obtain
		\begin{align}
			\begin{split}
				& \left|\int_\bbR f(w) \rho^{\Delta }(w) \d w \right| \leq C \int \d x \int_0^{\infty} \d y (|f(x)| + |y| |f'(x)|) |\chi'(y)| |m^{\Delta} (x+ \ii y)| \\
				& +C\left| \int \d x \int_0^{\eta} \d y f''(x) \chi(y) y \Im m^{\Delta} (x + \ii y)  \right| + C\left| \int \d x \int_{\eta}^{\infty} \d y f''(x) \chi(y) y \Im m^{\Delta} (x + \ii y)  \right|,
			\end{split}
		\end{align}
		where we denote the first, second, and third term in the right hand side by $A_1, A_2, A_3$, respectively. Using that $\chi'$ is supported on $[\caE, 2 \caE]$, we have
		\begin{align}
			A_1 &\prec \caE^{-1} \int \d x \int_{\caE}^{2\caE} \d y (|f(x)| + y |f'(x)|) \left( \min \left \{\frac{1}{q^{3/2}} + \frac{1}{N^{1/4}q^{1/2}},  \frac{1}{q^2 \sqrt{\vk_x + \caE }} \right \} + \frac{1}{N\caE } \right) \\&\prec \caE\min \left \{\frac{1}{q^{3/2}} + \frac{1}{N^{1/4}q^{1/2}},  \frac{1}{q^2 \sqrt{\vk_x + \caE }} \right \}  + \frac{1}{N} .
		\end{align}
		To bound $A_2$ and $A_3$, we first bound the imaginary part of $m^{\Delta}(x + \ii y)$. For $y \geq \eta'$, we can use the bound given in Theorem \ref{thm:strong}. Now assume that $0 < y < \eta'$. Using the spectral decomposition of $H = W+\lambda V$, we easily see that the function $y \mapsto y \Im g(x + \ii y) $ is monotone increasing. Also recall that from Lemma \ref{lem:stabbound}, there exist some $C>0$ such that
		\begin{align}\label{eq:sta77}
			|\Im \wt{m}_{\fc} (x + \ii y)| \leq C \sqrt{\vk_x + y}.
		\end{align} Thus, combining with the local law, we get
		\begin{align}
			y \Im g (x + \ii y) \leq {\eta}' \Im g (x + \ii {\eta}') \prec {\eta}' \left(\sqrt{\vk_{x} + {\eta}'} + \min \left\{ \frac{1}{q^{3/2}} + \frac{1}{N^{1/4} q^{1/2}} , \frac{1}{q^2 \sqrt{\vk_x + \eta'}}\right\} +\frac{1}{N{\eta}'} \right),
		\end{align}
		for $y \leq {\eta}'$ and $|x | \leq E_0$. Recalling \eqref{eq:sta77}, we arrive at
		\begin{align}
			|y \Im m^{\Delta} (x+ iy)| \prec {\eta}' \sqrt{\vk_x + {\eta}'} + \frac{1}{N} \leq C {\eta}',
		\end{align}
		for $y \leq {\eta}'$ and $|x| \leq E_0$. Since by assumption we have $\eta \leq \eta'$, $A_2$ is bounded by
		\begin{align}
			A_2 \prec \frac{1}{N} \int \d x |f''(x)| \int_0^{\eta} \d y \chi(y) \prec \frac{1}{N},
		\end{align}
		where we used that the support of $f''$ has measure $O(\eta)$. To bound $A_3$, we integrate by parts, first in $x$ then in $y$ to find the bound
		\begin{align}
			\begin{split}
				A_3 \leq  &C\left|\int \d x f'(x) \eta \Re m^{\Delta}(x+ \ii y)  \right| + C \left| \int \d x \int_{\eta}^{\infty} \d y f'(x) \chi'(y) y \Re m^{\Delta }(x+ \ii y)  \right|\\
				&+ C \left| \int \d x \int_{\eta}^{\infty} \d y f'(x) \chi(y) \Re m^{\Delta}(x+ \ii y) \right|,
			\end{split}
		\end{align}
		where we denote the first, second, and third term in the right hand side by $A_{3,1}, A_{3,2}, A_{3,3}$, respectively. We easily find that $A_{3,2}$ can be bounded analogously to $A_1$, and we have
		\begin{align}
			A_{3,2} \prec \caE\min \left \{\frac{1}{q^{3/2}} + \frac{1}{N^{1/4}q^{1/2}},  \frac{1}{q^2 \sqrt{\vk_x + \caE }} \right \}  + \frac{1}{N}.
		\end{align}
		To bound $A_{3,1}$ and $A_{3,3}$, we write, for $y \leq \eta'$,
		\begin{align}\label{eq:789}
			|m^{\Delta} (x + \ii y)| \leq |m^{\Delta}(x+ \ii \eta')|+ \int_{y}^{\eta'} \d u \left(|\pd_u g(x+ \ii u)| + |\pd_u \wt{m}_{\fc} (x+ \ii u)  | \right).
		\end{align}
		Also observe that we have from Ward's identity that
		\begin{align}
			|\pd_u m(x + \ii y)| = \frac{1}{N} \left|\Tr G^2(x + \ii u)  \right|\leq \frac{1}{N} \sum_{i j}\left| G_{ij} (x + \ii u) \right|^2 = \frac{1}{u} \Im m(x + \ii u) \leq \frac{1}{u^2} \eta' \Im m(x+ \ii \eta').
		\end{align}
		Similarly, we obtain
		\begin{align}
			|\pd_u \wt{m}_{\fc} (x + \ii u)| \leq \int \frac{\wt{\rho}_{\fc}(t) \d t}{|t - x - \ii u|^2} = \frac{1}{u} \Im \wt{m}_{\fc} (x + \ii u) \leq \frac{1}{u^2} \eta' \Im \wt{m}_{\fc}(x+ \ii \eta').
		\end{align}
		From \eqref{eq:789}, we retrieve
		\begin{align}\label{eq:987}
			|m^{\Delta}( x + \ii y) | \prec \left(1 + \int_{y}^{\eta'} \d u \frac{\eta'}{u^2} \right)\prec \frac{\eta'}{y},
		\end{align}
		for $y \leq \eta'$. Thus we have $A_{3,1} \prec \eta'$. To bound $A_{3,3}$, we split the integration in the $y$ variable into the pieces $[\eta, \eta')$ and $[\eta', \infty)$. Using \eqref{eq:789}, we can bound the first piece by
		\begin{align}
			\int \d x |f'(x)| \int_{\eta}^{\eta'} \d y |m^{\Delta} ( x + \ii y) | \prec \eta'.
		\end{align}
		For the second integration piece, we find
		\begin{align*}
			\int \d x |f'(x) | \int_{\eta'}^{2\caE} \d y |m^{\Delta}(x+ \ii y) | &\prec  \int_{\eta'}^{2\caE} \d y \left(\min  \left\{\frac{1}{q^{3/2}} + \frac{1}{N^{1/4}q^{1/2}}, \frac{1}{q^2 \sqrt{\vk + y}} \right\} + \frac{1}{N y} \right)\\
			&\prec \eta' + \min \left\{\caE \left(\frac{1}{q^{3/2}} + \frac{1}{N^{1/4}q^{1/2}} \right),  \frac{ \caE}{q^2\sqrt{\vk + \caE}} \right\}.
		\end{align*}
		Culminating all the contributions together, we have \eqref{eq:res7}.
	\end{proof}
	
	\begin{proof}[Proof of Proposition \ref{prop:loc}]
		Observe that for $\eta = N^{-1}$, 
		\begin{align}
			|\frn( x+ \eta) - \frn(x - \eta) | \leq C \eta \Im g(x+ \ii \eta) \prec N^{-1},
		\end{align}
		where we have also used the delocalization of eigenvectors. Hence,
		\begin{align}
			\left|\frn(E_1) - \frn(E_2) -\int f(w) \rho_N(w) \d w\right| \leq C\sum_{j=1, 2} (\frn(E_j + \eta) - \frn(E_j - \eta)) \prec N^{-1}.
		\end{align}
		Since $\wt{\rho}_{\fc}$ is a bounded function, we find
		\begin{align}
			\left| \wt{n}_{\fc}(E_1) -  \wt{n}_{\fc}(E_2)  - \int f(w) \wt{\rho}_{\fc} (w) \d w \right| \leq C \eta = \frac{C}N^{-1}.
		\end{align}
		Then, we combine Lemma \ref{lem:lem7} and obtain the desired result.
	\end{proof}
	
	With Lemma \ref{thm:opnorm} and Proposition \ref{prop:loc} in hand, we can prove Proposition \ref{prop:int}.
	
	\begin{proof}[Proof of Proposition \ref{prop:int}] We assume that $|E-\wt{L}_{-}| \leq |E - \wt{L}_+| $; the other case is treated similarly. Fix $\epsilon>0$ and set
		\begin{align}
			{E}_1 \deq \wt{L}_{-} -N^{\epsilon} \left(q^{-3} + N^{-2/3} \right).
		\end{align}
		By Lemma \ref{thm:opnorm}, we have $\wt{n}_{\fc} ({E}_1) = 0$ and $\frn({E}_1) =0$, with high probability. Setting $E_2 =E$ and applying Proposition \ref{prop:loc}, we obtain
		\begin{align}
			|\frn(E) - \wt{n}_{\fc} (E)|\leq N^{\epsilon} \left[\frac{N^{\epsilon}}{N} \ + \min \left\{ \left( \frac{1}{q^{3/2}}+\frac{1}{N^{1/4}q^{1/2}} \right)\left(E-\wt{L}_-\right),   \frac{1}{q^2} \sqrt{\vk_E  +N^{\epsilon}\left( E-\wt{L}_- \right)} \right\} \right],
		\end{align}
		which holds with high probability, and upon applying the Young's inequality, we retrieve
		\begin{align}\label{eq:comp}
			|\frn(E) -\wt{n}_{\fc}(E)| \prec \frac{1}{N} +   \min \left\{ \left(\frac{1}{q^{3/2} } + \frac{1}{N^{1/4}q^{1/2}}  \right)\left( \vk_E + \frac{1}{q^3} + \frac{1}{N^{2/3}} \right), \frac{\sqrt{\vk_E}}{q^2} + \frac{1}{q^{7/2} } + \frac{1}{q^2 N^{1/3}} \right\} .
		\end{align}
		After slight adjustment, we obtain the desired result.
	\end{proof}

\subsection{Proof of Lemma \ref{lem:gaussian}}
In this section, we complete the proof of Lemma \ref{lem:gaussian} by proving that the variance of the Gaussian obtained in the sparse case coincides with the Wigner case.
\begin{proof}
	
	Let us first consider when $\lambda \ll 1$, with $\lambda \gg q^{-1}$. Expansion yields 
	\begin{align}
		1 = \int \frac{\d \bv{\nu} (v)}{(\lambda v - \bv{\zeta}_{+})^2 } &= \int \d \bv{\nu}(v) \left(  \frac{1}{\bv{\zeta}_{+}^2} + \frac{2 \lambda v }{\bv{\zeta}_{+}^3} + \frac{3 \lambda^2 v^2}{\bv{\zeta}_{+}^4}  + O(\lambda^3) \right)
		\\
		&= \frac{1}{\bv{\zeta}_{+}^2 } + \frac{3 (\lambda^2+sq^{-2}) }{\bv{\zeta}_{+}^4}m^{(2)}(\nu) + O(\lambda^3),\label{eq:expansion}
	\end{align}
	referring to the moments of $\bv{\nu}$ calculated in \eqref{eq:moment}. Hence, we have $\bv{\zeta}_{+}^2 = 1  + O(\lambda^2 )$, and putting back into \eqref{eq:expansion}, we have
	\begin{align}
		\bv{\zeta}_{+} = 1 + \frac{3 (\lambda^2 + s q^{-2} ) }{2} m^{(2)} (\nu) + O(\lambda^3).
	\end{align}
	We now have 
	\begin{align}
		\bv{m}_{\fc} (\bv{L}_+)  = \int \frac{\d \nu(v) }{\lambda v - \bv{\zeta}_{+}} &= -\frac{1}{\bv{\zeta}_{+}} - \frac{\lambda^2 + sq^{-2} }{\bv{\zeta}_{+}^3 } m^{(2)}(\nu) + O(\lambda^{3})\\& = -1 + \frac{\lambda^2 + sq^{-2}}{2} m^{(2)}(\nu) + O(\lambda^3) ,
	\end{align}
	and thus 
	\begin{align} \label{eq:range}
		\sigma_N^2 = 1 - (\bv{m}_{\fc}(\bv{L}_+))^2 = (\lambda^2 + sq^{-2}) m^{(2)}(\nu) + O(\lambda^3).
	\end{align}
	Therefore, when $1 \gg\lambda \gg q^{-1}$, we have proved that
	\begin{align}
		\sigma^2 = \lim_{N \to \infty} \lambda^{-2} \sigma_N^2 = m^{(2)}(\nu).
	\end{align}
	Assuming that $\lambda < \varepsilon $ for some sufficiently small constant $\varepsilon>0$, independent of $N$, we can prove that variance of $\caX$ attains a bound similar to \eqref{eq:range}, and prove that $\sigma \sim 1$ in this regime. Finally, for $\lambda \geq \varepsilon$, we have
	\begin{align} 
		1 - (\bv{m}_{\fc} (\bv{L}_+))^2 =  \int \frac{\d \wt{\nu}(v)}{(\lambda v- \bv{\zeta}_+)^2}  -  \left(\int \frac{\d \wt{\nu}(v)}{(\lambda v- \bv{\zeta}_+)} \right)^2 >c >0,
	\end{align}
	for some constant $c >0$, hence $\text{Var} (\caX) \sim N^{-1} \sim N^{-1} \lambda^2$. Thus, we conclude that $\sigma \sim 1$ in any case. See also proof of \cite[Theorem 2.2]{LS14}.
	
	It remains to prove that 
	\begin{align}
		\lim_{N \to \infty} (1 - (\bv{m}_{\fc} (\bv{L}_+))^2  = \lim_{N \to \infty} (1- (m_{\fc} (L_+))^2 , \label{eq:finaleq}
	\end{align}
	for $\lambda \sim 1$. Recall that there exist a constant $\frg >0$, depending only on $\nu$ and uniform on the choices $\lambda \in \Theta_{\varpi}$ and $s \in \Theta_{s^{(4)}}$ such that 
	\begin{align}
		\mathrm{dist} (\{{\zeta}_{\pm}\} , {I}_{{\nu}}^{\lambda}) \geq \frg, \quad \mathrm{dist} (\{\bv{\zeta}_{\pm}\} , {I}_{\bv{\nu}}^{\lambda}) \geq \frg,
	\end{align}
	for $N$ sufficiently large. (See \eqref{eq:g}.) From \eqref{eq:int1} and \eqref{eq:int11}, we have
	\begin{align}
		\int \frac{1}{(\lambda v - \zeta_+)^2}\d \nu(v) = \int \frac{(\lambda v - \bv{\zeta})^2 + sq^{-2}}{\left((\lambda v - \bv{\zeta}_+)^2 - sq^2  \right)^2} \d \nu(v).
	\end{align}
	Thus, we have $\lim_{N \to \infty} |\zeta_+ - \bv{\zeta}_{+}|=0 $. Then it can be easily found that
	\begin{align}
		\lim_{N \to \infty} \left| \int \frac{1}{ \lambda v - \bv{\zeta}_+} \d \bv{\nu} (v) - \int \frac{1}{\lambda v - \zeta_+ } \d \nu(v) \right| = 0,
	\end{align}
	where \eqref{eq:finaleq} follows.
\end{proof}

\end{document}